\documentclass[
a4paper,
11pt,
numbers=auto, 
toc=bib, 
toc=numberline 
]{scrartcl}
\usepackage[utf8]{inputenc}
\usepackage[vmargin={1in},hmargin={0.6in,0.9in},nomarginpar]{geometry}
\usepackage{mathpazo}
\usepackage{microtype}
\usepackage{xpatch}
\usepackage{enumitem}
\usepackage{array}
\setlist{noitemsep}
\usepackage[font=small]{caption}
\usepackage{subcaption}
\usepackage{import}
\usepackage[svgnames]{xcolor}
\usepackage{mathMacros}
\usepackage{tikzit}
\usepackage{pgf}
\usepackage{hyperref}
\usepackage{nameref}
\usepackage[nameinlink]{cleveref}
\newcommand*{\fullref}[1]{\hyperref[{#1}]{\nameref*{#1} (\ref*{#1})}}
\newcommand*{\locallabel}[1]{\csname label\endcsname{locallabel: sec: \thesection subsec: \thesubsection thm: \thetheorem label: #1}}
\newcommand*{\localref}[1]{\hyperref[locallabel: sec: \thesection subsec: \thesubsection thm: \thetheorem label: #1]{(\ref*{locallabel: sec: \thesection subsec: \thesubsection thm: \thetheorem label: #1})}}
\usepackage{theoremStyles}

\newcommand{\nightmode}{0} 
\definecolor{sauron}{RGB}{31, 35, 41}
\definecolor{gandalf}{RGB}{168, 176, 192}
\if1\nightmode{%
		\pagecolor{sauron}
		\AtEndPreamble{\color{gandalf}}
	}
 \fi

\setkomafont{disposition}{\bfseries}
\makeatletter
\xpatchcmd{\@maketitle}{\huge}{\Large}{}{}
\makeatother
\setkomafont{title}{\Large\scshape}
\setkomafont{author}{\normalsize}
\setkomafont{date}{}
\setkomafont{section}{\centering\large\scshape\mdseries}
\setkomafont{sectionentry}{\mdseries}
\setkomafont{subsection}{\normalsize}

\usepackage[backend=biber,style=alphabetic,isbn=false,url=false]{biblatex}
\DeclareFieldFormat[misc,online]{title}{\mkbibquote{#1}}
\newtheorem{maintheorem}{Theorem}

\crefname{maintheorem}{theorem}{theorems}
\Crefname{maintheorem}{Theorem}{Theorems}
\newtheorem{maincorollary}[maintheorem]{Corollary}

\crefname{maincorollary}{corollary}{corollaries}
\Crefname{maincorollary}{Corollary}{Corollaries}

\newtheorem{innerproperty}{}
\newenvironment{property}[1]%
{\renewcommand*{\theinnerproperty}{(#1)}\innerproperty}
{\endinnerproperty}

\newcounter{diagram}
\makeatletter
\ExplSyntaxOn
\NewDocumentEnvironment{diagram*}{o}
{
    $$\IfNoValueTF{#1}{\begin{tikzcd}}{\begin{tikzcd}[#1]}
}{
\end{tikzcd}$$
}
\NewDocumentEnvironment{diagram}{o}
{
\setcounter{diagram}{\value{equation}}
\stepcounter{equation}
\refstepcounter{diagram}
    $$\IfNoValueTF{#1}{\begin{tikzcd}}{\begin{tikzcd}[#1]}
}{
\end{tikzcd}
\eqno\hbox{\normalfont\normalcolor(\thediagram)}$$\@ignoretrue
}
\ExplSyntaxOff
\makeatother
\crefname{diagram}{diagram}{diagrams}
\Crefname{diagram}{Diagram}{Diagrams}
\creflabelformat{diagram}{#2(#1)#3}

\usetikzlibrary{arrows} 
\usetikzlibrary{calc} 
\usetikzlibrary{through} 
\usetikzlibrary{intersections} 
\usetikzlibrary{positioning} 
\usetikzlibrary{perspective} 
\usetikzlibrary{3d} 
\usetikzlibrary{backgrounds} 
\tikzset{%
	Point/.style  = {circle,  inner sep = 0, fill = black, minimum size      = 4pt},
	GPoint/.style = {circle,  inner sep = 0, fill = DimGray, minimum size   = 4pt},
	BPoint/.style = {circle, inner sep  = 0, fill = SteelBlue, minimum size  = 4pt},
	OPoint/.style = {circle, inner sep  = 0, fill = DarkOrange, minimum size = 4pt},
	GLine/.style  = {draw = DimGray, semithick},
	BLine/.style  = {draw = SteelBlue, semithick},
	OLine/.style  = {draw = DarkOrange, semithick},
	GShape/.style = {fill = Gray, fill opacity = 0.3},
	BShape/.style = {fill = LightSkyBlue, fill opacity = 0.3},
	OShape/.style = {fill = Orange,   fill opacity = 0.3},
	Label/.style  = {inner sep = 0, outer sep = 0},
	mapsto/.style = {|->},
	circle through 3 points/.style n args={3}{%
			insert path={let
					\p1=($(#1)!0.5!(#2)$),
					\p2=($(#1)!0.5!(#3)$),
					\p3=($(#1)!0.5!(#2)!1!-90:(#2)$),
					\p4=($(#1)!0.5!(#3)!1!90:(#3)$),
					\p5=(intersection of \p1--\p3 and \p2--\p4)
					in},
			at={(\p5)},
			circle through={(#1)}
		}}

\newcommand{\mathdefault}[1][]{}

\DeclareMathOperator{\Crit}{Crit}
\DeclareMathOperator{\Alpha}{\Acal} 
\DeclareMathOperator{\Rad}{Rad} 
\DeclareMathOperator{\Offset}{\Ocal} 
\DeclareMathOperator{\Del}{Del} 
\DeclareMathOperator{\Cech}{\check{C}} 
\DeclareMathOperator{\DelCech}{\Del\Cech} 
\DeclareMathOperator{\DelVR}{\Del VR} 
\DeclareMathOperator{\VR}{VR} 
\DeclareMathOperator{\Aff}{Aff} 
\DeclareMathOperator{\Lin}{Lin} 
\DeclareMathOperator{\Conv}{Conv} 
\DeclareMathOperator{\On}{On}
\DeclareMathOperator{\Incl}{Incl}
\DeclareMathOperator{\Excl}{Excl}
\DeclareMathOperator{\Front}{Front}
\DeclareMathOperator{\Back}{Back}

\DeclareMathOperator{\dis}{dis} 
\DeclareMathOperator{\codim}{codim}

\makeatletter
\def\blfootnote{\xdef\@thefnmark{}\@footnotetext}
\makeatother

\title{\vspace{-1.5cm}Morse theory for chromatic Delaunay triangulations}
\author{
	Abhinav Natarajan\textsuperscript{1,\textasteriskcentered}, Thomas Chaplin\textsuperscript{1}, Adam Brown\textsuperscript{1}, and Maria-Jose Jimenez\textsuperscript{2,\textdagger}\\
	\smaller\textsuperscript{1}Mathematical Institute, University of Oxford, United Kingdom.\\
	\smaller\textsuperscript{2}Universidad de Sevilla, Spain.
}
\date{September 2025}
\begin{document}
\maketitle
\blfootnote{\textit{2020 Mathematics Subject Classification.} Primary: 55N31; Secondary: 55U10, 52-08.}
\blfootnote{\textit{Keywords and phrases:} Chromatic Delaunay triangulation, chromatic Delaunay--\v{C}ech, chromatic Delaunay--Rips, filtration, topological data analysis, simplicial collapse, discrete Morse theory.}
\blfootnote{\textsuperscript{\textasteriskcentered}Corresponding author: Abhinav Natarajan (\texttt{\href{mailto:abhinav.natarajan@maths.ox.ac.uk}{abhinav.natarajan@maths.ox.ac.uk}}).}
\blfootnote{\textsuperscript{\textdagger}\texttt{\href{mailto:majiro@us.es}{majiro@us.es}}.}
\vspace{-1.2cm}
\begin{abstract}
	\small The chromatic alpha filtration is a generalization of the alpha filtration that can encode spatial relationships among classes of labelled point cloud data, and has applications in topological data analysis of multi-species data.
	In this paper we introduce the \textit{chromatic Delaunay--\v{C}ech} and \textit{chromatic Delaunay--Rips} filtrations, which are computationally favourable alternatives to the chromatic alpha filtration.
	We use generalized discrete Morse theory to show that the \v{C}ech, chromatic Delaunay--\v{C}ech, and chromatic alpha filtrations are related by simplicial collapses.
	Our result generalizes a result of Bauer and Edelsbrunner from the non-chromatic to the chromatic setting.
	We also show that the chromatic Delaunay--Rips filtration is locally stable to perturbations of the underlying point cloud.
	Our results provide theoretical justification for the use of chromatic Delaunay--\v{C}ech and chromatic Delaunay--Rips filtrations in applications, and we demonstrate their computational advantage with numerical experiments.
\end{abstract}
\renewcommand{\baselinestretch}{0.5}
{\footnotesize \tableofcontents} 
\renewcommand{\baselinestretch}{1.1}
\section{Introduction}\label{sec: introduction}
\subsection{Motivation}
The field of topological data analysis (TDA) is founded on the idea of interpreting datasets as geometric objects that can be studied with tools from algebraic topology.
When the dataset is a point cloud $X \subset \RR^d$, we build a filtration of topological spaces, typically a filtered simplicial complex, to serve as a multiscale approximation of the hypothetical space from which the data are assumed to be sampled.
We can then describe the geometry of the dataset through topological invariants of the filtration, such as \textit{persistent homology}~\cite{edelsbrunner_topological_2000,zomorodian_computing_2005}, that are stable with respect to perturbations of the data.

More recent data collection techniques in fields like cancer biology, geospatial analysis, and ecology produce \textit{chromatic} (labelled) data with interactions between the points of each colour.
In these cases one is interested not just in the spatial structure of the data as a whole, but also in the spatial relationships among subsets of the data produced by the colouring; see \Cref{fig: example relational}.

\begin{figure}[h]
	\centering
	\resizebox{\width}{!}{\input{figures/intro_examples/points.pgf}}
	\caption{
		The point cloud above does not have interesting topological features \textit{a priori}, but when considering the colouring, the blue points enclose a void that is filled in by the orange points.
		This spatial relationship can be captured by the chromatic alpha filtration.
	}
	\label{fig: example relational}
\end{figure}

One way to formalize this is to study the inclusion maps among such subsets of points.
Mathematically speaking, a \textit{chromatic set} is given by a set of points $X \subset \RR^d$ and a surjective function $\mu: X \to \{0, \ldots, s\}$ for some $s \in \NN$.
Given sets of colours $I \subset J \subset \{0, \ldots, s\}$, the spatial relationship between the set of points with colours in $I$ and $J$ respectively is captured by the spatial properties of the inclusion map $\mu^{-1}(I) \hookrightarrow \mu^{-1}(J)$.
Here too $X$ is represented by a filtered simplicial complex.
If the subsets in the partition of $X$ produced by the colouring correspond to filtered subcomplexes of the original filtration---in other words, the construction of the filtration is functorial with respect to the input dataset---then we can study the induced inclusion maps among these filtered subcomplexes.

Common choices for building a filtration in TDA are the \textit{\v{C}ech filtration} $\Cech_{\bullet}(X)$, the \textit{Vietoris--Rips filtration} $\VR_{\bullet}(X)$, and the \textit{alpha filtration} $\Alpha_{\bullet}(X)$.
The \v{C}ech and Vietoris--Rips filtrations have the necessary functoriality property---indeed, $\Cech_{\bullet}(\mu^{-1}(I))$ and $\VR_{\bullet}(\mu^{-1}(I))$ are sub-filtrations of $\Cech_{\bullet}(\mu^{-1}(J))$ and $\VR_{\bullet}(\mu^{-1}(J))$ respectively.
Unfortunately, both of these constructions are filtrations of the complete simplicial complex on the vertex set $X$, in which the number of simplices scales exponentially with the number of data points $n$.
This makes their computation impractical for large datasets.
On the other hand, if $X$ is in general position, then the alpha filtration, which was first defined in~\cite{edelsbrunner_union_1995}, is a filtration of the Delaunay triangulation $\Del(X)$ associated to $X$~\cite{delaunay_1934}.
The number of simplices in $\Del(X)$ is bounded by a polynomial in $n$ of degree $\left\lfloor \frac{d}{2} \right\rfloor$~\cite{mcmullen_maximum_1970,seidel_upper_1995}, and when the points of $X$ are sampled independently and uniformly from the unit ball, then the expected number of simplices of $\Del(X)$ is linear in $n$~\cite{dwyer_higher-dimensional_1991}.
This makes it feasible to compute the alpha filtration when the ambient space is low-dimensional.
Moreover, $\Alpha_r(X)$ is homotopy equivalent to $\Cech_r(X)$ for any $r \geq 0$, so the alpha filtration captures the same topological information as the \v{C}ech filtration.
Unfortunately, alpha filtrations and the underlying Delaunay triangulations are not functorial with respect to inclusions of point clouds; that is, $\Del(\mu^{-1}(I))$ is not necessarily a sub-complex of $\Del(\mu^{-1}(J))$ (see \Cref{fig:del_non_functoriality}).
This deficiency is addressed by the \textit{chromatic Delaunay triangulation} $\Del(X,\mu)$, introduced by Biswas et al.~\cite{biswas2022size}, and the corresponding \textit{chromatic alpha filtration} $\Alpha_{\bullet}(X,\mu)$, introduced by~\cite{Montesano2025chromatic}.
The chromatic alpha filtration has the following properties:
\begin{enumerate}
	\item $\Alpha_r(X,\mu)$ is a deformation retract of $\Cech_r(X)$ for any colouring $\mu$ of the data.
	\item Given sets of colours $I \subset J \subset \{0, \ldots, s\}$, the chromatic alpha filtration associated to the subset of points with colours in $I$ is a filtered subcomplex of the chromatic alpha filtration associated to the subset of points with colours in $J$.
	\item For a random colouring, the expected number of simplices in $\Alpha_{\infty}(X,\mu)$ is $O(n^{\ceiling{\frac{d}{2}}})$~\cite{biswas2022size}, which is far fewer than the number of simplices in the \v{C}ech filtration.
\end{enumerate}

Thus, the chromatic alpha filtration interpolates between the functoriality of the \v{C}ech filtration and the sparsity of the alpha filtration, while preserving the same spatial information.
In fact, the \v{C}ech and alpha filtrations are extremal cases of chromatic alpha filtrations (see \Cref{sec: chromatic alpha filtration} for the precise relationship).

\begin{figure}
	\centering
	\begin{subfigure}{0.48\linewidth}
	\centering
	\begin{tikzpicture}[scale = 3]
		\coordinate (a) at (0.6, 0.8);
		\coordinate (b) at (0.4, 0.15);
		\coordinate (c) at (0.75, -0.05);
		\coordinate (d) at (0.95, 0.15);

		\draw[BShape, BLine] (a) -- (b) -- (c) -- (d) -- cycle;
		\draw[BLine] (b) -- (d);

		\node[BPoint] at (a) {} node[Label, above = 3pt of a] {$a$};
		\node[BPoint] at (b) {} node[Label, left  = 3pt of b] {$b$};
		\node[BPoint] at (c) {} node[Label, below = 3pt of c] {$c$};
		\node[BPoint] at (d) {} node[Label, right = 3pt of d] {$d$};
	\end{tikzpicture}
	\caption{$\Del(X)$}
\end{subfigure}
\hfill
\begin{subfigure}{0.48\linewidth}
	\centering
	\begin{tikzpicture}[scale = 3]
		\coordinate (a) at (0.6, 0.8);
		\coordinate (b) at (0.4, 0.15);
		\coordinate (c) at (0.75, -0.05);
		\coordinate (d) at (0.95, 0.15);
		\coordinate (e) at (0.6, 0.45);

		\draw[BShape, BLine] (a) -- (b) -- (c) -- (d) -- cycle;
		\draw[BLine] (a) -- (e);
		\draw[BLine] (b) -- (e);
		\draw[BLine] (c) -- (e);
		\draw[BLine] (d) -- (e);

		\node[BPoint] at (a) {} node[Label, above       = 3pt of a] {$a$};
		\node[BPoint] at (b) {} node[Label, left        = 3pt of b] {$b$};
		\node[BPoint] at (c) {} node[Label, below       = 3pt of c] {$c$};
		\node[BPoint] at (d) {} node[Label, right       = 3pt of d] {$d$};
		\node[BPoint] at (e) {} node[Label, above right = 3pt of e] {$e$};
	\end{tikzpicture}
	\caption{$\Del(X \cup \{e\})$}
\end{subfigure}
	\caption{
		In the above, $a=(0.6, 0.8)$, $b = (0.4, 0.15)$, $c = (0.75, -0.05)$, $d = (0.95, 0.15)$, and $e = (0.6, 0.45)$.
		Denoting $X = \{a, b, c, d\}$, the inclusion $X \hookrightarrow X \cup \{e\}$ does not induce a simplicial map $\Del(X) \hookrightarrow \Del(X \cup \{e\})$ since $\Del(X)$ contains the 1-simplex $bd$ while $\Del(X \cup \{e\})$ does not.
	}
	\label{fig:del_non_functoriality}
\end{figure}

\subsection{Contributions}
The properties stated above make the chromatic alpha filtration an attractive choice for quantifying spatial relations in low-dimensional data.
However, the sparsity of the filtration is partially offset by the computational complexity of producing the filtration values themselves.
To simplify computations, we introduce the \textit{chromatic Delaunay--\v{C}ech} filtration $\DelCech_{\bullet}(X,\mu)$ and the \textit{chromatic Delaunay--Rips} filtration $\DelVR_{\bullet}(X,\mu)$.
Under a general position assumption, these are filtrations of the chromatic Delaunay triangulation where the filtration value of each simplex is given by its filtration value in the \v{C}ech and Vietoris--Rips filtration respectively (see \Cref{sec: chromatic DelCech and chromatic DelRips filtrations} for the precise definitions).
Previous work has already shown the effectiveness of the monochromatic Delaunay--Rips filtration in machine learning applications~\cite{mishra_stability_2023}.
Moreover, the chromatic Delaunay--Rips filtration is substantially faster to compute than the chromatic alpha filtration (see \Cref{sec: consequences for practical computations}).
Our goal in this paper is to provide theoretical justification for the use of our constructions.

If $\mu$ and $\nu$ are colourings of $X$, we say that $\nu$ is a \textit{refinement} of $\mu$, and we write $\nu \preceq \mu$, if the partition of $X$ induced by $\nu$ is a refinement of the partition of $X$ induced by $\mu$.
In this case, for fixed $r \geq 0$ the various filtrations mentioned above fit into the following commutative diagram, where $\DelCech_{\bullet}(X)$ in the bottom row denotes the monochromatic Delaunay--\v{C}ech filtration (cf. \Cref{dgm: functoriality of chromatic DelCech filtration} in \Cref{sec: background}).
\begin{diagram}\label{dgm: diagram in introduction}
	\Alpha_{r}(X,\nu) \arrow[r, hook]  & \DelCech_{r}(X,\nu) \arrow[r, hook] & \Cech_{r}(X)\\
	\Alpha_{r}(X,\mu) \arrow[r, hook] \arrow[u, hook] & \DelCech_{r}(X,\mu) \arrow[r, hook] \arrow[u, hook] & \Cech_{r}(X) \arrow[u, equal]\\
	\Alpha_{r}(X) \arrow[r, hook] \arrow[u, hook] & \DelCech_{r}(X) \arrow[r, hook] \arrow[u, hook] & \Cech_{r}(X) \arrow[u, equal]
\end{diagram}
Using a generalization of discrete Morse theory, Bauer and Edelsbrunner~\cite{bauer_morse_2016} showed that the complexes in the bottom row are related by simplicial collapses---in particular, the inclusions are homotopy equivalences and induce isomorphisms on homology.
Montesano et al.~\cite{Montesano2025chromatic} proved that the inclusions in the left column are homotopy equivalences.
We use generalized discrete Morse theory to prove that there are simplicial collapses in the middle column and in each row:
\begin{maintheorem}\label{thm: theorem A}
	Let $X \subset \RR^d$ be a finite set of points in general position, and let $\mu, \nu$ be colourings of $X$ such that $\nu$ refines $\mu$.
	For any $r \geq 0$ there is a simplicial collapse $\DelCech_r(X,\nu) \searrow \DelCech_r(X,\mu)$.
\end{maintheorem}
\begin{maintheorem}\label{thm: theorem B}
	Let $X \subset \RR^d$ be a finite set of points in general position and let $\mu$ be a colouring of $X$.
	Then for any $r \geq 0$ there are simplicial collapses $\Cech_r(X) \searrow \DelCech_r(X,\mu) \searrow \Alpha_r(X,\mu)$.
\end{maintheorem}
\Cref{thm: theorem B} generalizes the main result of~\cite{bauer_morse_2016} to the chromatic case, thereby answering a question posed in~\cite{Montesano2025chromatic}.

Our proof of \Cref{thm: theorem A} relies on the construction of a novel generalized discrete Morse function on the chromatic Delaunay triangulation that is structurally compatible with the \v{C}ech filtration.
The details of this construction are given in \Cref{sec: collapse of chromatic triangulation}.
On the other hand, our proof of \Cref{thm: theorem B} is an adaptation of the proof in~\cite{bauer_morse_2016} to the chromatic setting.
To enable this adaptation, in \Cref{sec: KKT for stacks and pairing lemmas} we use tools from convex optimization to study the discrete Morse-theoretic structure of a family of filtrations which includes the \v{C}ech, chromatic alpha, and chromatic Delaunay--\v{C}ech filtrations.
We prove that for each of these filtrations, its filtration function (which assigns to each simplex its filtration value) is a generalized discrete Morse function, and we relate the combinatorial structure of the corresponding Morse gradient to the geometry of the underlying point cloud.
The Morse theoretic structure of filtration functions has been described before for the \v{C}ech filtration~\cite{attali_vietorisrips_2013} and alpha filtration~\cite{edelsbrunner_surface_2003,van_manen_power_2014}, and these results were unified in~\cite{bauer_morse_2016}.
More recently, it was shown in~\cite{Montesano2025chromatic} that the chromatic alpha filtration function is a generalized discrete Morse function.
Our work simultaneously extends each of these results to a larger class of filtrations.
Finally, in \Cref{sec: proof of main theorems} we prove \Cref{thm: theorem A,thm: theorem B} using the results from \Cref{sec: collapse of chromatic triangulation,sec: KKT for stacks and pairing lemmas}.

Simplicial expansions and collapses are examples of \textit{simple homotopy equivalence}, which is a combinatorial strengthening of the notion of homotopy equivalence for CW-complexes, and which is algebraically characterized by the vanishing of Whitehead torsion~\cite{cohen_course_1973}.
In light of this, \Cref{thm: theorem A,thm: theorem B} immediately imply the following result.
\begin{maincorollary}
	For any $r \geq 0$, each inclusion in \Cref{dgm: diagram in introduction} is a simple homotopy equivalence, and therefore each simplicial complex in that diagram has the same simple homotopy type.
\end{maincorollary}
Our result does not imply that there are simplicial collapses along the first column in \Cref{dgm: diagram in introduction}; indeed, our methods cannot be extended for this case, as we explain in \Cref{rem: counterexample to chromatic alpha collapse}.

From a computational perspective, the main consequences of our results are:
\begin{enumerate}
	\item The inclusion of the chromatic Delaunay--\v{C}ech filtration into the \v{C}ech filtration is a homotopy equivalence at each filtration level, which makes the chromatic Delaunay--\v{C}ech filtration a drop-in alternative to the \v{C}ech and chromatic alpha filtrations for the purpose of quantifying spatial relationships through persistent homology.
	\item The chromatic Delaunay--Rips filtration, which is related to the chromatic Delaunay--\v{C}ech filtration via nested inclusions (see \Cref{sec: consequences for practical computations}), serves to approximate the \v{C}ech filtration for persistent homology computations.
\end{enumerate}

To further justify the use of the chromatic Delaunay--Rips filtration in practical computations, in \Cref{sec: consequences for practical computations} we show that the chromatic Delaunay--Rips filtration associated to a point cloud is locally stable to perturbations of the underlying point cloud.
In fact, we show that this local stability extends to maps between chromatic Delaunay--Rips filtrations that are induced by inclusions of the underlying point clouds (\Cref{cor:local_stability_of_delvr}).
This is a generalization of the stability result in~\cite{mishra_stability_2023}, in which the authors prove that the monochromatic Delaunay--Rips filtration is locally stable to perturbations of the underlying point cloud.
In \Cref{sec: consequences for practical computations} we also document the empirical runtime performance of the chromatic Delaunay--\v{C}ech and chromatic Delaunay--Rips filtrations, vis-\`a-vis the chromatic alpha filtration.
The numerical evidence suggests that the chromatic Delaunay--Rips filtration is significantly faster to compute than the chromatic alpha filtration, and is therefore a viable alternative for analysing spatial relationships in low-dimensional point clouds.

\subsection{Related Work}

The \v{C}ech filtration originated in the work of Alexandroff and \v{C}ech in algebraic topology, and gained importance in TDA due to the work of E. Carlsson, G. Carlsson, and de Silva~\cite{carlsson_algebraic_2006,silva_topological_2004}.
Its foundational role in TDA is justified by the fact that the \v{C}ech complex $\Cech_r(X)$ has the same homotopy type as the union of closed balls of radius $r$ centred on the points in $X$ (see \Cref{sec: offset and Cech filtrations}).
Vietoris--Rips complexes originated in the work of Vietoris~\cite{vietoris_uber_1927}, and were popularized in the work of Rips and Gromov~\cite{gromov_hyperbolic_1987}.
While the Vietoris--Rips filtration does not have the same geometric interpretation as the \v{C}ech filtration, it still contains meaningful information about the spatial structure of $X$~\cite{hausmann_vietoris-rips_1996,latschev_vietoris-rips_2001,chambers_vietorisrips_2010,attali_vietorisrips_2013}.
Moreover, it is much faster to compute the filtration values of the Vietoris--Rips filtration, and this advantage carries over to the chromatic Delaunay--Rips filtration.
The chromatic alpha filtration was first defined for the special case of a bi-colouring $\mu:X \to \{0, 1\}$ by Reani and Bobrowski~\cite{reani_coupled_2023}, under the name \textit{coupled alpha filtration}.

The chromatic alpha filtration is also related to the Delaunay bi-filtration~\cite{alonso_delaunay_2023}, which is defined for a function on a point cloud, in the case when the class labels are hierarchically nested.
We leave the consideration of this special case for future work.
Aside from chromatic alpha filtrations, other approaches for studying spatial relations have been previously considered.
Bauer and Edelsbrunner~\cite{bauer_morse_2016} generalize alpha filtrations to ``selective Delaunay filtrations''.
Given point clouds $X \subset Y$, they construct filtrations $\Del_{\bullet}(Y, Y) \subset \Del_{\bullet}(X, Y)$ that are pointwise homotopy equivalent to $\Cech_{\bullet}(Y)$ and $\Cech_{\bullet}(X)$ respectively (see \Cref{sec: convex optimization for stacks} for more details).
We will make use of their construction in our proofs, but we remark that the chromatic filtrations in this paper are functorial over a larger class of inclusions (for example, see \Cref{rem: functoriality of chromatic alpha filtration} for a precise description of the functoriality of the chromatic alpha filtration).
Stolz et al.~\cite{stolz2023relational} analyse spatial relations using Dowker filtrations~\cite{dowker_1952_homology,chowdhury2018functorial} and multi-species witness filtrations.
The multi-species witness filtration is based upon the witness complex~\cite{desilva2004topological}, which can often be seen as a Delaunay triangulation with respect to an intrinsic metric on the data points~\cite{de_silva_weak_2008}.

Discrete Morse theory, the analogue of Morse theory for finite CW-complexes, was originally formulated by Forman~\cite{forman_morse_1998} in terms of real-valued functions on the set of cells of the complex.
Chari~\cite{chari_discrete_2000} showed that the theory could equivalently be described in terms of pairings in the face poset of the complex.
In our work we use a generalization of discrete Morse theory, proposed by~\cite{freij_equivariant_2009}, that considers arbitrary intervals in the face poset instead of pairings.
Similar ideas have been considered earlier in the form of the ``Cluster Lemma''~\cite{hersh2005optimizing,jonsson_topology_2003}.
Separately, the idea of collapsing simplicial complexes along intervals in the face poset has been considered before~\cite{wegner_d-collapsing_1975,kozlov_collapsibility_2000}.

\section{Preliminaries}\label{sec: background}

In this section we will briefly lay out the relevant background material for the rest of the paper.
In \Crefrange{sec: offset and Cech filtrations}{sec: chromatic DelCech and chromatic DelRips filtrations} we recall the definition and properties of the offset, \v{C}ech, alpha, and chromatic alpha filtrations, and the Delaunay and chromatic Delaunay triangulations.
We also define the chromatic Delaunay--\v{C}ech and chromatic Delaunay--Rips filtrations.
In \Cref{sec: geometric simplicial complexes} we will define some general notions relating to geometric realizations of simplicial complexes.
In \Cref{sec: discrete Morse theory} we recall results from a generalization of discrete Morse theory allowing for intervals larger than pairs.
Finally, in \Cref{sec: convex optimisation} will state the Karush-Kuhn-Tucker conditions from the theory of convex optimization, which we will use to study the filtration function of the various filtrations we work with.

\subsection{Notation}\label{sec: notation}

$X$ will always denote a finite point cloud in $\RR^d$.
All point clouds and simplicial complexes are assumed to be finite unless otherwise stated.
If $s$ is a positive integer then $[s]$ will denote the set $\{0, \ldots, s-1\}$.
For a point cloud $X \subset \RR^d$ we define a \textit{colouring} of $X$ to be a surjective function $\mu : X \to \{0, \ldots, s\}$, and a \textit{chromatic set} is a point cloud with a colouring.
There is an induced partition of $X$ into disjoint non-empty subsets $X_m = \mu^{-1}(m)$ indexed by the colours, and by abuse of notation we will write $\mu = (X_0, \ldots, X_s)$.
We can equivalently define a colouring of $X$ to be an indexed partition of $X$ into disjoint non-empty subsets.
Now let $X \subset Y \subset \RR^d$, $s \leq s'$, and let $\mu = (X_0, \ldots, X_s)$ and $\nu = (Y_0, \ldots, Y_{s'})$ be colourings of $X$ and $Y$ respectively.
We say that $\nu$ is a \textit{refinement} of $\mu$, and we write $\nu \preceq \mu$, if for each $i \in \{0, \ldots, s\}$ there exists $J \subset \{0, \ldots, s'\}$ such that $X_i = \bigcup_{j \in J} Y_j$.

If $K$ is an (abstract or geometric) simplicial complex and $\sigma$ is an $n$-simplex in $K$ we write $\sigma^{(n)} \in K$.
If $\tau$ is a face (resp. strict face) of $\sigma$ we write $\tau \leq \sigma$ (resp. $\tau < \sigma$).
A \textit{filtration} $\{K_t\}_{t \in [0, \infty]}$ is a family of topological spaces or simplicial complexes, such that $K_s$ is a subspace (resp. subcomplex) of $K_t$ whenever $0 \leq s \leq t \leq \infty$.
The \textit{underlying space} (resp. \textit{underlying simplicial complex}) of the filtration is $K_{\infty}$, and we say that \textit{$K_{\bullet}$ is a filtration of $K_{\infty}$}.
When $K_{\bullet}$ is a filtered simplicial complex, the \textit{filtration function} of $K_{\bullet}$ is the function $K_{\infty} \to \RR$ defined by $\sigma \mapsto \inf \{t \in \RR \mid \sigma \in K_t \}$, and the \textit{filtration value} of a simplex $\sigma \in K_{\infty}$ is the value of this function at $\sigma$.
If $L_{\bullet}$ is a filtration with $L_t \subset K_t$ for each $t \in [0, \infty]$, we say that $L_{\bullet}$ is a \textit{subfiltration} of $K_{\bullet}$.

Let $\cat{Top}$ denote the category of topological spaces and continuous maps, and let $\cat{SimpComp}$ be the category of simplicial complexes and simplicial maps.
Then a filtration is a functor $[0, \infty] \to \cat{Top}$ or $[0, \infty] \to \cat{SimpComp}$.
If $\Ccal$ and $\Dcal$ are categories then we use $[\Ccal, \Dcal]$ to denote the category of functors from $\Ccal$ to $\Dcal$.
If $c$ is an object of $\Ccal$, then $\Ccal \downarrow c$ will denote the \textit{slice category}; that is, the category whose objects are morphisms $d \to c$ in $\Ccal$ and whose morphisms are commuting triangles:
\begin{diagram*}
	d \ar[rr] \ar[rd] & & e \ar[ld]\\ & c &
\end{diagram*}

\subsection{Offset and \v{C}ech Filtrations}\label{sec: offset and Cech filtrations}
For $x \in X$ and $r \geq 0$ let $D_r(x) = \{y \in \RR^d \mid |x - y| \leq r\}$ denote the closed $d$-dimensional ball of radius $r$ centred on $x$.
The \textit{offset filtration} $\Ocal_{\bullet}(X)$ is a filtration of $\RR^d$ defined by $\Ocal_r(X) = \cup_{x \in X} D_r(x)$.
The \textit{\v{C}ech filtration} $\Cech_{\bullet}(X)$ is a filtration of the complete simplicial complex on the vertex set $X$, and $\Cech_r(X)$ is defined as the nerve of the cover $\{D_r(x)\}_{x \in X}$ of $\Offset_r(X)$:
\begin{align}
	\Cech_r(X) & \defeq \left\{\sigma \subset X \mid \bigcap_{x \in \sigma} D_r(x) \neq \emptyset\right\}
	\label{eq: definition of Cech filtration}                                                              \\
	           & =  \{\sigma \subset X \mid \text{$\sigma$ is contained in a closed ball of radius $r$}\}.
	\label{eq: alternative definition of Cech filtration}
\end{align}

Let $\cat{Point}_d$ be the category of finite point sets in $\RR^d$ and their inclusions.
If $X \subset Y \subset \RR^d$ and $0 \leq s \leq t \leq \infty$ then $\Offset_s(X) \subset \Offset_t(Y)$, so the offset filtration defines a functor $\cat{Point}_d \times [0, \infty] \to \cat{Top}$.
Similarly, an inclusion of point sets induces an inclusion of the respective \v{C}ech filtrations, so the \v{C}ech filtration defines a functor from $\cat{Point}_d \times [0, \infty]$ to the category of simplicial complexes $\cat{SimpComp}$.
We let $|\cdot| : \cat{SimpComp} \to \cat{Top}$ denote the geometric realization functor.
Then there is a functor $\cat{Point}_d \times [0, \infty] \to \cat{Top}$ defined by $(X, t) \mapsto |\Cech_t(X)|$.

We say that two functors $F, G : [0, \infty] \to \cat{Top}$ have the same \textit{persistent homotopy type} if there is a natural transformation $F \Rightarrow G$ whose components are homotopy equivalences.
By a functorial version of the Nerve lemma (Theorem B in~\cite{bauer_unified_2023}), there is a natural transformation from the geometric realization of the \v{C}ech filtration to the offset filtration that is a pointwise homotopy-equivalence.
More explicitly there are homotopy equivalences $|\Cech_t(X)| \to \Offset_t(X)$ for each $(X, t) \in \cat{Point}_d \times [0, \infty]$, and if $0 \leq s \leq t \leq \infty$ and $X \subset Y$ then these maps fit into the following commutative diagram.
\begin{diagram}[column sep=0.5ex, row sep=2ex]\label{dgm: functoriality of the Cech filtration}
	& {|\Cech_s(Y)|} && {|\Cech_t(Y)|} \\
	{|\Cech_s(X)|} && {|\Cech_t(X)|} \\
	& {\Offset_s(Y)} && {\Offset_t(Y)} \\
	{\Offset_s(X)} && {\Offset_t(X)}
	\arrow[hook, from=1-2, to=1-4]
	\arrow[hook, from=2-1, to=1-2]
	\arrow[hook, from=2-3, to=1-4]
	\arrow[hook, from=4-1, to=3-2]
	\arrow[hook, from=4-1, to=4-3]
	\arrow[hook, from=3-2, to=3-4]
	\arrow[hook, from=4-3, to=3-4]
	\arrow["\simeq", swap, from=2-1, to=4-1]
	\arrow["\simeq"{pos=0.3}, from=1-2, to=3-2]
	\arrow[hook, from=2-1, to=2-3, crossing over]
	\arrow["\simeq"{pos=0.3}, from=2-3, to=4-3, crossing over]
	\arrow["\simeq", from=1-4, to=3-4]
\end{diagram}
In particular the \v{C}ech filtration and the offset filtration have the same persistent homotopy type (and hence persistent homology) for any fixed $X \in \cat{Point}_d$, and any homotopy-invariant of the inclusion $\Offset_{\bullet}(X) \hookrightarrow \Offset_{\bullet}(Y)$ (such as the induced map on persistent homology) is isomorphic to the corresponding invariant of the inclusion $|\Cech_{\bullet}(X)| \hookrightarrow |\Cech_{\bullet}(Y)|$.
These facts justify the foundational role of the \v{C}ech filtration in TDA.

\subsection{Alpha Filtration}\label{sec: alpha filtration}

For $x \in X$ the \textit{Voronoi domain} of $x$, and the \textit{Voronoi ball} of $x$ with radius $r \geq 0$ are defined as
\begin{align}
	\dom(x, X) & \defeq \{ y \in \RR^d \mid \norm{y-x} \leq \norm{y - x'} \text{ for all } x' \in X\},
	\label{eq: definition of Voronoi domains}                                                          \\
	D'_r(x, X) & \defeq D_r(x) \cap \dom(x, X).
	\label{eq: definition of Voronoi balls}
\end{align}
The \textit{alpha filtration} $\Alpha_{\bullet}(X)$ is the nerve of the collection $\{D_{\bullet}'(x, X)\}_{x \in X}$:
\begin{equation}
	\Alpha_r(X) \defeq \left\{\sigma \subset X \mid \bigcap_{x \in \sigma} D'_r(x, X) \neq \emptyset\right\}.
	\label{eq: definition of alpha filtration}
\end{equation}
By construction the underlying simplicial complex of the alpha filtration has vertex set $X$.

We say that a sphere in $\RR^d$ is \textit{empty} of $E \subset X$ if it does not enclose any points of $E$.
Then the alpha filtration can be equivalently defined by
\begin{equation}
	\Alpha_r(X) = \{ \sigma \subset X \mid \text{$\sigma$ has a circumsphere of radius $r$ empty of $X$}\}.
	\label{eq: alternative definition of alpha filtration}
\end{equation}
For fixed $X$ the alpha filtration defines a functor $[0, \infty] \to \cat{SimpComp}$.
It is straightforward to check that the union of Voronoi balls equals the union of closed balls: $\cup_{x \in X} D'_r(x, X) = \cup_{x \in X} D_r(x)$.
Then the functorial nerve lemma (Theorem B in~\cite{bauer_unified_2023}) implies that the alpha filtration has the same persistent homotopy type as the offset filtration.

The Voronoi domains comprise a polyhedral decomposition of $\RR^d$ called the \textit{Voronoi tessellation} associated to $X$.
The geometric dual to the Voronoi tessellation is called the \textit{Delaunay mosaic} of $X$ and denoted by $\Del(X)$, and is a polyhedral decomposition of the convex hull of $X$.
If the Delaunay mosaic is simplicial then we call it the \textit{Delaunay triangulation} of $X$, and in this case it is the unique geometric realization of $\Alpha_{\infty}(X)$ whose vertices coincide with $X$.
Using \Cref{eq: alternative definition of alpha filtration}, the Delaunay mosaic is guaranteed to be simplicial if $X$ satisfies the following general position condition.
\begin{property}{GP1}\label{prop: GP1}
	We say $X \subset \RR^d$ satisfies \Cref{prop: GP1} if for all $0 \leq k < d$, no $k+3$ points of $X$ lie on a common $k$-sphere.
\end{property}
The condition \Cref{prop: GP1} is mild; for example, it is satisfied almost surely when $X$ is the realization of a stationary Poisson point process (see Lemma 2.1 in~\cite{edelsbrunner_expected_2017}).

\subsection{Chromatic Alpha Filtration}\label{sec: chromatic alpha filtration}

Now we describe chromatic alpha filtrations, which generalize alpha filtrations to chromatic sets.
We follow the terminology of~\cite{Montesano2025chromatic} in this section.

Let $\mu = (X_0, \ldots, X_s)$ be a colouring of $X\subset \RR^d$.
For each $m \in \{0, \ldots, s\}$ we have a collection of open sets $\{D'_r(x, X_m)\}_{x \in X_m}$, that is, the collection of Voronoi balls of radius $r$ with respect to the set of $m$-coloured points.
The chromatic alpha filtration is defined as the nerve of the union of this collection $\cup_{m=0}^s \{D'_r(x, X_m)\}_{x \in X_m}$ (see \Cref{fig: chromatic voronoi tesselation blue,fig: chromatic voronoi tesselation orange,fig: chromatic voronoi tesselation overlay}):
\begin{equation}
	\Alpha_r(X,\mu) \defeq \left\{\sigma \subset X \mid \bigcap_{x \in \sigma}D'_r(x, X_{\mu(x)}) \neq \emptyset\right\}.
	\label{eq: definition of chromatic alpha filtration}
\end{equation}
\begin{figure}
	\centering
	\begin{subfigure}[t]{0.24\linewidth}
		\centering
		\resizebox{\linewidth}{!}{\input{figures/voronoi_overlay/monochromatic.pgf}}
		\caption{Voronoi tessellation and Voronoi balls for $X$.}
		\label{fig: monochromatic voronoi tessellation}
	\end{subfigure}
	\hfill
	\begin{subfigure}[t]{0.24\linewidth}
		\centering
		\resizebox{\linewidth}{!}{\input{figures/voronoi_overlay/bichromatic_blue_only.pgf}}
		\caption{Voronoi tessellation and Voronoi balls for $X_0$.}
		\label{fig: chromatic voronoi tesselation blue}
	\end{subfigure}
	\hfill
	\begin{subfigure}[t]{0.24\linewidth}
		\centering
		\resizebox{\linewidth}{!}{\input{figures/voronoi_overlay/bichromatic_orange_only.pgf}}
		\caption{Voronoi tessellation and balls for $X_1$.}
		\label{fig: chromatic voronoi tesselation orange}
	\end{subfigure}
	\hfill
	\begin{subfigure}[t]{0.24\linewidth}
		\centering
		\resizebox{\linewidth}{!}{\input{figures/voronoi_overlay/bichromatic.pgf}}
		\caption{Overlay of Voronoi tessellations and Voronoi balls for $X_0$ and $X_1$.}
		\label{fig: chromatic voronoi tesselation overlay}
	\end{subfigure}
	\caption{A bi-chromatic set $X = X_0 \cup X_1$ in $\RR^2$, and the associated Voronoi tessellations and collections of Voronoi balls.}
\end{figure}
By construction, the underlying simplicial complex of the chromatic alpha filtration has vertex set $X$.
The filtration is invariant to a permutation of the labels $\{0, \ldots, s\}$ and only depends on the set $\{X_0, \ldots, X_s\}$.
As in the case of the alpha filtration, the functorial nerve lemma shows that the chromatic alpha filtration has the same persistent homotopy type as the offset filtration.

The alpha filtration and \v{C}ech filtrations are extremal cases of the chromatic alpha filtration.
When $\mu$ gives each point the same colour (that is, $X$ is monochromatic) then $X_{\mu(x)} = X$ for all $x \in X$, and in this case the chromatic alpha filtration reduces to the alpha filtration.
When $\mu$ gives each point a distinct colour (the \textit{maximal colouring}) then $\dom(x, X_{\mu(x)}) = \dom(x, \{x\}) = \RR^d$, and hence the Voronoi balls are just Euclidean balls: $D_r'(x, X_{\mu(x)}) = D_r(x)$.
In this case the chromatic alpha filtration is isomorphic to the \v{C}ech filtration.

The two most important properties of the chromatic alpha filtration are its sparsity and its functoriality.
For example, if $X$ has $n$ points and the colouring is chosen uniformly randomly, then the expected number of simplices of the chromatic alpha filtration is $O\left(n^{\ceiling*{\frac{d}{2}}}\right)$; we refer the reader to~\cite{biswas2022size} for more details.
We focus here on the functoriality.

From the definition we see that the chromatic alpha filtration is functorial with respect to inclusions of points in the following sense.
Suppose $I \subset \{0, \ldots, s\}$ is some subset of colours, $X_I \defeq \mu^{-1}(I)$ is the set of $I$-coloured points, and $\mu_I$ is the restriction of $\mu$ to $X_I$.
Then for each $r \geq 0$, $\Alpha_r(X_I,\mu_I)$  is the full subcomplex of $\Alpha_r(X,\mu )$ spanned by the vertices in $X_I$.

Slightly less obvious is the fact that the chromatic alpha filtration is functorial with respect to refinement of class labels.
That is, if $\mu = (X_0, \ldots, X_s)$ and $\nu = (X_0', \ldots, X_{s'}')$ are colourings of $X$ such that $\nu \preceq \mu$, then $\Alpha_{\bullet}(X,\mu)$ is a sub-filtration of $\Alpha_{\bullet}(X,\nu)$.
To see this, observe that for any $x \in X$, we have $X_{\mu(x)} \supset X'_{\nu(x)}$ and hence $\dom(x, X_{\mu(x)}) \subset \dom(x, X'_{\nu(x)})$.
Then if $\sigma \in \Alpha_r(X,\mu)$ we have by definition
\[
	\emptyset \neq \bigcap_{x \in \sigma} D'_r(x, X_{\mu(x)}) \subset \bigcap_{x \in \sigma} D'_r(x, X'_{\nu(x)})
\]
and hence $\sigma \in \Alpha_r(X,\nu)$.
An important consequence of this fact is that $\Alpha_r(X,\mu)$ is a strong deformation retract of $\Alpha_r(X,\nu)$ for every $r \geq 0$.
This is a consequence of Corollary 0.20 in~\cite{hatcher_topology_2002} applied to the inclusion $\Alpha_r(X,\mu) \hookrightarrow \Alpha_r(X,\nu)$, which is a homotopy equivalence by the nerve lemma.

\begin{remark}\label{rem: functoriality of chromatic alpha filtration}
	We can summarize the functoriality of the chromatic alpha filtration in terms of a category of chromatic sets.
	Let $\cat{ChrSet}_d$ be the category defined as follows:
	\begin{itemize}
		\item Objects are pairs $(X, \mu)$ where $X \subset \RR^d$ is a finite point set and $\mu$ is a colouring of $X$.
		\item There is a morphism $(X, \mu) \to (Y, \nu)$ if $X \subset Y$ and $\nu \preceq \mu$.
	\end{itemize}
	Then the chromatic alpha filtration is a functor
	\begin{equation}
		\begin{array}{rcl}
			\Alpha: \cat{ChrSet}_d \times [0, \infty] & \to     & \cat{SimpComp}, \\
			(X, \mu, r)                               & \mapsto & \Alpha_r(X,\mu)
		\end{array}
	\end{equation}
	that takes morphisms in the domain to inclusions of subcomplexes.
	More precisely, the inclusion $\Alpha_{s}(X,\mu) \hookrightarrow \Alpha_t(Y,\nu)$ is the simplicial map induced by the vertex map that is the inclusion of $X$ into $Y$.
	Let $\Fcal$ be the forgetful functor $\cat{ChrSet}_d \to \cat{Point}_d$ that forgets the colouring.
	Then the persistent homotopy equivalence between the chromatic alpha filtration and the offset filtration can be restated as follows: there is a natural transformation $\Alpha \Rightarrow \Offset\circ(\Fcal \times \id_{[0,\infty]})$ that is a pointwise homotopy equivalence, yielding commutative diagrams of the form:
	\begin{diagram}[column sep=0.5ex, row sep=2ex]\label{dgm: functoriality of chromatic alpha filtration}
		& {|\Alpha_s(Y,\nu)|} && {|\Alpha_t(Y,\nu)|} \\
		{|\Alpha_s(X,\mu)|} && {|\Alpha_t(X,\mu)|} \\
		& {\Offset_s(Y)} && {\Offset_t(Y)} \\
		{\Offset_s(X)} && {\Offset_t(X)}
		\arrow[hook, from=1-2, to=1-4]
		\arrow[hook, from=2-1, to=2-3]
		\arrow[hook, from=2-1, to=1-2]
		\arrow[hook, from=2-3, to=1-4]
		\arrow[hook, from=4-1, to=3-2]
		\arrow[hook, from=4-1, to=4-3]
		\arrow[hook, from=3-2, to=3-4]
		\arrow[hook, from=4-3, to=3-4]
		\arrow["\simeq", swap, from=2-1, to=4-1]
		\arrow["\simeq"{pos=0.3}, from=1-2, to=3-2]
		\arrow["\simeq"{pos=0.3}, from=2-3, to=4-3]
		\arrow["\simeq", from=1-4, to=3-4]
	\end{diagram}
\end{remark}
Similar to the alpha filtration, there is an equivalent characterization of the chromatic alpha filtration in terms of spheres.
Suppose $\sigma \in \Alpha_r(X,\mu)$.
By definition there exists a point $z \in \bigcap_{x \in \sigma} D'_r(x, X_{\mu(x)})$.
In particular for each $m \in \mu(\sigma)$ we have $\displaystyle z \in \bigcap_{x \in \sigma \cap X_m} D'_r(x, X_m)$.
This implies that for each $m \in \mu(\sigma)$, there is a sphere $S_m$ of radius at most $r$ and centred at $z$ such that $S_m$ is empty of points in $X_m$ and passes through all the points of $\sigma$ of colour $m$.
This motivates the following definitions.
For any subset of colours $\gamma \subset \{0, \ldots, s\}$ we define a $\gamma$-\textit{stack} to be a collection of $|\gamma|$ concentric $(d-1)$-spheres $S = (S_m)_{m \in \gamma}$.
The radius of the stack $S$ is the radius of the largest sphere in $S$, denoted $\Rad(S)$, and the \textit{centre} of the stack $S$ is the common centre of all the spheres in $S$.
We say that a $\gamma$-stack $S = (S_m)_{m \in \gamma}$ \textit{passes through $\sigma$} if each $S_m$ is a circumsphere of the points of $\sigma$ of colour $m$ and each point of $\sigma$ lies on a sphere in $S$.
We say that $S$ is an \textit{empty stack} if each $S_m$ is empty of points of $X$ of colour $m$ (see \Cref{fig: empty stack}).
Hence, we have
\begin{equation}
	\Alpha_r(X,\mu) = \{\sigma \subset X \mid \text{there exists an empty stack of radius $\leq r$ passing through $\sigma$}\}.
	\label{eq: alternative definition of chromatic alpha filtration}
\end{equation}
If $\sigma \in \Alpha_{\infty}(X,\mu)$ then we let $S(\sigma,\mu)$ denote the empty stack of minimum radius that passes through $\sigma$.
Notice that if $X$ is monochromatic then $S(\sigma,\mu)$ is the empty circumsphere of minimum radius for $\sigma$, and when $\mu$ is the maximal colouring then the largest sphere of $S(\sigma,\mu)$ is simply the minimum bounding ball of $\sigma$.

\subsection{Chromatic Delaunay Triangulation and General Position Conditions}\label{sec: chromatic Delaunay triangulation}

We now describe a particular geometric realization of the underlying simplicial complex of the chromatic alpha filtration.
As before, let $\mu = (X_0, \ldots, X_s)$ be a colouring of $X\subset \RR^d$.
Let $\{e_0, \ldots, e_s\} \subset \RR^s$ be the vertices of the $s$-simplex in $\RR^s$ formed by the zero vector and the standard basis vectors.
We define the \textit{chromatic lift} of $X$ with respect to $\mu$ as the subset of $\RR^{d+s}$ given by
\begin{equation}
	X^{\mu} = \bigcup_{m=0}^s X_m \times \{e_m\},
	\label{eq: definition of chromatic lift}
\end{equation}
i.e., points of distinct colours are lifted to distinct parallel $d$-dimensional affine subspaces of $\RR^{d+s}$.
For $x \in X$ we let $x^{\mu}$ denote its lift in $X^{\mu}$.

Lemma 3.6 in~\cite{Montesano2025chromatic} shows that there is a bijective correspondence between $[s+1]$-stacks in $\RR^d$ and $(d+s-1)$-spheres in $\RR^{d+s}$ that intersect each of the $d$-planes $\RR^d \times e_m$ for $m \in \{0, \ldots, s\}$.
In one direction, a $(d+s-1)$-sphere $S$ in $\RR^{d+s}$ corresponds to the stack comprised of the $(d-1)$-spheres obtained by intersecting $S$ with each of the $d$-planes $\RR^d \times \{e_m\}$ for each $m \in \{0, \ldots, s\}$ (see \Cref{fig: correspondence between stacks and spheres}).
\begin{figure}
	\centering
	\begin{subfigure}[t]{0.33\textwidth}
		\centering
		\begin{tikzpicture}
	\coordinate (a) at (-1,0);
	\coordinate (c) at (1,0);
	\coordinate (b) at (0,2);
	\coordinate (d) at (0,-2);
	\coordinate (centre) at (0, 0);
	\coordinate (unlabelled_orange_1) at (1, 1.5);
	\coordinate (unlabelled_orange_2) at (-1.5, -0.7);
	\coordinate (unlabelled_blue_1) at (-2, 1.8);
	\node[OLine] at (centre) [circle through={(a)}]{};
	\node[BLine] at (centre) [circle through={(b)}]{};
	\node[OPoint] at (a) {} node[Label, left  = 3pt of a] {$a$};
	\node[OPoint] at (c) {} node[Label, left  = 3pt of c] {$c$};
	\node[BPoint] at (b) {} node[Label, below = 3pt of b] {$b$};
	\node[BPoint] at (d) {} node[Label, below = 3pt of d] {$d$};
	\node[Point]  at (centre) {} node[Label, below = 3pt of centre] {$o$};
	\node[OPoint] at (unlabelled_orange_1) {};
	\node[OPoint] at (unlabelled_orange_2) {};
	\node[BPoint] at (unlabelled_blue_1) {};
\end{tikzpicture}
		\caption{An empty stack passing through the chromatic set $\{a, b, c, d\}$.}
		\label{fig: empty stack}
	\end{subfigure}
	\hfill
	\begin{subfigure}[t]{0.60\textwidth}
		\centering
		\begin{tikzpicture}[
		scale=0.9,
		x={(0:1cm)},
		y={(45:2mm)},
		z={(90:1cm)}
	]
	\coordinate (a) at (-1,0,3);
	\coordinate (c) at (1,0,3);
	\coordinate (b) at (0,2, 0);
	\coordinate (d) at (0,-2, 0);
	\coordinate (unlabelled_orange_1) at (1, 1.5, 3);
	\coordinate (unlabelled_orange_2) at (-1.5, -0.7, 3);
	\coordinate (unlabelled_blue_1) at (-2, 1.8, 0);
	\coordinate (orect1) at (-4, -4, 3);
	\coordinate (orect2) at (4, -4, 3);
	\coordinate (orect3) at (4, 4, 3);
	\coordinate (orect4) at (-4, 4, 3);
	\coordinate (brect1) at (-4, -4, 0);
	\coordinate (brect2) at (4, -4, 0);
	\coordinate (brect3) at (4, 4, 0);
	\coordinate (brect4) at (-4, 4, 0);
	\coordinate (orectlabel) at (-4, 0, 3);
	\coordinate (brectlabel) at (-4, 0, 0);

	\fill[OShape] (orect1) -- (orect2) -- (orect3) -- (orect4) -- cycle;
	\node[left=3pt of orectlabel] {$\RR^d \times e_1$};

	\fill[BShape] (brect1) -- (brect2) -- (brect3) -- (brect4) -- cycle;
	\node[left=3pt of brectlabel] {$\RR^d \times e_0$};

	\begin{scope}[canvas is xy plane at z=3]
		\draw[OLine] (1, 0) arc (0:180:1) {};
	\end{scope}

	\begin{scope}[canvas is xy plane at z=0]
		\draw[BLine] (2, 0) arc (0:180:2) {};

		\draw[OLine, dashed] (0, 0) circle[radius=1] {};
	\end{scope}

	\begin{scope}[canvas is xz plane at y=0]
		\draw (0, 1) circle [radius=2.236] {};
	\end{scope}

	\begin{scope}[canvas is xy plane at z=0]
		\draw[BLine] (2, 0) arc (0:-180:2) {};
	\end{scope}

	\draw[white, ultra thick] (a) -- (c) -- (b) -- (a);
	\draw[GLine] (a) -- (c) -- (b) -- (a);
	\draw[white, ultra thick] (b) -- (d);
	\draw[GLine] (b) -- (d);
	\draw[white, ultra thick] (a) -- (d);
	\draw[GLine] (a) -- (d);
	\draw[white, ultra thick] (c) -- (d);
	\draw[GLine] (c) -- (d);

	\begin{scope}[canvas is xy plane at z=3]
		\draw[OLine] (1, 0) arc (0:-180:1) {};
	\end{scope}

	\node[OPoint] at (unlabelled_orange_1) {};
	\node[OPoint] at (unlabelled_orange_2) {};
	\node[OPoint] at (a) {} node[Label, above = 3pt of a] {$a^{\mu}$};
	\node[OPoint] at (c) {} node[Label, below left = 8pt of c] {$c^{\mu}$};
	\node[BPoint] at (b) {} node[Label, above right = 2.5pt and 2.5pt of b] {$b^{\mu}$};
	\node[BPoint] at (d) {} node[Label, below = 3pt of d] {$d^{\mu}$};
	\node[BPoint] at (unlabelled_blue_1) {};
\end{tikzpicture}
		\caption{An empty circumsphere passing through the chromatic lift of $\{a, b, c, d\}$.}
		\label{fig: empty circumsphere}
	\end{subfigure}
	\caption{
		Correspondence between stacks and spheres.
		For $X=\{a, b, c, d\}$ with colour classes $\{a, c\}$ and $\{b, d\}$, the abstract 3-simplex $abcd$ belongs to $\Alpha_{\infty}(X,\mu)$, and has geometric realization as the tetrahedron $(abcd)^{\mu}$ in $\Del(X,\mu)$.
	}
	\label{fig: correspondence between stacks and spheres}
\end{figure}
Under this correspondence, an $[s+1]$-stack passing through $\sigma\subset X$ corresponds to a $(d+s-1)$-sphere passing through $\sigma^{\mu} \subset X^{\mu}$, and empty stacks correspond to empty spheres.
In light of \Cref{eq: alternative definition of alpha filtration,eq: alternative definition of chromatic alpha filtration}, the correspondence implies that there is an isomorphism $\Alpha_{\infty}(X,\mu) \cong \Alpha_{\infty}(X^{\mu})$.
Therefore, the \textit{chromatic Delaunay mosaic} $\Del(X,\mu)$ is defined to be the Delaunay mosaic $\Del(X^{\mu})$.
If $\Del(X,\mu)$ is simplicial then it is the unique geometric realization of $\Alpha_{\infty}(X,\mu)$ with vertex set $X^{\mu}$, and it is called the \textit{chromatic Delaunay triangulation} of $X$.
Unlike $\Alpha_{\infty}(X,\mu)$, the geometric complex $\Del(X,\mu)$ is not invariant to permuting the class labels of $\mu$.
However, such a permutation induces a linear isometry of the ambient space $\RR^{d+s}$.
If $\mu'$ is the colouring of $X$ after changing $\mu$ by a permutation of class labels, then there is an isometric isomorphism of geometric simplicial complexes $\Del(X,\mu') \cong \Del(X,\mu)$, which is the restriction of the ambient linear isometry.

We will need a condition on $X$ to ensure that $\Del(X,\mu)$ is simplicial for all colourings $\mu$ of $X$, namely that $X^{\mu}$ satisfies \Cref{prop: GP1} for all $\mu$.
By the correspondence between stacks and spheres, this requirement can be formulated in terms of stacks.
\begin{property}{GP2}\label{prop: GP2}
	We say $X \subset \RR^d$ satisfies \Cref{prop: GP2} if for any colouring $\mu: X \to \{0, \ldots, s\}$, any $[s+1]$-stack passes through at most $(d+s+1)$-points.
\end{property}
Notice that \Cref{prop: GP2} is a generalization of \Cref{prop: GP1}, in that $X$ satisfies \Cref{prop: GP1} if it satisfies \Cref{prop: GP2} for the monochromatic colouring.
For our proofs we will also need a condition to ensure that the chromatic Delaunay mosaic has non-empty interior in $\RR^{d+s}$ if $X$ has sufficiently many points.
\begin{example}
	If $X$ consists of 3 collinear points on the plane each having the same colour, then $\Del(X, \mu)$ is a 1-dimensional simplicial complex in $\RR^2$.
\end{example}%
\begin{example}\label{eg: trapezium}
	Let $a = (0, 0)$, $b = (0, 1)$, $c = (1, 0)$, and $d = (1, 2)$.
	Then $X=\{a,b,c,d\}$ comprises the four corners of a non-rectangular trapezium in $\RR^2$ with $ab$ parallel to $cd$.
	Define the colouring $\mu = (\{a,b\},\{c,d\})$; then one can check that $\Del(X,\mu)$ is a 2-dimensional simplicial complex in $\RR^3$.
\end{example}
The problem in the first example above is that $X$ is not in general linear position, while the problem in the second example is that there are parallel linear subspaces (the lines $ab$ and $cd$) generated by points in $X$.
To avoid situations like these, we will place restrictions on the linear subspaces generated by subsets of $X$.
The \textit{lineality space} of $P \subset X$ is the linear subspace $\Lin(P)$ of $\RR^d$ (passing through the origin) that is parallel to the affine hull of $P$.
It is equivalently the set of linear combinations $\sum_{p \in P} \lambda_p p$ of elements in $P$ with coefficients that sum to zero: $\sum_{p \in P} \lambda_p = 0$.
\begin{property}{GP3}\label{prop: GP3}
	We say that $X \subset \RR^d$ satisfies \Cref{prop: GP3} if for all $k \geq 0$, all subsets $P \subset X$ with $|P| \leq d + k + 1$, and any partition of $P$ into $k+1$ non-empty subsets $(P_0, \ldots, P_k)$, we have
	\begin{equation}
		\dim\left(\sum_{i=0}^k \Lin(P_i) \right) = \sum_{i=0}^k \dim(\Lin(P_i)) = |P| - k - 1.
	\end{equation}
\end{property}
\begin{remark}
	For any $P \subset X$ with $|P| \leq d + k + 1$ and a partition of $P$ into $k+1$ non-empty subsets $(P_0, \ldots, P_k)$, basic linear algebra shows that
	\[
		\dim\left(\sum_{i = 0}^k \Lin(P_i)\right) \leq \sum_{i=0}^k \dim(\Lin(P_i)) \leq |P| - k - 1.
	\]
	Thus, \Cref{prop: GP3} is the assertion that the inequalities in the above expression are equalities.
	The second inequality is an equality if and only if each $P_i$ is an affinely independent set, and taken over arbitrary $P$ and $k=0$, this ensures that $X$ is in general linear position.
	The first inequality is an equality if and only if $\sum_{i=0}^k \Lin(P_i)$ is a direct sum, which addresses the issue of parallel linear subspaces.%
\end{remark}
\begin{example}
	A set of points in $\RR^2$ satisfies \Cref{prop: GP3} if and only if no three points are collinear and no two lines formed by pairs of points are parallel.
\end{example}

\begin{lemma}\label{thm: dimension of chromatic lift}
	Let $X \subset \RR^d$ and let $\mu:X \to \{0, \ldots, s\}$ be any colouring of $X$.
	Suppose $X$ satisfies \Cref{prop: GP3}.
	Then the affine hull of $X^{\mu}$ in $\RR^{d+s}$ has dimension $k = \min(|X|-1, d+s)$.
\end{lemma}
\begin{proof}
	We will find set of $k+1$ points in $X^{\mu}$ that are affinely independent.
	Choose a subset $P^{\mu} = \{x_0^{\mu}, \ldots, x_{k}^{\mu}\}$ of points in $X^{\mu}$ that contains a point of every colour, i.e., $P \cap X_i \neq \emptyset$ for each $i \in \{0, \ldots, s\}$.
	Suppose there exist coefficients $\{\lambda_i\}_{i=0}^{k}$ such that $\sum_{i=0}^{k} \lambda_i = 0$ and $\sum_{i = 0}^{k} \lambda_i x_i^{\mu} = 0$.
	We will show that each $\lambda_i = 0$, thereby showing affine independence of $P^{\mu}$.

	Now notice that for each $x \in P_i$, by construction we have $x^{\mu} = (x^T, 0, \ldots, 0, 1, 0, \ldots, 0)^T$ where the unit entry appears in the $(d+i)$\textsuperscript{th} position.
	Therefore, we must have $\sum_{\mu(x_j) = m} \lambda_j = 0$ for each $i\in I$.
	Let $y_i = \sum_{\mu(x_j) = i} \lambda_j x_j$; then $y_i \in \Lin(P_i)$.
	Now $y_i$ is the projection of $\sum_{\mu(x_j) = i} \lambda_j x_j^{\mu}$ to $\RR^d$ from $\RR^{d+s}$, and hence
	\begin{equation*}
		0 = \sum_{x_i \in P} \lambda_i x_i^{\mu} \implies \sum_{i \in I} y_i = 0.
	\end{equation*}
	Since $X$ satisfies \Cref{prop: GP3}, we have that $y_i = 0$ for each $i \in I$.
	Moreover, \Cref{prop: GP3} also implies affine independence of each $P_j$, whence it follows that $\lambda_i = 0$ for all $i$.
	This shows that $P^{\mu}$ is an affinely independent set as desired.
\end{proof}

We say that $X$ is in \textit{general position} if it satisfies \Cref{prop: GP2} and \Cref{prop: GP3}.
Our terminology is justified by the following statement which we prove in \Cref{sec: genericity of general position assumptions}: for any $d, n \geq 1$, the space of $d$-dimensional point clouds with $n$ points that are not in general position is contained in a semi-algebraic subset of $\RR^{nd}$ of dimension at most $nd-1$.
We note that in Section 4.1 of~\cite{Montesano2025chromatic}, the authors define a notion of `chromatic genericity' that is very similar to our definition for general position.
However, our notion of general position is strictly stronger than chromatic genericity; see \Cref{sec: general position vs chromatic genericity} for details.

The previous lemma immediately implies the following.
\begin{corollary}
	Suppose $X \subset \RR^d$ is in general position and $\mu : X \to \{0, \ldots, s\}$ is a colouring of $X$.
	Then the chromatic Delaunay mosaic $\Del(X,\mu)$ is a simplicial complex in which every maximal simplex has dimension $\min(|X| - 1, d + s)$.
\end{corollary}
\begin{remark}\label{rem: lifts and subsets of general position}
	If $X$ is in general position and $\mu$ is any colouring of $X$ then $X^{\mu}$ is in general position.
	If $E$ is any subset of $X$ then $E$ is in general position.
\end{remark}
\begin{remark}\label{rem: functoriality of chromatic Delaunay triangulations}
	Chromatic Delaunay triangulations retain the functoriality of chromatic alpha filtrations, in that whenever $X$ and $Y$ are in general position and $(X, \mu) \to (Y, \nu)$ is a morphism in $\cat{ChrSet}_d$ (the category defined in \Cref{rem: functoriality of chromatic alpha filtration}), then there is an embedding (not necessarily isometric) of $\Del(X,\mu)$ as a subcomplex of $\Del(Y,\nu)$.
	This is because $\Del(X,\mu) \cong \Alpha_{\infty}(X,\mu)$ and $\Del(Y,\nu) \cong \Alpha_{\infty}(Y,\nu)$, and by \Cref{rem: functoriality of chromatic alpha filtration} we know that there is an inclusion $\Alpha_{\infty}(X,\mu) \subset \Alpha_{\infty}(Y,\nu)$ induced by the inclusion of $X$ in $Y$.
	Correspondingly, the embedding of $\Del(X,\mu)$ into $\Del(Y,\nu)$ is induced by the vertex map that sends $x^{\mu} \in X^{\mu}$ to $x^{\nu} \in Y^{\nu}$.
\end{remark}
\begin{figure}[h]
	\centering
	\begin{subfigure}[t]{0.32\linewidth}
		\centering
		\resizebox{\linewidth}{!}{\input{figures/membrane_example/monochromatic_voronoi.pgf}}
		\caption{}
	\end{subfigure}
	\hfill
	\begin{subfigure}[t]{0.32\linewidth}
		\centering
		\resizebox{\linewidth}{!}{\begin{tikzpicture}[
		scale=2.5,
	]
	\coordinate (v0) at (-1.00, -0.05);
	\coordinate (v1) at ( 0.80, -0.08);
	\coordinate (v2) at (-0.45,  0.08);
	\coordinate (v3) at ( 0.00,  0.40);
	\coordinate (v4) at ( 0.00, -0.40);
	\coordinate (v5) at ( 0.45, -0.10);
	\draw[GShape, GLine] (v0) -- (v4) -- (v2) -- cycle;
	\draw[GShape, GLine] (v1) -- (v5) -- (v3) -- cycle;
	\draw[GShape, GLine] (v1) -- (v5) -- (v4) -- cycle;
	\draw[GShape, GLine] (v0) -- (v3) -- (v2) -- cycle;
	\draw[GShape, GLine] (v2) -- (v3) -- (v4) -- cycle;
	\draw[GShape, GLine] (v5) -- (v3) -- (v4) -- cycle;

	\node[GPoint] at (v0) {} node[Label, left        = 3pt of v0] {$v_0$};
	\node[GPoint] at (v1) {} node[Label, right       = 3pt of v1] {$v_1$};
	\node[GPoint] at (v2) {} node[Label, right       = 6pt of v2] {$v_2$};
	\node[GPoint] at (v3) {} node[Label, above       = 3pt of v3] {$v_3$};
	\node[GPoint] at (v4) {} node[Label, below       = 3pt of v4] {$v_4$};
	\node[GPoint] at (v5) {} node[Label, left        = 6pt of v5] {$v_5$};
\end{tikzpicture}}
		\caption{}
	\end{subfigure}
	\hfill
	\begin{subfigure}[t]{0.32\linewidth}
		\centering
		\resizebox{\linewidth}{!}{\input{figures/membrane_example/bichromatic_voronoi.pgf}}
		\caption{}
	\end{subfigure}
	\vspace{\baselineskip}
	\begin{subfigure}[t]{0.31\textwidth}
		\centering
		\begin{tikzpicture}[
				scale=2.5,
				x={(40:8mm)},
				y={(160:8mm)},
				z={(90:8mm)},
			]
			\draw[->] (0,0,0) -- (1,0,0);
			\draw[->] (0,0,0) -- (0,1,0);
			\draw[->] (0,0,0) -- (0,0,1);
			\node[Label, above right=3pt of {(1,0,0)}] {$x$};
			\node[Label, above left=3pt of {(0,1,0)}] {$y$};
			\node[Label, above =3pt of {(0,0,1)}] {$z$};
		\end{tikzpicture}
	\end{subfigure}
	\hfill
	\begin{subfigure}[t]{0.31\textwidth}
		\centering
		\begin{tikzpicture}[
		scale=2.5,
		x={(40:8mm)},
		y={(160:12mm)},
		z={(90:1cm)},
	]

	\coordinate (v0) at (-1.00, -0.05, 1);
	\coordinate (v1) at ( 0.80, -0.08, 1);
	\coordinate (v2) at (-0.45,  0.08, 1);
	\coordinate (v0') at (-1.00, -0.05, 0);
	\coordinate (v1') at ( 0.80, -0.08, 0);
	\coordinate (v2') at (-0.45,  0.08, 0);
	\coordinate (v3) at ( 0.00,  0.40, 0);
	\coordinate (v4) at ( 0.00, -0.40, 0);
	\coordinate (v5) at ( 0.45, -0.10, 0);

	\draw[GLine, dashed] (v4) -- (v1') -- (v3);
	\draw[GLine, dashed] (v1') -- (v5);
	\draw[GLine, dashed] (v3) -- (v2') -- (v4);
	\draw[GLine, dashed] (v3) -- (v0') -- (v4);
	\draw[GLine, dashed] (v2') -- (v0');

	\draw[GLine, GShape] (v1) -- (v5) -- (v3) -- cycle;
	\draw[GLine, GShape] (v3) -- (v5) -- (v4) -- cycle;
	\draw[GLine, GShape] (v1) -- (v5) -- (v4) -- cycle;
	\draw[GLine, GShape] (v0) -- (v2) -- (v3) -- cycle;
	\draw[GLine, GShape] (v2) -- (v3) -- (v4) -- cycle;
	\draw[GLine, GShape] (v0) -- (v2) -- (v4) -- cycle;

	\node[OPoint] at (v0)  {} node[Label, left       = 3pt of v0] {$v_0^{\mu}$};
	\node[OPoint] at (v1)  {} node[Label, right      = 3pt of v1] {$v_1^{\mu}$};
	\node[OPoint] at (v2)  {} node[Label, above      = 3pt of v2] {$v_2^{\mu}$};
	\node[OPoint] at (v0') {} node[Label, left       = 3pt of v0'] {$v_0$};
	\node[OPoint] at (v1') {} node[Label, right      = 3pt of v1'] {$v_1$};
	\node[OPoint] at (v2') {} node[Label, above right= 3pt of v2'] {$v_2$};
	\node[BPoint] at (v3)  {} node[Label, left       = 3pt of v3] {$v_3^{\mu}$};
	\node[BPoint] at (v4)  {} node[Label, right      = 3pt of v4] {$v_4^{\mu}$};
	\node[BPoint] at (v5)  {} node[Label, above left = 3pt of v5] {$v_5^{\mu}$};

\end{tikzpicture}
		\caption{}
	\end{subfigure}
	\hfill
	\begin{subfigure}[t]{0.31\textwidth}
		\centering
		\begin{tikzpicture}[
		scale=2.5,
		x={(40:8mm)},
		y={(160:12mm)},
		z={(90:1cm)},
	]

	\coordinate (v0)  at (-1.00, -0.05, 1);
	\coordinate (v1)  at ( 0.80, -0.08, 1);
	\coordinate (v2)  at (-0.45,  0.08, 1);
	\coordinate (v3)  at ( 0.00,  0.40, 0);
	\coordinate (v4)  at ( 0.00, -0.40, 0);
	\coordinate (v5)  at ( 0.45, -0.10, 0);

	\draw[BShape, BLine] (v3) -- (v4) -- (v5) -- cycle;
	\draw[GLine] (v0) -- (v3);
	\draw[GLine] (v1) -- (v3);
	\draw[GLine] (v1) -- (v4);
	\draw[GLine] (v1) -- (v5);
	\draw[GLine] (v2) -- (v3);
	\draw[white, ultra thick] (v2) -- (v5);
	\draw[GLine] (v2) -- (v5);
	\draw[white, ultra thick] (v2) -- (v4);
	\draw[GLine] (v2) -- (v4);
	\draw[white, ultra thick] (v0) -- (v4);
	\draw[GLine] (v0) -- (v4);
	\draw[OShape, OLine] (v0) -- (v1) -- (v2) -- cycle;

	\node[OPoint] at (v0)  {} node[Label, left  = 3pt of v0] {$v_0^{\mu}$};
	\node[OPoint] at (v1)  {} node[Label, right = 3pt of v1] {$v_1^{\mu}$};
	\node[OPoint] at (v2)  {} node[Label, above = 3pt of v2] {$v_2^{\mu}$};
	\node[BPoint] at (v3)  {} node[Label, left  = 3pt of v3] {$v_3^{\mu}$};
	\node[BPoint] at (v4)  {} node[Label, right = 3pt of v4] {$v_4^{\mu}$};
	\node[BPoint] at (v5)  {} node[Label, above = 5pt of v5, fill = white] {$v_5^{\mu}$};

\end{tikzpicture}
		\caption{}
	\end{subfigure}
	\caption{
		Chromatic Delaunay triangulation of a bi-chromatic set $X \subset \RR^2$.
		(a) Voronoi tessellation of $X$.
		(b) Delaunay triangulation $\Del(X)$ of $X$.
		(c) Overlay of Voronoi tessellations of blue and orange points in $X$.
		(d) Embedded image $\iota(\Del(X))$ of the Delaunay triangulation (shaded gray) in the chromatic Delaunay triangulation $\Del(X,\mu)$.
		The 1-simplices of $\Del(X)$ are shown with dashed lines, and $\iota(\Del(X))$ is the graph of a piecewise-linear function on $\Del(X)$.
		(e) $\Del(X,\mu)$ with bicoloured 2- and 3-simplices left unshaded for clarity.
		For figures (d) and (e) the coordinate axes are shown for reference at the bottom left.
	}
	\label{fig: membrane example}
\end{figure}
\begin{remark}\label{rem: description of membrane}
	Suppose $\mu: X \to \{0,1\}$ is a colouring of $X$.
	By \Cref{rem: functoriality of chromatic Delaunay triangulations} there is an embedding of simplicial complexes $\iota: \Del(X) \hookrightarrow \Del(X,\mu)$ induced by the vertex map $x \mapsto x^{\mu}$.
	Let $h:\RR^{d+1} \to \RR$ be the projection to the last coordinate.
	It is straightforward to check that $\iota(\Del(X))$ is the graph of the piecewise-linear function $h\circ \iota: \Del(X) \to \RR$.
	See \Cref{fig: membrane example} for an example.

	More generally suppose $\mu$ and $\nu$ are colourings of $X$ such that $\mu$ is obtained from $\nu$ by merging the last two classes of $\nu$.
	That is, $\nu = (X_0, \ldots, X_{s+1})$ and $\mu = (X_0, \ldots, X_{s-1}, X_{s} \cup X_{s+1})$.
	Then by \Cref{rem: functoriality of chromatic Delaunay triangulations}, there is an embedding $\iota: \Del(X,\mu) \hookrightarrow \Del(X,\nu)$ induced by the vertex map $x^{\mu} \mapsto x^{\nu}$.
	Let $\zeta: X^{\mu} \to \{0,1\}$ be the colouring that sends $x^{\mu}$ to $0$ if $\mu(x) = \nu(x)$, and $1$ otherwise.
	Then $X^{\nu}$ is the chromatic lift $(X^{\mu})^{\zeta}$.
	It follows from our discussion of the bicoloured case that $\iota(\Del(X,\mu))$ is the graph of the function $h \circ \iota : \Del(X,\mu) \to \RR$ where $h : \RR^{d+s+1} \to \RR$ is the projection to the last coordinate.
\end{remark}

\subsection{Chromatic Delaunay--\v{C}ech and Chromatic Delaunay--Rips Filtrations}\label{sec: chromatic DelCech and chromatic DelRips filtrations}

Given a chromatic set $(X, \mu)$, we define the \textit{chromatic Delaunay--\v{C}ech} filtration as follows:
\begin{equation}\label{eq: definition of chromatic DelCech filtration}
	\DelCech_t(X,\mu) \defeq \Cech_t(X) \cap \Alpha_{\infty}(X,\mu)
\end{equation}
In other words, the chromatic Delaunay--\v{C}ech filtration has the same underlying simplicial complex as the chromatic alpha filtration, and the simplices are assigned filtration values from the \v{C}ech filtration.
Using the definitions in \Cref{eq: alternative definition of Cech filtration,eq: alternative definition of chromatic alpha filtration} and noting that the outermost sphere of an empty stack passing through some $\sigma \subset X$ defines a bounding ball of $\sigma$, we see that the chromatic alpha filtration is a sub-filtration of the chromatic Delaunay--\v{C}ech filtration.
More precisely, if $\cat{ChrSet}_d$ is the category defined in \Cref{rem: functoriality of chromatic alpha filtration}, then $\DelCech$ is a functor from $\cat{ChrSet}_d \times [0, \infty]$ to $\cat{SimpComp}$, and if $(X, \mu) \to (Y, \nu)$ is a morphism in $\cat{ChrSet}_d$ and $0 \leq s \leq t \leq \infty$ then the following diagram commutes.

\begin{diagram}[column sep=0.5ex, row sep=1ex]\label{dgm: functoriality of chromatic DelCech filtration}
	& \Alpha_s(Y,\nu) && \Alpha_t(Y,\nu)\\
	\Alpha_{s}(X,\mu) && \Alpha_t(X,\mu)\\
	& \DelCech_{s}(Y,\nu) && \DelCech_t(Y,\nu)\\
	\DelCech_{s}(X,\mu) && \DelCech_t(X,\mu)\\
	& \Cech_s(Y) && \Cech_t(Y)\\
	\Cech_s(X) && \Cech_t(X)
	\ar[from=1-2, to=1-4, hook]
	\ar[from=1-2, to=3-2, hook]
	\ar[from=1-4, to=3-4, hook]
	\ar[from=2-1, to=1-2, hook]
	\ar[from=2-1, to=2-3, hook, crossing over]
	\ar[from=2-1, to=4-1, hook]
	\ar[from=2-3, to=1-4, hook]
	\ar[from=3-2, to=3-4, hook]
	\ar[from=2-3, to=4-3, hook, crossing over]
	\ar[from=3-2, to=5-2, hook]
	\ar[from=3-4, to=5-4, hook]
	\ar[from=4-1, to=3-2, hook]
	\ar[from=4-1, to=4-3, hook, crossing over]
	\ar[from=4-1, to=6-1, hook]
	\ar[from=4-3, to=3-4, hook]
	\ar[from=5-2, to=5-4, hook]
	\ar[from=4-3, to=6-3, hook, crossing over]
	\ar[from=6-1, to=5-2, hook]
	\ar[from=6-1, to=6-3, hook]
	\ar[from=6-3, to=5-4, hook]
\end{diagram}

We define the \textit{chromatic Delaunay--Rips} filtration similarly.
Recall that the \textit{Vietoris--Rips} filtration on $X$, denoted by $\VR_{\bullet}(X)$, is defined by
\begin{equation}
	\VR_r(X) \defeq \{\sigma \subset X \mid \forall x, y \in \sigma:\ \norm{x - y} \leq 2r \}.
\end{equation}
Then the chromatic Delaunay--Rips filtration is given by
\begin{equation}\label{eq: definition of chromatic DelRips filtration}
	\DelVR_t(X,\mu) \defeq \VR_t(X) \cap \Alpha_{\infty}(X,\mu).
\end{equation}
The Vietoris--Rips filtration is related to the \v{C}ech filtration by a series of nested inclusions, and a similar property holds for their chromatic counterparts as defined here.
We will explore this aspect in more detail in \Cref{sec: consequences for practical computations}.

\subsection{Generalities on Geometric Simplicial Complexes}\label{sec: geometric simplicial complexes}

In this section we define some notions related to geometric realizations of simplicial complexes.
A \textit{geometric $k$-simplex} $\sigma^{(k)}$ in $\RR^{d+1}$ is a linear embedding of the standard $k$-simplex in $\RR^{d+1}$.
We say that a vector $z \in \RR^{d+1}\setminus\{0\}$ is \textit{transverse} to a geometric simplex $\sigma^{(k)}$ if $z$ is not parallel to the affine hull of any of its $m$-dimensional faces for $m \leq \min(k, d)$.
Note that if $z$ is transverse to $\sigma^{(k)}$ then $z$ is transverse to every face of $\sigma^{(k)}$, but the converse is not necessarily true unless $k = d+1$.
Now suppose $K$ is a geometric simplicial complex in $\RR^{d+1}$.
We say that a non-zero vector $z \in \RR^{d+1}$ is \textit{transverse to $K$} if it is transverse to $\sigma$ whenever $\sigma$ is a maximal simplex in $K$; in other words, if $z$ is not parallel to the affine hull of (the embedding of) any $k$-simplex in $K$ where $k \leq d$.

Let $\tau^{(d)} < \sigma^{(d+1)}$ be geometric simplices in $\RR^{d+1}$, let $\Aff(\tau)$ be the affine hull of $\tau$ (i.e., the minimal affine subspace of $\RR^{d+1}$ that contains $\tau$), and let $v$ be the vertex of $\sigma$ opposite to $\tau$.
We say that a non-zero vector $w \in \RR^{d+1}$ is an \textit{outward-pointing normal vector} to $\tau$ with respect to $\sigma$ if $w \perp \Aff(\tau)$ and $v = x - \lambda w$ for some $x \in \Aff(\tau)$ and $\lambda > 0$.

Suppose we are given $z \in \RR^{d+1}\setminus \{0\}$ and geometric simplices $\tau^{(d)} < \sigma^{(d+1)}$ in $\RR^{d+1}$.
We say that $\tau$ is an \textit{upper (resp. lower) face of $\sigma$ with respect to $z$} if some (and hence any) outward-pointing normal vector $w$ to the face $\tau$ satisfies $\braket{w, z} > 0$ (resp. $\braket{w, z} < 0$).
If $k < d$ then we say that $\eta^{(k)} < \sigma^{(d+1)}$ is an \textit{upper (resp. lower) face of $\sigma$ with respect to $z$} if it is the intersection of $d$-dimensional upper (resp. lower) faces of $\sigma$.
When $z$ is implicit we will just use the terms ``upper face'' and ``lower face'' respectively.
We will use $\sigma^*$ and $\sigma_*$ to denote the minimal upper and lower faces of $\sigma$, respectively.
See \Cref{fig: example of upper face and outward normal} for an illustration of the terms defined here.
\begin{figure}
	\begin{subfigure}[t]{0.48\textwidth}
		\begin{equation*}
			\begin{tikzpicture}[scale=0.4]
	\coordinate (a) at (4, 5);
	\coordinate (b) at (0, 1);
	\coordinate (c) at (8, -2);
	\coordinate (x) at ($(a)!(b)!(c)$);
	\coordinate (wo) at ($(a)!0.5!(c)$);
	\coordinate (w) at ($(wo)!0.5!90:(c)$);
	\coordinate (z) at (11, 1);
	\draw [GShape, GLine] (b) -- (a) -- (c) -- cycle;
	\draw [->] (wo) -- (w);
	\draw [dashed] (b) -- (x);
	\draw [->] (z) -- ++(0, 3) node[Label, midway, right=3pt] {$z$};
	\node [Label, above = 3pt of a] {$a$};
	\node [Label, left  = 3pt of b] {$b$};
	\node [Label, right = 3pt of c] {$c$};
	\node [Label, right = 3pt of x] {$x$};
	\node [Label, right = 3pt of w] {$w$};
\end{tikzpicture}
		\end{equation*}
		\caption{
			The vector $w$ is an outward normal to $ac$ with respect to $abc$ because $w \perp ac$, and $b = x - \lambda w$ for some $\lambda > 0$.
			Since $\braket{w, z} > 0$, $ac$ is an upper face for $abc$ with respect to $z$.
		}
	\end{subfigure}
	\hfill
	\begin{subfigure}[t]{0.48\textwidth}
		\begin{equation*}
			\begin{tikzpicture}[scale=0.4]
	\coordinate (a) at (4, 5);
	\coordinate (b) at (0, 1);
	\coordinate (c) at (8, -2);
	\coordinate (x) at ($(a)!(b)!(c)$);
	\coordinate (z) at (11, 1);
	\path[name path=xvert] (x)--++(-90:5);
	\path[name path=bc] (b)--(c);
	\path[name intersections={of=xvert and bc,by=y}];
	\draw[GLine, GShape] (b)--(a)--(c)--cycle;
	\draw[dashed] (x) -- (y);
	\draw[->] (z) -- ++(0, 3) node[Label, midway, right=3pt] {$z$};
	\node[Label, above = 3pt of a] {$a$};
	\node[Label, left  = 3pt of b] {$b$};
	\node[Label, right = 3pt of c] {$c$};
	\node[Label, right = 3pt of x] {$x$};
	\node[Label, below = 3pt of y] {$y$};
\end{tikzpicture}
		\end{equation*}
		\caption{
			Illustration of \Cref{thm: characterisation of upper and lower faces}.
			Here $ac$ is an upper face because there exists $\epsilon > 0$ such that $\lambda_x=(-\epsilon , 0)$.
			Indeed, $y=x-\epsilon z$ and for any $\lambda \in (-\epsilon, 0)$ we have $(x+\lambda z) \in\topint(abc)$.
		}
	\end{subfigure}
	\caption{
		A geometric simplex $\sigma = abc$ in $\RR^2$.
		The upper faces of $\sigma$ are $ab, ac$ and $a$, and the only lower face is $bc$.
		The minimal upper face is $\sigma^* = a$ and the minimal lower face is $\sigma_* = bc$.
	}
	\label{fig: example of upper face and outward normal}
\end{figure}
\begin{remark}
	If $z \in \RR^{d+1}\setminus \{0\}$ is transverse to $\sigma^{(d+1)}$ then every $d$-dimensional face of $\sigma$ is either an upper or lower face of $\sigma$ with respect to $z$.
	However, for $k < d$ there are $k$-dimensional ``side'' faces of $\sigma$ which are neither upper nor lower.
	These are precisely the faces that each lie in the intersection of both some upper and lower $d$-dimensional faces simultaneously.
\end{remark}

Now let $\sigma^{(d+1)}$ be a geometric simplex in $\RR^{d+1}$ and let $z \in \RR^{d+1} \setminus \{0\}$ be transverse to $\sigma$.
For any $x \in \RR^{d+1}$, if $x + \lambda_1 z$ and $x + \lambda_2 z$ are both in $\sigma$ with $\lambda_1 \leq \lambda_2$, then $x + \lambda z$ is in $\sigma$ for every $\lambda \in [\lambda_1, \lambda_2]$ by convexity of $\sigma$.
It follows that the following set is an open interval:
\begin{equation}
	\lambda_x(\sigma) \defeq \{\lambda \in \RR \mid x + \lambda z \in \topint(\sigma)\}.
\end{equation}
It is not hard to see that if $x$ lies on an upper (resp. lower) face of $\sigma$ (with respect to $z$), then these intervals must be of the form $(-\epsilon, 0)$ (resp. $(0, \epsilon)$). We state this observation as a lemma:

\begin{lemma}\label{thm: characterisation of upper and lower faces}
	Let $\sigma^{(d+1)}$ be a geometric simplex in $\RR^{d+1}$, let $\tau < \sigma$, and let $z \in \RR^{d+1}\setminus \{0\}$ be transverse to $\sigma$.
	Let $\relint(\tau)$ denote the relative interior of $\tau$.
	Then the following are equivalent:
	\begin{enumerate}
		\item $\tau$ is an upper (resp. lower) face of $\sigma$.
		\item There exists $x \in \relint(\tau)$ and $\epsilon > 0$ such that $\lambda_x (\sigma)= (-\epsilon, 0)$ (resp. $(0, \epsilon)$).
		\item For every $x \in \relint(\tau)$ there exists $\epsilon > 0$ such that $\lambda_x (\sigma) = (-\epsilon, 0)$ (resp. $(0, \epsilon)$).
	\end{enumerate}
\end{lemma}

\subsection{Generalized Discrete Morse Theory}\label{sec: discrete Morse theory}

Let $K$ be a simplicial complex.
Let $\Hcal(K)$ be the Hasse diagram of the co-face poset of $K$, that is, the directed graph whose vertices are simplices in $K$ with a directed edge $\sigma \to \tau$ if $\sigma^{(p+1)} > \tau^{(p)}$.
An \textit{interval} in $\Hcal(K)$ is a collection of simplices of $K$ of the form
\begin{equation}
	[\tau, \sigma] = \{ \eta \in K \mid \tau \leq \eta \leq \sigma \}.
\end{equation}
A \textit{generalized discrete vector field} $V$ is a partition of the vertices of $\Hcal(K)$ by intervals.
If $V$ is a generalized discrete vector field, then a function $f: K \to \RR$ is a \textit{generalized discrete Morse function} with \textit{generalized discrete Morse gradient} $V$ if the following are true:
\begin{enumerate}
	\item $f$ is monotonic on $\Hcal(K)$;
	\item For any pair $\tau\leq \sigma$, $f(\sigma) = f(\tau)$ if and only if $\sigma$ and $\tau$ belong to a common interval in $V$.
\end{enumerate}
In this case the singleton intervals in $V$ are called the \textit{critical simplices} of $V$ or $f$, and we denote the set of critical simplices by $\Crit(V)$ or $\Crit(f)$.
Under these definitions, a discrete Morse gradient in the sense of Forman's original discrete Morse theory~\cite{forman_morse_1998} is nothing but a generalized gradient in which every interval has at most two simplices.
For simplicity, we will use the terms `generalized vector fields', `generalized gradients', and `generalized Morse functions', with the implicit understanding that these refer to the discrete setting in this paper.
Generalized gradients are characterized by the following lemma.
\begin{lemma}[Generalised Gradient Acyclicity Lemma]\label{thm: generalized gradient acyclicity lemma}
	Let $V$ be a generalized vector field on a simplicial complex $K$.
	Let $\Hcal(K)/V$ be the quotient graph induced by the equivalence relation on $K$ defined by the intervals in $V$.
	Then $V$ is the generalized gradient of some generalized Morse function $f$ (not necessarily unique) if and only if $\Hcal(K)/V$ is a directed acyclic graph.
\end{lemma}
\begin{proof}
	If $V$ is the generalized gradient of $f$, then $f$ descends to a well-defined function on $\Hcal(K)/V$ since it is constant on intervals in $V$.
	Furthermore, $f$ is strictly increasing across the edges of $\Hcal(K)/V$, so $\Hcal(K)/V$ must be acyclic.
	Conversely, if $\Hcal(K)/V$ is directed acyclic, let $P$ be the poset defined by the reachability graph of $\Hcal(K)/V$.
	Choose a linear extension of $P$, and for $x \in \Hcal(K)/V$ let $f(x)$ be the height of $x$ in this linear extension.
	Then $f$ is a generalized Morse function on $K$ whose generalized gradient is $V$.
\end{proof}
Generalized gradients are primarily useful because they can encode sequences of elementary simplicial collapses.
We have the following proposition~\cite[Theorem 2.2]{bauer_morse_2016}.
\begin{theorem}[Generalised Gradient Collapsing Theorem]\label{thm: generalized gradient collapsing theorem}
	Let $K$ be a simplicial complex with a subcomplex $L$.
	Then there is a simplicial collapse $K\searrow L$ if and only if there is a generalized gradient $V$ (not necessarily unique) on $K$ such that $K \setminus L$ is a union of non-critical intervals of $V$.
\end{theorem}
\begin{proof}
	The reverse implication is simply Theorem 2.2 in~\cite{bauer_morse_2016}, and we lay out a proof here for clarity.
	The gradient $V$ can be refined to a set $V'$ of intervals of length two as follows.
	For each interval $[\tau, \sigma]$, choose an arbitrary $x \in \sigma \setminus \tau$ and partition $[\tau, \sigma]$ into pairs $\{\eta \setminus \{x\}, \eta \cup \{x\}\}$ for all $\eta \in [\tau, \sigma]$.
	The set of pairs thus obtained is a genuine discrete Morse gradient (i.e., non-generalised) by the Cluster Lemma~\cite{hersh2005optimizing,jonsson_topology_2003}, and clearly $\Crit(V) = \Crit(V')$.
	It follows from the results in~\cite{forman_morse_1998} that there is a simplicial collapse $K\searrow L$, and the pairs in $V'$ correspond to the elementary collapses in the sequence of elementary collapses that define $K \searrow L$.

	For the forward implication of the theorem, we induct on the number of elementary collapses in the collapsing sequence $K \searrow L$.
	First suppose $K = L \sqcup \{\tau, \sigma\}$.
	Then it is straightforward to check that $V = \{[\tau, \sigma]\} \cup L$ is a generalized gradient on $K$ with $\Crit(V) = L$.
	The result then follows using induction and the Gradient Composition Lemma~\cite[Lemma 5.2]{bauer_morse_2016}.
\end{proof}
In Forman's original theory, the restriction of a discrete Morse function to a subcomplex is also a discrete Morse function~\cite[Lemma 4.1]{forman_morse_1998}.
The following lemma is the analogue of that result for the generalized case, and the proof is a straightforward verification.
\begin{lemma}[Generalised Gradient Restriction Lemma]\label{thm: generalized gradient restriction lemma}
	Let $V$ be a generalized gradient on a simplicial complex $K$ and let $L$ be a subcomplex of $K$ such that $K\setminus L$ is a union of intervals of $V$.
	Then the set of intervals of $V$ contained in $L$ is a generalized gradient on $L$ which we denote as $V_L$, and $\Crit(V_{L}) = \Crit(V) \cap L$.
\end{lemma}
The following lemma, stated in~\cite{bauer_morse_2016}, allows us to obtain a common refinement of a family of generalized gradients.
\begin{lemma}[Sum Refinement Lemma]\label{thm: sum refinement lemma}
	Suppose $f: K \to \RR$ and $g : L \to \RR$ are generalized Morse functions with generalized gradients $V$ and $W$ respectively. Then $f+g$ is a generalized Morse function on $K \cap L$ whose generalized gradient is given by
	\begin{equation}
		Z = \{I \cap J \mid I \in V, J \in W, I \cap J \neq \emptyset \}.
	\end{equation}
\end{lemma}

\subsection{Karush-Kuhn-Tucker Conditions}\label{sec: convex optimisation}

In \Cref{sec: KKT for stacks and pairing lemmas} we will show that the chromatic alpha filtration function is a generalized discrete Morse function.
A key observation is that the filtration value of each simplex $\sigma \in \Alpha_{\infty}(X,\mu)$ is the solution to a convex optimization problem with affine constraints.
Such problems have solutions that are characterized by the Karush-Kuhn-Tucker (KKT) conditions~\cite{karush_minima_1939,kuhn_nonlinear_1951}, which are a generalization of the method of Lagrange multipliers to non-smooth convex functions.
To state the KKT conditions we need a non-smooth analogue of the differential of a function.
If $f: \RR^d \to \RR$ is a convex function then the \textit{sub-differential} of $f$~\cite[Chapter 2]{morduchovic_easy_2014} at a point $x \in \RR^d$ is the closed convex set $\partial f(x) \subset \RR^d$ defined by
\begin{equation}
	\partial f(x) \defeq \{ v \in \RR^d \mid f(x + h) \geq f(x) + \braket{v, h} \text{ for all } h \in \RR^d\}.
\end{equation}
If $f$ is a smooth function then its sub-differential coincides with the ordinary gradient, i.e., $\partial f(x) = \{\nabla f(x)\}$~\cite[Proposition 2.36]{morduchovic_easy_2014}.

\begin{theorem}[KKT Conditions]\label{thm: KKT conditions}
	Let $f$ be a convex function on $\RR^d$, let $\Ical, \Jcal$ be finite disjoint sets, let $\Kcal = \Ical \sqcup \Jcal$, and let $\{g_k : \RR^d \to \RR\}_{k \in \Kcal}$ be a set of affine linear functions. Then $x_*$ is a minimum of the following optimization problem:
	\begin{equation*}
		\arraycolsep=3.0pt
		\begin{array}{crcll}
			\textnormal{minimise}   & f(x)                                              \\
			\textnormal{subject to} & g_i(x) & \geq & 0 & \text{ for all } i \in \Ical, \\
			                        & g_j(x) & =    & 0 & \text{ for all } j \in \Jcal,
		\end{array}
	\end{equation*}
	if and only if there exist scalars $\{\lambda_k\}_{k\in \Kcal} \subset \RR$ such that the following conditions hold:
	\begin{center}
		\begin{tabular}{>{\raggedleft\arraybackslash\upshape\bfseries}p{0.3\linewidth} p{0.65\linewidth}}
			Primal feasibility      & $g_i(x_*) \geq 0$ for all $i \in \Ical$ and $g_j(x_*) = 0$ for all $j \in \Jcal$.                                                                                                                                   \\
			Stationarity            & The zero vector belongs to the set $\partial f(x_*) + \sum_{k \in \Kcal} \lambda_k \partial g_k(x_*) \subset \RR^d$, where $\partial f$ is the sub-differential of $f$ and $\partial g_k$ is the gradient of $g_k$. \\
			Dual feasibility        & $\lambda_i \leq 0$ for all $i \in \Ical$.                                                                                                                                                                           \\
			Complementary slackness & $\lambda_i g_i(x_*) = 0$ for all $i \in \Ical$.
		\end{tabular}
	\end{center}
\end{theorem}

We will also need the following well-known lemma from the theory of convex sub-differentials~\cite[Theorem 2.93]{morduchovic_easy_2014}.
\begin{lemma}\label{thm: subdifferential of max of convex functions}
	Suppose $\{f_1, \ldots, f_m\}$ is a set of convex functions on $\RR^d$ and $f(x) \defeq \max_{1 \leq k \leq m} f_k(x)$.
	Then $\partial f(x)$ is the convex hull of the union of sub-differentials of the active functions of $f$ at $x$, i.e.,
	\begin{equation}
		\partial f(x) = \Conv\left\{ \bigcup_{f_k(x) = f(x)} \partial f_k(x) \right\}.
	\end{equation}
\end{lemma}

\section{Collapsing the Chromatic Delaunay Triangulation}\label{sec: collapse of chromatic triangulation}

Throughout this section, we fix a point cloud $X \subset \RR^d$ in general position.
The main goal of this section is to prove the following theorem:

\begin{theorem}\label{thm: chromatic_collapse}
	Suppose $X \subset \RR^d$ is in general position, and $\mu$ and $\nu$ are colourings of $X$ such that $\nu$ is a refinement of $\mu$.
	Then there is a simplicial collapse $\Alpha_{\infty}(X,\nu) \searrow \Alpha_{\infty}(X,\mu)$.
\end{theorem}

We will construct a generalized discrete Morse gradient that induces the collapse in the theorem.
We begin with a series of observations to reformulate the problem.
Since $X$ is in general position, we have isomorphisms $\Alpha_{\infty}(X,\mu) \cong \Del(X,\mu)$ and $\Alpha_{\infty}(X,\nu) \cong \Del(X,\nu)$ as abstract simplicial complexes.
By \Cref{rem: functoriality of chromatic Delaunay triangulations} there is an embedding $\iota : \Del(X,\mu) \hookrightarrow \Del(X,\nu)$.
Therefore, it suffices to exhibit a simplicial collapse $\Del(X,\nu) \searrow \iota(\Del(X,\mu))$.

Next, we may assume that $\nu$ is an \textit{elementary refinement} of $\mu$, by which we mean that $\nu$ is obtained from $\mu$ by splitting up some class.
More precisely if $\mu = (X_0, \ldots, X_s)$ and $\nu = (X'_0, \ldots, X'_{s'})$ then we may assume that $\nu \preceq \mu$ and $s' = s + 1$.
This is because an arbitrary refinement $\nu \preceq \mu$ can be written as a sequence of elementary refinements $\nu = \mu_k \preceq \ldots \preceq \mu_0 = \mu$, whereby there is a sequence of embeddings $\Del(X,\mu) = \Del(X,\mu_0) \hookrightarrow \ldots \hookrightarrow \Del(X,\mu_k) = \Del(X,\nu)$.
Then the result for elementary refinements applied to this sequence yields the general result.

Since the chromatic alpha filtration is invariant to a permutation of class labels, we may also assume that $\nu$ is obtained from $\mu$ by splitting up specifically the last class.
In other words, $X_i = X_i'$ for $i < s$ and $X_s = X'_s \cup X'_{s+1}$.
Then we are in the situation of \Cref{rem: description of membrane}, wherein the chromatic lift $X^{\nu} = (X^{\mu})^{\zeta}$ for some colouring $\zeta:X \to \{0, 1\}$.
Using the definition of the chromatic Delaunay triangulation gives
\begin{equation}
	\Del(X,\nu) = \Del(X^{\nu}) = \Del((X^{\mu})^{\zeta}) = \Del(X^{\mu},\zeta) \quad \text{and} \quad \Del(X,\mu) = \Del(X^{\mu}).
\end{equation}
Therefore, we must show that $\Del(X^{\mu},\zeta) \searrow \iota(\Del(X^{\mu}))$, for which it suffices to show that $\Del(X,\mu) \searrow \iota(\Del(X))$ whenever $\mu$ is a bi-colouring of $X$.
Put together, these arguments show that \Cref{thm: chromatic_collapse} is a corollary of the following special case:

\begin{theorem}\label{thm: bichrom_triangulation_collapse}
	Suppose $X \subset \RR^d$ is in general position and $\mu : X \to \{0,1\}$ is a bi-colouring of $X$.
	Then there is a simplicial collapse $\Del(X,\mu) \searrow \iota(\Del(X))$, where $\iota : \Del(X) \hookrightarrow \Del(X,\mu)$ is the embedding described in \Cref{rem: functoriality of chromatic Delaunay triangulations}.
\end{theorem}

The rest of the section is devoted to proving \Cref{thm: bichrom_triangulation_collapse}.
We will assume that $\mu$ is a fixed bi-colouring of $X$.

\subsection{Construction of Generalized Vector Field Inducing the Collapse}\label{sec: construction of vertical gradient}
By virtue of the \fullref{thm: generalized gradient collapsing theorem}, to prove \Cref{thm: bichrom_triangulation_collapse} it suffices to construct a generalized discrete Morse gradient on $\Del(X,\mu)$ whose critical set is $\Del(X)$.
If $|X| \leq d$ then there is nothing to show since $\Del(X,\mu) \cong \Delta^{|X|-1} \cong \Del(X)$ by \Cref{thm: dimension of chromatic lift}.
Therefore, we assume without loss of generality that $|X| \geq d+1$.
Then the intuition behind our approach is given in \Cref{fig: 1d example for vertical collapse}.

\begin{figure}
	\centering
	\begin{tikzpicture}[
		scale=0.5
	]
	\coordinate (B1) at (-4, 0);
	\coordinate (B2) at (0, 0);
	\coordinate (B3) at (3, 0);
	\coordinate (O1) at (-6, 4);
	\coordinate (O2) at (-2, 4);
	\coordinate (O3) at (6, 4);
	\coordinate (O4) at (9, 4);

	\fill[GShape] (O1.center) -- (B1.center) -- (O2.center) -- cycle;
	\fill[GShape] (B1.center) -- (O2.center) -- (B2.center) -- cycle;
	\fill[GShape] (O2.center) -- (B2.center) -- (B3.center) -- cycle;
	\fill[GShape] (B3.center) -- (O3.center) -- (O2.center) -- cycle;
	\fill[GShape] (B3.center) -- (O3.center) -- (O4.center) -- cycle;
	
	\draw[OLine] (O1) -- (O2) -- (O3) -- (O4);
	\draw[BLine] (B1) -- (B2) -- (B3);
	\draw[GLine] (O1) -- (B1) -- (O2) -- (B2);
	\draw[GLine] (O2) -- (B3) -- (O3);
	\draw[GLine] (B3) -- (O4);

	\draw[dashed,thick] (O1.center) -- (B1.center) -- (O2.center) -- (B2.center) -- (B3.center) -- (O3.center) -- (O4.center);

	\node[BPoint] at (B1) {};
	\node[BPoint] at (B2) {};
	\node[BPoint] at (B3) {};
	\node[OPoint] at (O1) {};
	\node[OPoint] at (O2) {};
	\node[OPoint] at (O3) {};
	\node[OPoint] at (O4) {};

	\coordinate (M1) at ($(O2)!(B2)!(B3)$);
	\coordinate (M2) at ($(B3)!(O3)!(O4)$);
	\draw[->] ($(O1)!0.5!(O2)$) -- ++(0, -1);
	\draw[->] ($(O2)!0.5!(O3)$) -- ++(0, -1);
	\draw[->] ($(B1)!0.5!(B2)$) -- ++(0, 1);
	\draw[->] (M1) -- ($(M1)!0.5!(B2)$);
	\draw[->] (M2) -- ($(M2)!0.5!(O3)$);
\end{tikzpicture}
	\caption{
		Chromatic Delaunay triangulation $\Del(X,\mu) \subset \RR^2$ for a bi-chromatic set in $\RR$.
		The dashed line is the embedded image $\iota(\Del(X))$ of the Delaunay triangulation, and the arrows indicate the discrete Morse pairings inducing the collapse $\Del(X,\mu) \searrow \iota(\Del(X))$.
	}
	\label{fig: 1d example for vertical collapse}
\end{figure}

In that figure, the 2-simplices can be partitioned into whether they lie above or below $\iota(\Del(X))$, which is merely a curve in $\RR^2$.
If a 2-simplex lies above $\iota(\Del(X))$, we pair it with its uppermost face, and if it lies below $\iota(\Del(X))$ we pair it with its lowermost face.
The set of pairings constitutes a discrete Morse gradient inducing the desired collapse.
For $d \geq 2$, the uppermost or lowermost face of a $(d+1)$-simplex $\sigma$ might not be $d$-dimensional, and instead of pairings we will consider arbitrary intervals in the face poset of $\Del(X,\mu)$.
For example, in \Cref{fig: membrane example}, the intervals we will consider are $[v_0v_1, v_0v_1v_2v_4]$, $[v_0v_3v_4, v_0v_2v_3v_4]$, $[v_1v_2, v_1v_2v_3v_5]$, $[v_1v_2v_4, v_1v_2v_4v_5]$, $[v_2v_5, v_2v_3v_4v_5]$.

By \Cref{rem: description of membrane}, the embedded image $\iota(\Del(X))$ is the graph of some function $h: \Del(X) \to \RR$.
Let $\Ucal$ and $\Lcal$ be the epigraph and hypograph of $h$ respectively, i.e., points above and below the graph of $h$:
\begin{align}
	\Ucal & \defeq \{ (x, y) \in \Conv(X) \times \RR \mid y \geq h(x) \} \subset \RR^{d+1}, \\
	\Lcal & \defeq \{ (x, y) \in \Conv(X) \times \RR \mid y \leq h(x) \} \subset \RR^{d+1}.
\end{align}
We say that $\sigma \in \Del(X,\mu)$ \textit{lies above} (resp. \textit{below}) $\iota(\Del(X))$ in $\Del(X,\mu)$ if $\sigma$ intersects the interior of $\Ucal$ (resp. $\Lcal$).
Since $\iota(\Del(X))$ is $d$-dimensional, each $(d+1)$-simplex of $\Del(X,\mu)$ intersects it in a face of dimension at most $d$, from which it follows that every $(d+1)$-simplex of $\Del(X,\mu)$ is either above or below $\iota(\Del(X))$ but not both.
We wish to have a notion of upper and lower faces for each $(d+1)$-simplex of $\Del(X,\mu)$, for which we need the following lemma.

\begin{lemma}\label{thm: chromatic delaunay has no vertical faces}
	The $(d+1)$\textsuperscript{th} basis vector $e_{d+1} \in \RR^{d+1}$ is transverse to $\Del^\mu(X)$.
\end{lemma}
\begin{proof}
	Suppose $\tau$ is a simplex in $\Del(X,\mu)$ that is not transverse to $e_{d+1}$.
	Without loss of generality we can assume that $\tau$ is $d$-dimensional.
	Let $\pi: \RR^{d+1} \to \RR^d$ be the projection to the first $d$ coordinates.
	Then $\pi(\Aff(\tau))$ is a $(d-1)$-dimensional subspace of $\RR^d$, which implies that the vertices of $\tau$ are an affinely dependent subset of $X$.
	This contradicts the fact that $X$ is in general position.
\end{proof}

By virtue of \Cref{thm: chromatic delaunay has no vertical faces}, for any $\sigma^{(d+1)} \in \Del(X,\mu)$ there is a well-defined notion of upper and lower faces of $\sigma$ with respect to $e_{d+1}$ (see \Cref{sec: geometric simplicial complexes} for our definition of upper and lower faces).
Now we can define our generalized discrete Morse gradient.
Recall that $\sigma^*$ and $\sigma_*$ are the minimal upper and lower faces of $\sigma$ respectively.
We define $\Sigma'$ and $\Sigma$ to be the following collection of intervals:
\begin{align}
	\Sigma' & = \{[\sigma^*, \sigma] \mid \sigma^{(d+1)} \in \Del(X,\mu), \sigma \subset \Ucal\} \cup \{[\sigma_*, \sigma] \mid \sigma^{(d+1)} \in \Del(X,\mu), \sigma \subset \Lcal\}, \\
	\Sigma  & = \Sigma' \cup \iota(\Del(X)). \label{eq: definition of vertical gradient}
\end{align}
In other words, whenever $\sigma^{(d+1)} \in \Del(X,\mu)$ lies above (resp. below) $\iota(\Del(X))$ then there is a corresponding interval in $\Sigma$ comprising all the upper (resp. lower) faces of $\sigma$.

\begin{proposition}\label{thm: description of bichromatic gradient}
	$\Sigma$ is a generalized discrete Morse vector field on $\Del(X,\mu)$ whose set of singleton intervals is $\iota(\Del(X))$.
\end{proposition}

Before proving \Cref{thm: description of bichromatic gradient} we need a technical lemma.

\begin{lemma}\label{thm: faces above membrane are upper for something}
	If $\tau^{(k)} \in \Del(X,\mu)$ lies above (resp. below) $\iota(\Del(X))$ with $k \leq d$ then there is a unique $\sigma^{(d+1)} \in \Del(X,\mu)$ such that $\tau$ is an upper (resp. lower) face for $\sigma$ with respect to $e_{d+1}$ and $\sigma$ lies above (resp. below) $\iota(\Del(X))$.
\end{lemma}
\begin{proof}
	We only consider the case where $\tau$ lies above $\iota(\Del(X))$; the proof for the other case is similar.
	Let $\pi : \RR^{d+1} \to \RR^d$ be the projection to the first $d$ coordinates.
	Consider any point $x \in \tau \cap \topint{\Ucal}$, and let $y = \iota(\pi(x))$.
	Note $x \neq y$.
	The line segment $[x, y]$ lies in $\Del(X,\mu)$ by convexity of the latter.
	By \Cref{thm: chromatic delaunay has no vertical faces} we know that $e_{d+1}$ is transverse to $\Del(X,\mu)$, so no non-degenerate subsegment of $[x, y]$ can lie in a $d$-simplex.
	In particular the subsegment $[x, x-\lambda e_{d+1}]$ must lie in some $\sigma^{(d+1)}$ for $\lambda > 0$ small enough.
	Then by \Cref{thm: characterisation of upper and lower faces} $\tau$ is an upper face for $\sigma$.
	Moreover, $\sigma$ lies above $\iota(\Del(X))$ because it contains the point $x$ which is in $\topint{\Ucal}$.
\end{proof}

\begin{proof}[Proof of \Cref{thm: description of bichromatic gradient}]
	We check that (1) the intervals in $\Sigma'$ are contained in $\Del(X,\mu) \setminus \iota(\Del(X))$, (2) the intervals are disjoint, and (3) their union is all of $\Del(X,\mu) \setminus \iota(\Del(X))$.

	\begin{enumerate}
		\item If $[\sigma^*, \sigma]$ is an interval in $\Sigma'$, then by construction $\sigma$ lies above $\iota(\Del(X))$.
		      Any simplex $\tau$ in $[\sigma^*, \sigma]$ is an upper face of $\sigma$, and by \Cref{thm: characterisation of upper and lower faces} the relative interior of any upper face of $\sigma$ lies strictly in the interior of $\Ucal$.
		      Therefore, any $\tau \in [\sigma^*, \sigma]$ is not contained in $\iota(\Del(X))$.
		      A similar argument holds for when $\sigma$ lies below $\iota(\Del(X))$.

		\item Suppose $\sigma_1^{(d+1)}$ and $\sigma_2^{(d+1)}$ are distinct simplices in $\Del(X,\mu)$.
		      We must show that the corresponding intervals in $\Sigma'$ are disjoint.
		      First suppose $\sigma_1$ and $\sigma_2$ lie above and below $\iota(\Del(X))$ respectively, so that $[\sigma_1^*, \sigma_1]$ and $[(\sigma_2)_*, \sigma_2]$ are the corresponding intervals in $\Sigma$.
		      Then $\sigma_1 \cap \sigma_2$ is either empty, in which case the intervals are disjoint and there is nothing left to show, or it must be a simplex in $\iota(\Del(X))$.
		      In the latter case, $\sigma_1 \cap \sigma_2$ cannot belong to the intervals corresponding to $\sigma_1$ and $\sigma_2$ due to part (1).

		      Otherwise, $\sigma_1$ and $\sigma_2$ both lie above $\iota(\Del(X))$ (without loss of generality).
		      If $\tau \in [\sigma_1^*, \sigma_1] \cap [\sigma_2^*, \sigma_2]$ then $\tau$ is an upper face for both $\sigma_1$ and $\sigma_2$.
		      But this is impossible due to \Cref{thm: characterisation of upper and lower faces} and the fact that $\sigma_1$ and $\sigma_2$ have disjoint interiors.

		\item Let $\tau$ be any simplex in $\Del(X,\mu)\setminus\iota(\Del(X))$ and without loss of generality assume that $\tau$ lies above $\iota(\Del(X))$.
		      We will show that $\tau$ is contained in some interval of $\Sigma'$.
		      If $\dim(\tau) = d+1$ then $\tau \in [\tau^*, \tau] \in \Sigma'$ and we are done.
		      Therefore, suppose $\dim(\tau) \leq d$.
		      Then \Cref{thm: faces above membrane are upper for something} implies that there exists some $\sigma^{(d+1)}$ that lies above $\iota(\Del(X))$ such that $\tau$ is an upper face for $\sigma$.
		      Hence, $\tau \in [\sigma^*, \sigma] \in \Sigma'$.\qedhere
	\end{enumerate}
\end{proof}

\subsection{Proof that the Generalized Vector Field is a Generalized Gradient}\label{sec: construction of vertical collapse Morse function}

The next step is to check that $\Sigma$ is a gradient.
Using the notation from \Cref{sec: discrete Morse theory} and the statement of the \fullref{thm: generalized gradient acyclicity lemma}, it suffices to check that $\Hcal(\Del(X,\mu))/\Sigma$ is a directed acyclic graph, where $\Hcal(\Del(X,\mu))$ is the Hasse diagram of the co-face poset of $\Del(X,\mu)$.
We will do this by constructing a strictly monotonic real-valued function on $\Hcal(\Del(X,\mu))/\Sigma$.
Indeed, any such function can be lifted to a generalized discrete Morse function on $\Del(X,\mu)$ whose gradient is $\Sigma$.
Intuitively, our function will describe the ``distance'' of each simplex from $\iota(\Del(X))$ in terms of a height function.
For the rest of this section, let $h:\RR^{d+1} \to \RR$ denote the height function along the $(d+1)$\textsuperscript{st} coordinate.

A vertex in $\Hcal(\Del(X,\mu))/\Sigma$ can either correspond to a simplex $\tau$ in $\iota(\Del(X))$, or to an interval of the form $[\sigma^*, \sigma]$ or $[\sigma_*, \sigma]$ in $\Sigma'$.
We represent the latter kind of vertex by $[\sigma]$, where $\sigma$ is the $(d+1)$-dimensional representative of the interval.
We will also write $C(\sigma)$ for the unique circumcentre of $\sigma$.
Define $f: \Hcal(\Del(X,\mu))/\Sigma \to \RR$ as follows
\begin{equation}
	f(x) = \begin{cases}
		\dim(\tau)                                    & x = \tau \in \iota(\Del(X)), \\
		\left(h(C(\sigma)) - \frac{1}{2}\right)^2 + d & x = [\sigma].
	\end{cases}
	\label{eq: definition of vertical collapse Morse function}
\end{equation}
In the rest of this section we will prove that $f$ is strictly monotonic.
We first make some observations about $\Hcal(\Del(X,\mu))/\Sigma$.
There are no edges in $\Hcal(\Del(X,\mu))/\Sigma$ that point from vertices of $\Hcal(\Del(X))$ into vertices of the form $[\sigma]$ for $\sigma^{(d+1)} \in \Del(X,\mu)$.
This is because $\Del(X)$ is a subcomplex of $\Del(X,\mu)$.
There are also no edges of the form $[\sigma] \to [\sigma']$, if $\sigma$ and $\sigma'$ lie above and below $\iota(\Del(X))$ respectively (or vice-versa).
This is because such an edge, by definition of the vector field $\Sigma$, would imply the existence of a shared face $\tau$ that is an upper face of $\sigma$ and a lower face of $\sigma'$.
By construction of $\Ucal$ and $\Lcal$ this would imply that $\sigma \subset \Lcal$ and $\sigma' \subset \Ucal$, a contradiction.
Therefore, we have the following possibilities for the types of edges in $\Hcal(\Del(X,\mu))/\Sigma$:

\begin{enumerate}
	\item $\eta \to \tau$ where $\eta, \tau \in\iota(\Del(X))$ and $\eta > \tau$.
	\item $[\sigma] \to \tau$ where $\sigma^{(d+1)} \in \Del(X,\mu)$ and $\tau \in \iota(\Del(X))$, with $\sigma > \tau$.
	\item $[\sigma] \to [\sigma']$ where $\sigma, \sigma'$ are $(d+1)$-simplices in $\Del(X,\mu)$ and both lie above or below $\iota(\Del(X))$.
\end{enumerate}

\begin{remark}\label{rem: monotonicity of f across type-1 edges}
	Clearly $f(\eta) > f(\tau)$ whenever $\eta \to \tau$ is an edge of type (1) in $\Hcal(\Del(X,\mu))/\Sigma$.
\end{remark}
Next we consider edges of type (2) in $\Hcal(\Del(X,\mu))$.
If $[\sigma] \to \tau$ is such an edge, then $f([\sigma]) > f(\tau)$ holds automatically if $\dim(\tau) < d$.

\begin{lemma}\label{thm: membrane simplices have a circumsphere equidistant from the two planes}
	Let $\tau \in \Del(X,\mu)$.
	Then $\tau \in \iota(\Del(X))$ if and only if there is an empty circumsphere for $\tau$ in $\RR^{d+1}$ whose centre $C$ satisfies $h(C) = \frac{1}{2}$.
\end{lemma}
\begin{proof}
	Let $V(\tau)\subset X$ be the set of points whose chromatic lift is the vertex set of $\tau$, i.e., $V(\tau)^{\mu} = \tau$.
	By \Cref{eq: alternative definition of chromatic alpha filtration} there exists an empty $[2]$-stack $S$ passing through $V(\tau)$ in $\RR^d$.
	By the correspondence between stacks and spheres, any $[2]$-stack $S$ passing through $V(\tau)$ lifts to an empty $d$-sphere $S^{\mu}$ passing through $\tau$ in $\RR^{d+1}$.
	Let $r_0(S)$ and $r_1(S)$ be the radii of the $(d-1)$-spheres in $S$ and let $C$ be the centre of $S^{\mu}$.
	Applying Pythagoras' theorem, we get $r_1(S)^2 - r_0(S)^2 = 2h(C) - 1$ (see \Cref{fig: Pythagoras theorem for stacks}).
	\Cref{eq: alternative definition of alpha filtration} implies that $\tau \in \iota(\Del(X))$ if and only if there is an empty stack $S$ passing through $V(\tau)$ with $r_0(S) = r_1(S)$, and by the previous remark this is equivalent to $h(C) = 1/2$.
\end{proof}

\begin{figure}
	\centering
	\begin{tikzpicture}[scale=0.7]
	\coordinate (P) at (0,2.5);
	\coordinate (Q) at (0,-1);
	\coordinate (Pr) at ($(P)+({sqrt(9-2.5^2)},0)$);
	\coordinate (Qr) at ($(Q)+({sqrt(8)},0)$);
	\coordinate (Pm) at ($(P)-(4, 0)$);
	\coordinate (Pp) at ($(P)+(4,0)$);
	\coordinate (P1) at ($(Pm)-(0.7,0.7)$);
	\coordinate (P2) at ($(Pm)+(0.7,0.7)$);
	\coordinate (P3) at ($(Pp)+(0.7,0.7)$);
	\coordinate (P4) at ($(Pp)-(0.7,0.7)$);
	\coordinate (Qm) at ($(Q)-(4, 0)$);
	\coordinate (Qp) at ($(Q)+(4,0)$);
	\coordinate (Q1) at ($(Qm)-(0.7,0.7)$);
	\coordinate (Q2) at ($(Qm)+(0.7,0.7)$);
	\coordinate (Q3) at ($(Qp)+(0.7,0.7)$);
	\coordinate (Q4) at ($(Qp)-(0.7,0.7)$);

	\draw (0,0) circle (3);
	\draw[OLine] (P) circle[x radius=0.98*sqrt(9-2.5^2), y radius=0.2];
	\draw[BLine] (Q) circle[x radius=0.99*sqrt(8), y radius=0.5];
	\fill[OShape] (P1)--(P2)--(P3)--(P4)--cycle;
	\fill[BShape] (Q1)--(Q2)--(Q3)--(Q4)--cycle;
	\draw[dashed] (P)--(Pr)   node[Label, above      = 3pt, midway] {$r_1$};
	\draw[dashed] (Q)--(Qr)   node[Label, below      = 3pt, midway] {$r_0$};
	\draw[dashed] (0,0)--(Q)  node[Label, left       = 3pt, midway] {$h$};
	\draw[dashed] (0,0)--(P)  node[Label, left       = 3pt, midway] {$|1-h|$};
	\draw[dashed] (0,0)--(Pr) node[Label, right      = 3pt, midway] {$R$};
	\draw[dashed] (0,0)--(Qr) node[Label, above      = 3pt, midway] {$R$};
	\node[Point] at (0,0) {}  node[Label, above left = 3pt of {(0,0)}] {$C$};
\end{tikzpicture}
	\caption{
		Visual proof of \Cref{thm: membrane simplices have a circumsphere equidistant from the two planes}.
		If $r_0=r_1$ then $h^2 = |1-h|^2 \implies h=\frac{1}{2}$.
	}
	\label{fig: Pythagoras theorem for stacks}
\end{figure}

\begin{lemma}\label{thm: height of circumcentres 1}
	Suppose $\sigma^{(d+1)} \subset \RR^{d+1}$ is a geometric simplex, $z \in \RR^{d+1}$ is transverse to $\sigma$, and $\tau^{(d)} < \sigma$.
	Let $C$ be the circumcentre of $\sigma$, let $C'$ be the centre of any circumsphere of $\tau$ that is empty of $\sigma$, and let $h_z: \RR^{d+1} \to \RR$ denote the height function along $z$.
	If $\tau$ is a lower (resp. upper) face of $\sigma$ with respect to $z$, then $h_z(C) \geq h_z(C')$ (resp. $h_z(C) \leq h_z(C')$), with equality if and only if $C = C'$.
\end{lemma}
\begin{proof}
	$C'$ must lie on the set $\ell$ of points equidistant from the vertices of $\tau$, and since $\tau$ is a $d$-simplex $\ell$ is a line orthogonal to $\Aff(\tau)$.
	The circumsphere of $\sigma$ is an empty circumsphere for $\tau$, so $C$ lies on $\ell$.
	Let $w$ be the unit outward normal to $\tau$ with respect to $\sigma$, so that $w$ is parallel to $\ell$.
	Then we can write $C' = C + tw$ for some $t \in \RR$.

	First we claim that $t \geq 0$.
	Let $v_0$ be the vertex of $\sigma$ opposite to $\tau$, and let $v_1$ be any vertex of $\tau$.
	We have
	\begin{align*}
		|C' - v_0|^2 - |C' - v_1|^2 & = \braket{2C' - v_0 - v_1, v_1 - v_0}                 \\
		                            & = \braket{2C + 2tw - v_0 - v_1, v_1 - v_0}            \\
		                            & = |C - v_0|^2 - |C - v_1|^2 + 2t\braket{w, v_1 - v_0} \\
		                            & = 2t\braket{w, v_1 - v_0},
	\end{align*}
	where the fourth equality is because $C$ is a circumsphere for $\sigma$.
	By the definition of outward normal, there exists $x \in \Aff(\tau)$ and $\lambda > 0$ such that $v_0 = x - \lambda w$.
	Then
	\begin{equation*}
		\braket{w, v_1 - v_0} = \braket{w, v_1 - x + \lambda w} = \lambda
	\end{equation*}
	since $v_1 - x \in \Lin(\tau)$ and $w \perp \Lin(\tau)$.
	Thus, $|C' - v_0|^2 - |C' - v_1|^2 = 2t\lambda$.
	This shows that if $t < 0$ then the circumsphere through $\tau$ centred at $C'$ contains $v_0$ and cannot be empty.

	Now we have
	\begin{equation*}
		h_z(C) - h_z(C') = \braket{z, C - (C + tw)} = -t\braket{z, w}.
	\end{equation*}
	Without loss of generality assume that $\tau$ is a lower face of $\sigma$.
	Then by definition $\braket{z, w} < 0$, and the quantity above is non-negative with equality if and only if $t = 0$, i.e., $C = C'$.
\end{proof}

\begin{lemma}\label{thm: height of circumcentres 2}
	Let $Y \subset \RR^{d+1}$ be in general position, let $z$ be transverse to $\Del(Y)$, and let $\sigma_1^{(d+1)}, \sigma_2^{(d+1)} \in \Del(Y)$ intersect in an upper $d$-dimensional face of $\sigma_2$ with respect to $z$.
	Let $C_i$ be the circumcentre of $\sigma_i$, and let $h_z : \RR^{d+1} \to \RR$ denote the height function along the $z$ direction.
	Then $h_z(C_1) > h_z(C_2)$.
\end{lemma}
\begin{proof}
	Since $z$ is transverse to $\sigma_1$, the face $\tau = \sigma_1 \cap \sigma_2$ must be either a lower or upper face of $\sigma_1$.
	But we are given that $\tau$ is an upper face of $\sigma_2$, and thus by \Cref{thm: characterisation of upper and lower faces} it cannot be an upper face of $\sigma_1$.
	Thus, $\tau$ is a lower face of $\sigma_1$.
	The circumcentre $C_2$ is the centre of a circumsphere for $\tau$ that is empty of $\sigma_1$ because of the Delaunay condition.
	Moreover, since $Y$ is in general position, we must have $C_2 \neq C_1$.
	Then applying \Cref{thm: height of circumcentres 1} to $\sigma_1$ and $\tau$ gives $h(C_1) > h(C_2)$ as desired.
\end{proof}

\begin{figure}
	\centering
	\begin{tikzpicture}[scale=0.5]
	\coordinate (A) at (-2, -2);
	\coordinate (B) at (2, -2);
	\coordinate (C) at (0, 4);
	\coordinate (D) at (4, 4);
	\coordinate (Z) at (7, 0);
	\draw[GShape, GLine] (A) -- (B) -- (C) -- cycle;
	\draw[GShape, GLine] (D) -- (B) -- (C) -- cycle;
	\node[circle through 3 points={B}{C}{D},draw=black] (P) {};
	\node[circle through 3 points={A}{C}{B},draw=black] (O) {};
	\node[Point] at (O.center) {} node[Label, above = 3pt of O.center] {$O$};
	\node[Point] at (P.center) {} node[Label, above = 3pt of P.center] {$P$};
	\draw[densely dotted, black] (O.center) -- (P.center);

	\draw[->] (Z) -- ++(0,3) node[Label, midway, right=3pt] {$z$}; 
	\node[Point] at (A) {} node[Label, below left  = 3pt of A] {$a$};
	\node[Point] at (B) {} node[Label, below right = 3pt of B] {$b$};
	\node[Point] at (C) {} node[Label, above       = 3pt of C] {$c$};
	\node[Point] at (D) {} node[Label, above right = 3pt of D] {$d$};
\end{tikzpicture}
	\caption{
		Illustration of \Cref{thm: height of circumcentres 2}.
		Here $C_1 = P$ and $C_2 = O$ are the circumcentres of $\sigma_1 = bcd$ and $\sigma_2 = abc$ respectively, and $bc$ is an upper face for $abc$.
	}
	\label{fig: non-Delaunay counterexample}
\end{figure}

\begin{lemma}\label{thm: height of circumcentres of simplices above the membrane}
	Suppose $\sigma^{(d+1)} \in \Del(X,\mu)$ and $C$ is the circumcentre of $\sigma$.
	Then $\sigma$ lies above (resp. below) $\iota(\Del(X))$ if and only if $h(C) > \frac{1}{2}$ (resp. $h(C) < \frac{1}{2}$).
\end{lemma}
\begin{proof}
	It suffices to show that if $\sigma$ lies above $\iota(\Del(X))$ then $h(C) > \frac{1}{2}$ and if $\sigma$ lies below $\iota(\Del(X))$ then $h(C) < \frac{1}{2}$.
	The argument for both cases is similar, so we prove only the former case.
	Hence, suppose $\sigma$ lies above $\iota(\Del(X))$.

	First suppose $\sigma$ has a face $\tau^{(d)}$ that lies in $\iota(\Del(X))$.
	Since $\tau$ is in $\iota(\Del(X))$, \Cref{thm: membrane simplices have a circumsphere equidistant from the two planes} shows that there is an empty circumsphere for $\tau$ with centre $C'$ such that $h(C') = \frac{1}{2}$.
	The proof of \Cref{thm: description of bichromatic gradient} shows that any upper face of $\sigma$ lies in $\Del(X,\mu) \setminus \iota(\Del(X))$, so $\tau$ is necessarily a lower face of $\sigma$.
	Then $h(C) \geq h(C') = \frac{1}{2}$ by \Cref{thm: height of circumcentres 1}, with equality if and only if $C = C'$.
	Now $\sigma$, being $(d+1)$-dimensional, is not in $\iota(\Del(X))$.
	By \Cref{thm: membrane simplices have a circumsphere equidistant from the two planes}, we infer that $C \neq C'$ and therefore the inequality above is in fact a strict inequality: $h(C) > \frac{1}{2}$.

	Now for general $\sigma$: we construct a finite sequence of $(d+1)$-dimensional simplices $\sigma = \sigma_0, \sigma_1, \ldots, \sigma_k$ such that $\sigma_i \cap \sigma_{i+1}$ is a $d$-dimensional upper face of $\sigma_{i+1}$ for each $i < k$, and $\sigma_k$ has a lower $d$-dimensional face in $\iota(\Del(X))$.
	Suppose $\sigma_i$ has been chosen.
	If $\sigma_i$ has a lower $d$-dimensional face on $\iota(\Del(X))$ set $k = i$.
	Otherwise, let $\tau_i$ be any lower $d$-dimensional face of $\sigma_i$, so $\tau_i$ must lie above $\iota(\Del(X))$.
	By \Cref{thm: faces above membrane are upper for something} there exists a unique $\sigma_{i+1}$ that lies above $\iota(\Del(X))$ for which $\tau$ is an upper face.
	This sequence is non-repeating because, by \Cref{thm: height of circumcentres 2}, we have $h(C(\sigma_0)) > \ldots > h(C(\sigma_k))$.
	The sequence terminates because it is non-repeating and because there are a finite number of simplices in $\Del(X,\mu)$.
	Having constructed our sequence, we note that $h(C(\sigma_k)) > \frac{1}{2}$ by our previous argument, and the result follows since $h(C(\sigma_0)) > \ldots > h(C(\sigma_k))$.
\end{proof}

From \Cref{thm: height of circumcentres of simplices above the membrane} and the definition of the function $f$ in \Cref{eq: definition of vertical collapse Morse function} we immediately get the following.
\begin{lemma}\label{thm: monotonicity of f across type-2 edges}
	If $f$ is the function defined in \Cref{eq: definition of vertical collapse Morse function} then $f([\sigma]) > f(\tau)$ whenever $[\sigma] \to \tau$ is an edge of type (2) in $\Hcal(\Del(X,\mu))/\Sigma$.
\end{lemma}

Finally, we consider edges of type (3) in $\Hcal(\Del(X,\mu))$.

\begin{lemma}\label{thm: sufficient condition for edge in Hasse diagram}
	Suppose $\sigma, \sigma'$ are $(d+1)$-simplices in $\Del(X,\mu)$ that lie above (resp. below) $\iota(\Del(X))$.
	\begin{enumerate}
		\item If there is an edge $[\sigma] \to [\sigma']$ in $\Hcal(\Del(X,\mu))/\Sigma$, then $\sigma \cap \sigma'$ is an upper (resp. lower) face of $\sigma'$.
		\item If $\sigma \cap \sigma'$ is an upper (resp. lower) $d$-dimensional face of $\sigma'$ then there is an edge $[\sigma] \to [\sigma']$ in $\Hcal(\Del(X,\mu))/\Sigma$.
	\end{enumerate}
\end{lemma}
\begin{proof}
	We consider the case where $\sigma$ and $\sigma'$ lie above $\iota(\Del(X))$; the proof for the other case is similar.
	\begin{enumerate}
		\item Suppose there is an edge $[\sigma] \to [\sigma']$ in $\Hcal(\Del(X,\mu))/\Sigma$.
		      Then there are upper faces $\tau \in [\sigma^*, \sigma]$ and $\tau'\in [(\sigma')^*, \sigma']$ such that there is an edge $\tau \to \tau'$ in $\Hcal(\Del(X,\mu))$, i.e., $\tau \geq \tau'$.
		      Then $\tau'$ must be a face of $\sigma \cap \sigma'$, and the latter must therefore also belong to the interval $[(\sigma')^*, \sigma']$.
		\item Suppose $\sigma \cap \sigma'$ is an upper $d$-dimensional face of $\sigma'$.
		      Then the edge $\sigma \to \sigma \cap \sigma'$ in $\Hcal(\Del(X,\mu))$ descends to an edge $[\sigma] \to [\sigma']$ in $\Hcal(\Del(X,\mu))/\Sigma$.\qedhere
	\end{enumerate}
\end{proof}

\begin{lemma}\label{thm: paths in quotient graph lift to codimension 1 paths}
	Suppose $\sigma, \sigma'$ are $(d+1)$-simplices in $\Del(X,\mu)$ such that $\sigma$ and $\sigma'$ both lie above (resp. below) $\iota(\Del(X))$ and $\sigma \cap \sigma'$ is an upper (resp. lower) face of $\sigma'$.
	Then there is a path of $(d+1)$-simplices $[\sigma] = [\sigma_1] \to [\sigma_2] \to \ldots \to [\sigma_k] = [\sigma']$ in $\Hcal(\Del(X,\mu))/\Sigma$ such that $\sigma_i \cap \sigma_{i+1}$ is $d$-dimensional for each $1 \leq i < k-1$.
\end{lemma}
\begin{proof}
	We consider the case when $\sigma, \sigma'$ lie above $\iota(\Del(X))$.
	Consider the star $V \defeq \st(\sigma\cap\sigma')$ in $\Hcal(\Del(X,\mu))/\Sigma$.
	We claim:
	\begin{equation}\locallabel{claim 1}
		\text{$\sigma'$ is the unique $(d+1)$-simplex in $V$ that has no lower face in $V$}.\tag{$*$}
	\end{equation}
	Let $\mu$ be any $(d+1)$-simplex in $V$ with no lower face in $V$.
	This is equivalent to the property that every $d$-dimensional face of $\mu$ that contains $\sigma \cap \sigma'$ is an upper face of $\mu$, which is equivalent to the property that $\sigma \cap \sigma'$ is an upper face of $\mu$.
	This is equivalent to $\mu = \sigma'$ because \Cref{thm: characterisation of upper and lower faces} implies that there is a unique $\mu$ for which $\sigma \cap \sigma'$ is an upper face.

	We now construct the required path inductively.
	Begin by setting $\sigma_1 = \sigma$ and assume by induction that $\sigma_i \in V$.
	If $\sigma_i = \sigma'$ then we are done; otherwise by our previous remarks we can choose a lower $d$-dimensional face $\tau$ of $\sigma_i$ that is in $V$.
	Note that $\tau $ cannot be in $\iota(\Del(X))$.
	This is because $\tau > \sigma \cap \sigma'$, and $\sigma \cap \sigma'$ cannot be in $\iota(\Del(X))$ due to \Cref{thm: description of bichromatic gradient} and the fact that $\sigma \cap \sigma'$ is an upper face of $\sigma'$ which lies above $\iota(\Del(X))$.
	By \Cref{thm: faces above membrane are upper for something} there exists a unique $(d+1)$-dimensional simplex $\sigma_{i+1}$ in $\Hcal(\Del(X,\mu))/\Sigma$ for which $\tau$ is an upper face.
	Moreover, $\sigma_{i+1} \in V$ because $\sigma_{i+1} > \tau > \sigma \cap \sigma'$.
	Since $\tau$ is an upper face of $\sigma_{i+1}$ and $\tau = \sigma_i \cap \sigma_{i+1}$, there is an edge $[\sigma_i] \to [\sigma_{i+1}]$ in $\Hcal(\Del(X,\mu))/\Sigma$ by \Cref{thm: sufficient condition for edge in Hasse diagram}.
	The sequence $(\sigma_i)_{i \geq 1}$ is non-repeating because the sequences of heights of the circumcentres $h(C(\sigma_i))$ is strictly decreasing by \Cref{thm: height of circumcentres 2}.
	The sequence must terminate in a finite number of steps since it is non-repeating and $V$ is finite.
	Furthermore, it must terminate at $\sigma'$ by \localref{claim 1}.
\end{proof}

\begin{lemma}\label{thm: monotonicity of f across type 3 edges}
	The function $f$ defined in \Cref{eq: definition of vertical collapse Morse function} satisfies $f([\sigma]) > f([\sigma'])$ whenever $[\sigma] \to [\sigma']$ is an edge of type (3) in $\Hcal(\Del(X,\mu))$.
\end{lemma}
\begin{proof}
	Let $[\sigma] \to [\sigma']$ be an edge of type (3) in $\Hcal(\Del(X,\mu))$ and without loss of generality assume $\sigma$ and $\sigma'$ lie above $\iota(\Del(X))$.
	By \Cref{thm: paths in quotient graph lift to codimension 1 paths} there is a path $[\sigma] = [\sigma_0] \to [\sigma_1] \to \ldots \to [\sigma_k] = [\sigma']$ where each $\sigma_i \cap \sigma_{i+1}$ is a $d$-dimensional upper face for $\sigma_{i+1}$.
	By \Cref{thm: height of circumcentres 2} the sequence $h(C(\sigma_i))$ is strictly decreasing, and hence $(h(C(\sigma_i)) - \frac{1}{2})^2 + d$ is also strictly decreasing.
	It follows that $f([\sigma]) > f([\sigma'])$.
\end{proof}

\begin{proposition}\label{thm: acyclicity function is Morse}
	The function $f$ defined in \Cref{eq: definition of vertical collapse Morse function} lifts to a generalized discrete Morse function on $\Del(X,\mu)$, whose generalized discrete Morse gradient is $\Sigma$ as defined in \Cref{eq: definition of vertical gradient}.
\end{proposition}
\begin{proof}
	\Cref{rem: monotonicity of f across type-1 edges,thm: monotonicity of f across type-2 edges,thm: monotonicity of f across type 3 edges} imply that the function $f$ defined in \Cref{eq: definition of vertical collapse Morse function} is strictly monotonic on $\Hcal(\Del(X,\mu))/\Sigma$.
	This means that $\Hcal(\Del(X,\mu))/\Sigma$ is a directed acyclic graph, and \Cref{thm: generalized gradient acyclicity lemma} implies that $\Sigma$ is a generalized discrete Morse gradient.
	Then it is clear that the lift of $f$ to $\Hcal(\Del(X,\mu))$ is still monotonic and its level sets are exactly $\Sigma$, which shows that the lift is a generalized discrete Morse function whose gradient is $\Sigma$.
\end{proof}
Now since $\Crit(\Sigma) = \Del(X)$ (\Cref{thm: description of bichromatic gradient}), \Cref{thm: bichrom_triangulation_collapse} follows by applying the \fullref{thm: generalized gradient collapsing theorem} to the Morse gradient $\Sigma$.
Finally, \Cref{thm: chromatic_collapse} follows from the remarks at the beginning of \Cref{sec: construction of vertical gradient}.

\section{Morse Theoretic Structure of the Chromatic Alpha Filtration Function}\label{sec: KKT for stacks and pairing lemmas}

The main goal of this section is to study the \v{C}ech and the chromatic alpha filtration functions.
In \Cref{sec: convex optimization for stacks} we show that these functions are generalized discrete Morse functions.
This fact was already proved in~\cite{attali_vietorisrips_2013} for the \v{C}ech filtration and in~\cite{Montesano2025chromatic} for the chromatic alpha filtration, as noted in the introduction.
However, we will need to study the relationship between the generalized gradients of these filtration functions (\Cref{sec: pairing lemmas}), for which it is useful to specifically formulate a proof by adapting the methods in~\cite{bauer_morse_2016}, from the case of spheres to the case of stacks.
In that paper, the authors define a family of ``selective Delaunay filtrations'' that interpolate between the alpha and \v{C}ech filtrations, and prove that the corresponding filtration functions all share similar properties.
We will generalize their construction to interpolate between the chromatic alpha and \v{C}ech filtrations.
While in~\cite{bauer_morse_2016} the intermediate filtrations are interesting objects in their own right, we use the intermediate filtrations only as a tool for the proofs of this section.
Throughout the section $X \subset \RR^d$ will denote a fixed point cloud in general position and $\mu: X \to [s+1]$ will denote a fixed colouring of $X$.

\subsection{Showing that the Filtration Function is Morse}\label{sec: convex optimization for stacks}

Recall that for $\gamma \subset [s+1]$ an empty $\gamma$-stack in $\RR^d$ comprises a set of concentric $(d-1)$-spheres $S = (S_m)_{m \in \gamma}$ such that $S_m$ is empty of $X_m$.
We now consider a relaxation of the emptiness condition.
For a $\gamma$-stack $S$ (not necessarily empty) we define:
\begin{align}
	\Incl(S) & \defeq \bigcup_{j \in \gamma} \{x \in X_j \mid \text {$x$ lies on or inside $S_j$}\},                 \\
	\Excl(S) & \defeq \bigcup_{j \in \gamma} \{x \in X_j \mid \text {$x$ lies on or outside $S_j$}\},                \\
	\On(S)   & \defeq \bigcup_{j \in \gamma} \{x \in X_j \mid \text {$x$ lies on $S_j$} \} = \Incl(S) \cap \Excl(S).
\end{align}
For $\sigma, E \subset X$, we say $S$ \textit{excludes $E$} if $E \cap \mu^{-1}(\gamma) \subset \Excl(S)$, and we say that $S$ \textit{includes} $\sigma$ if $\sigma \subset \Incl(S)$.
Denoting $\gamma = \mu(\sigma)$, we will let $S(\sigma,\mu; E)$ denote the $\gamma$-stack of minimum radius that includes $\sigma$ and excludes $E$, if such a stack exists.
Note that if $S(\sigma,\mu; E)$ exists, then it is unique because it is the minimizer of a strictly convex function (see the proof of \Cref{thm: KKT for stacks}).
With this notation, the empty stack of minimum radius passing through $\sigma$ is $S(\sigma,\mu; X)$.

By \Cref{eq: alternative definition of chromatic alpha filtration}, a simplex $\sigma$ belongs to the chromatic alpha complex $\Alpha_r(X,\mu)$ if and only if $S(\sigma,\mu; X)$ exists and $\Rad(S(\sigma,\mu; X)) \leq r$.
By analogy, for each $E \subset X$ we can define a filtered simplicial complex on the vertex set $X$ as follows:
\begin{equation}
	\Alpha_r(X,\mu; E) \defeq \{\sigma \subset X \mid S(\sigma,\mu; E) \text{ exists and } \Rad(S(\sigma,\mu; E)) \leq r\}.
	\label{eq: definition of selective chromatic Delaunay filtration}
\end{equation}

\begin{remark}\label{rem: special case of selective Delaunay filtrations}
	From \Cref{eq: alternative definition of Cech filtration,eq: alternative definition of alpha filtration} we observe that for any colouring $\mu$ we have $\Alpha_{\bullet}(X,\mu; \emptyset) = \Cech_{\bullet}(X)$ and $\Alpha_{\bullet}(X,\mu; X) = \Alpha_{\bullet}(X,\mu)$.
	In the monochromatic case, $\Alpha_{\bullet}(X,\mu; E)$ is the ``selective Delaunay filtration'' defined in~\cite{bauer_morse_2016}.
\end{remark}

Next we use the \fullref{thm: KKT conditions} to characterize the stacks $S(\sigma,\mu; E)$.
For $\gamma \subset [s+1]$ and a stack $S = (S_m)_{m \in \gamma}$, we let $\Out(S)$ be the set of ``outer'' colours of $S$, i.e.,
\begin{equation}
	\Out(S) \defeq \{m \in \gamma \mid \Rad(S_m) = \Rad(S)\}.
\end{equation}
\begin{proposition}[KKT Conditions for Stacks]\label{thm: KKT for stacks}
	Let $E, \sigma \subset X$ with $\mu(\sigma) = \gamma \subset [s+1]$, and let $S = (S_m)_{m \in \gamma}$ be a stack that includes $\sigma$ and excludes $E$.
	Suppose $S$ is centred on $y \in \RR^d$ and its spheres have radii $(r_m)_{m \in \gamma}$.
	Let $\gamma' = \Out(S)$.
	Then $S = S(\sigma,\mu; E)$ if and only if there exists a set of real numbers $\{\lambda_v\}_{v \in \On(S)}$ such that
	\begin{enumerate}
		\item $\sum_{v \in \On(S)} \lambda_v v = y$.
		\item $\{v \in \On(S) \mid \lambda_v > 0\} \subset \sigma$.
		\item $\{v \in \On(S) \mid \lambda_v < 0\} \subset E$.
		\item $\sum_{v \in \On(S_m)} \lambda_v = 0$ for all $m \notin \gamma'$.
		\item $\sum_{v \in \On(S_m)} \lambda_v \geq 0$ for all $m \in \gamma'$.
		\item $\sum_{m \in \gamma'} \sum_{v \in \On(S_m)} \lambda_v = 1$.
	\end{enumerate}
	In this case, the coefficients $\lambda_v$ are unique.
\end{proposition}
\begin{proof}
	To simplify notation assume that $\gamma = [\ell+1] = \{0,\ldots,\ell\}$, where $\ell\leq s$.
	First we will prove the forward direction, so suppose $S = S(\sigma,\mu; E)$.
	For $v \in \mu^{-1}(\gamma)$ define
	\[
		g_v: \RR^{d}\times\RR^{ \ell + 1} &\to \RR, & f : \RR^{d} \times \RR^{\ell+1} &\to \RR,\\
		(z, w_0, \ldots, w_{\ell}) &\mapsto \norm{v - z}^2 - w_{\mu(v)}^2, & (z, w_0, \ldots, w_{\ell}) &\mapsto \max_{0 \leq i \leq \ell} w_i^2.
	\]
	Then $S$, which is defined by its centre $y$ and radii $(r_m)_{m \in \gamma}$, is given by the solution to the following optimization problem:
	\begin{equation}\label{eqn: empty stacks optimization problem}
		\begin{array}{rrcll}
			\textnormal{minimise}   & f(z, w_0, \ldots, w_{\ell})                                                                     \\
			\textnormal{subject to} & g_v(z, w_0, \ldots, w_{\ell}) & \geq & 0 & \textnormal{for all } v \in E \cap \mu^{-1}(\gamma), \\
			                        & g_v(z, w_0, \ldots, w_{\ell}) & \leq & 0 & \textnormal{for all } v \in \sigma.                  \\
		\end{array}
	\end{equation}
	We can reformulate this optimization problem using affine constraints in order to apply the KKT conditions.
	Notice that
	\begin{equation*}
		\norm{v - z}^2 - w_{\mu(v)}^2 = \norm{v}^2 + \norm{z}^2 - 2\braket{z, v} - w_{\mu(v)}^2.
	\end{equation*}
	Define the change of variables $T: \RR^{d} \times \RR^{\ell + 1} \to \RR^{d} \times \RR^{\ell + 1}$ by
	\begin{equation*}
		T(z, w_0, \ldots, w_{\ell}) = (z, \norm{z}^2 - w_0^2, \ldots, \norm{z}^2 - w_{\ell}^2).
	\end{equation*}
	Under this change of variable, we have $f = \tilde fT$ where $\tilde f: \RR^{d} \times \RR^{\ell + 1} \to \RR$ is defined by
	\begin{equation*}
		\tilde f(z, u_0, \ldots, u_{\ell}) = \max_{0 \leq i \leq \ell} \norm{z}^2 - u_i.
	\end{equation*}
	We also have $g_v = \tilde g_v T$ where $\tilde g_v : \RR^{d} \times \RR^{\ell + 1} \to \RR$ is defined by
	\begin{equation*}
		\tilde g_v(z, u_0, \ldots, u_{\ell}) =  \norm{v}^2 + u_{\mu(v)} -2\braket{z, v}.
	\end{equation*}
	Now let $E_\gamma \defeq E \cap \mu^{-1}(\gamma)$ and consider the following optimization problem
	\begin{equation}\label{eqn: empty stacks optimization problem new form}
		\begin{array}{rrcl}
			\textnormal{minimise}   & \tilde f(z, u_0, \ldots, u_{\ell})                                                                      \\
			\textnormal{subject to} & \tilde g_v(z, u_0, \ldots, u_{\ell}) & \geq 0 & \textnormal{for all } v \in E_{\gamma}\setminus \sigma, \\
			                        & \tilde g_v(z, u_0, \ldots, u_{\ell}) & = 0    & \textnormal{for all } v \in \sigma\cap E_{\gamma},      \\
			                        & \tilde g_v(z, u_0, \ldots, u_{\ell}) & \leq 0 & \textnormal{for all } v \in \sigma\setminus E_{\gamma}. \\
		\end{array}
	\end{equation}
	Every solution $(z, w_0, \ldots, w_{\ell})$ to (\ref{eqn: empty stacks optimization problem}) yields a solution $T(z, w_0, \ldots, w_{\ell}) = (z, u_0, \ldots, u_{\ell})$ to (\ref{eqn: empty stacks optimization problem new form}).
	Conversely, every solution $(z, u_0, \ldots, u_{\ell})$ to (\ref{eqn: empty stacks optimization problem new form}) yields a set of solutions $T^{-1}(z, u_0, \ldots, u_{\ell})$ to (\ref{eqn: empty stacks optimization problem}), of which there is a unique solution that satisfies $w_i \geq 0$ for all $0 \leq i \leq \ell$.
	By assumption $S = S(\sigma,\mu; E)$ so $(y, r_0, \ldots, r_{\ell})$ is a solution to (\ref{eqn: empty stacks optimization problem}).
	Therefore, $T(y, r_0, \ldots, r_{\ell})\eqdef (y, t_0, \ldots, t_{\ell})$ is a solution to (\ref{eqn: empty stacks optimization problem new form}).
	Applying the \fullref{thm: KKT conditions}, there must exist coefficients $\lambda_v$ for $v \in E_{\gamma} \cup \sigma$ such that
	\begin{enumerate}
		\item (Primal feasibility) The constraints in (\ref{eqn: empty stacks optimization problem new form}) hold.
		\item (Stationarity) $0 \in \sum_{v \in E_{\gamma} \cup \sigma} \lambda_v \nabla \tilde g_v(y, t_0, \ldots, t_{\ell}) + \partial \tilde f(y, t_0, \ldots, t_{\ell})$, where $\partial\tilde f$ is the sub-differential of $\tilde f$.
		\item (Dual feasibility) $\lambda_v \leq 0$ for all $v \in E_{\gamma} \setminus \sigma$ and $\lambda_v \geq 0$ for all $v \in \sigma \setminus E_{\gamma}$.
		\item (Complementary slackness) $\lambda_v \tilde g_v(y, t_0, \dots, t_l) = 0$ for all $v \in E_{\gamma} \cup \sigma$.
	\end{enumerate}

	The complementary slackness condition states that $\lambda_v = 0$ if the constraint for $v$ is not satisfied strictly; that is, $v \notin \On(S)$.
	Therefore, we may as well index the coefficients $\lambda_v$ by the index set $\On(S)$, setting $\lambda_v = 0$ for $v \in (E_{\gamma} \cup \sigma) \setminus \On(S)$.
	Furthermore, note that $E \cap \On(S) \subset E_{\gamma}$, since by definition $S$ contains spheres indexed by colours in $\gamma$.
	Therefore, the dual feasibility condition above implies conditions (2) and (3) of the proposition.

	Next we come to the stationarity condition. We have
	\begin{equation*}
		\partial \tilde g_v(y, t_0, \ldots, t_{\ell}) = (-2v, 0, \ldots, 0, 1, 0, \ldots, 0)^T \in \RR^{d} \times \RR^{\ell + 1}
	\end{equation*}
	where $1$ is in the $(d + \mu(v) + 1)$\textsuperscript{th} position. For ease of notation, we identify $\RR^d$ with the first $d$ coordinates of $\RR^{d+ \ell + 1}$, so we can write
	\[
		\partial \tilde g_v(y, t_0,\ldots, t_{\ell}) = -2v + e_{d + \mu(v) + 1},
	\]
	where $e_j$ is the $j$\textsuperscript{th} standard basis vector.
	On the other hand, to compute the sub-differential of $\tilde f$ we use \Cref{thm: subdifferential of max of convex functions} to get
	\begin{equation*}
		\partial \tilde f(y, t_0, \ldots, t_{\ell}) = \Conv\{(2y - e_{d+i+1})^T \mid i \in \gamma'\}.
	\end{equation*}
	The stationarity condition can then be restated as follows: there exist coefficients $\mu_i \geq 0$ for $i \in \gamma'$ such that $\sum_{i \in \gamma'} \mu_i = 1$ and
	\begin{equation}\label{eqn: stationarity of KKT}
		\sum_{v \in \On(S)} \lambda_v (2v - e_{d+ \mu(v) + 1}) = \sum_{i \in \gamma'}\mu_i (2y - e_{d + i +1}).
	\end{equation}
	By considering the first $d$ coordinates in the above equation, we get condition (1) of the proposition.
	By considering the $(d + i + 1)$\textsuperscript{th} coordinate entries for each $i \notin \gamma'$, we see that $\sum_{v \in \On(S_i)} \lambda_v = 0$, which is condition (4) of the proposition.
	By considering the $(d + i + 1)$\textsuperscript{th} coordinate entries for each $i \in \gamma'$ we see that $\sum_{v \in \On(S_i)} \lambda_v = \mu_i \geq 0$, which is condition (5) of the proposition.
	Finally, we see $\sum_{i \in \gamma'} \sum_{v \in \On(S_i)} \lambda_v = \sum_{i \in \gamma'} \mu_i = 1$, which proves
	condition (6) of the proposition.

	Next we consider the backwards direction of the proposition.
	Therefore, suppose we are given coefficients $\{\lambda_v\}_{v \in \On(S)}$ that satisfy conditions (1)--(6) of the proposition.
	Set $\lambda_v = 0$ for $v \in (E_{\gamma} \cup \sigma) \setminus \On(S)$.
	We will check that the set of coefficients thus obtained satisfies the KKT conditions stated above for the optimization problem in (\ref{eqn: empty stacks optimization problem new form}).
	First note that by assumption $\sigma \subset \Incl(S)$ and $E \subset \Excl(S)$, so the primary feasibility condition is satisfied.
	By setting $\lambda_v = 0$ for $v \in (E_{\gamma} \cup \sigma) \setminus \On(S)$ we force the complementary slackness condition to hold.
	Conditions (2) and (3) of the proposition imply that the dual feasibility condition holds.
	For $i \in \gamma'$ let $\mu_i = \sum_{v \in \On(S_i)} \lambda_v$.
	Then conditions (5) and (6) of the proposition imply that each $\mu_i$ is non-negative and $\sum_{i \in \gamma'} \mu_i = 1$.
	In conjunction with conditions (1) and (4) of the proposition, we see that the following equation holds:
	\begin{equation}\label{eqn: stationarity of KKT rewritten}
		\sum_{v \in \On(S)} \lambda_v (v + e_{d + \mu(v) + 1}) = y + \sum_{i \in \gamma'} \mu_i e_{d + i + 1} = \sum_{i \in \gamma'} \mu_i(y + e_{d+i+1}).
	\end{equation}
	After applying an appropriate invertible linear transformation to both sides of this equation, we are left with \Cref{eqn: stationarity of KKT} which shows that the stationarity condition holds.
	The upshot is that $S$ defines a solution to the optimization problem in (\ref{eqn: empty stacks optimization problem new form}), and by our previous arguments it also defines a solution to the original problem in (\ref{eqn: empty stacks optimization problem}).

	It remains to show uniqueness of the coefficients $\lambda_v$ that satisfy the conditions of our proposition.
	Without loss of generality assume that $\gamma' = \{k, \ldots, \ell\}$ for some $k \geq 0$.
	Consider the linear map $L : \RR^{d} \times \RR^{\ell + 1} \to \RR^{d} \times \RR^{k + 1}$ defined on basis vectors by
	\[
		L(e_j) = \begin{cases}
			e_j           & j \leq d+k+1, \\
			e_{d + k + 1} & \text{else}.
		\end{cases}
	\]
	This linear map is identity on the first $d+k$ coordinates and sums up the remaining coordinates into the $(d+k+1)$\textsuperscript{th} coordinate.
	Applying $L$ to \Cref{eqn: stationarity of KKT rewritten}, we get
	\begin{equation}\label{eqn: stationarity in KKT after reduction}
		\sum_{i = 0}^{k-1} \sum_{v \in \On(S_i)} \lambda_v(v + e_{d + \mu(v) + 1}) + \sum_{i = k}^{\ell} \sum_{v \in \On(S_i)} \lambda_v(v + e_{d+k+ 1}) = y + e_{d+k + 1}.
	\end{equation}
	Consider the colouring $\nu$ of $\sigma$ obtained from $\mu$ by merging the colours $\{k, \ldots, \ell\}$.
	The vectors appearing on the LHS of \Cref{eqn: stationarity in KKT after reduction} are the points in the chromatic lift of $\On(S)$ with respect to $\nu$, up to rescaling and translation.
	Since any subset of $X$ is also in general position (\Cref{rem: lifts and subsets of general position}), $\On(S)$ is in general position.
	Then by \Cref{prop: GP2} we must have that $|\On(S)| \leq d+k+1$, and by \Cref{thm: dimension of chromatic lift} $\On(S)$ is affinely independent.
	Since the sum of all coefficients appearing on the LHS of \Cref{eqn: stationarity in KKT after reduction} is one, we see that the RHS is in the affine span of an affine independent set, and therefore the corresponding affine coefficients must be unique.
\end{proof}
Although the conditions of \Cref{thm: KKT for stacks} are technical, they have a more straightforward combinatorial formulation.
Suppose $\sigma, E \subset X$ are such that $S = S(\sigma,\mu; E)$ exists.
Then, denoting the coefficients from \Cref{thm: KKT for stacks} by $\{\lambda_v\}_{v \in \On(S)}$, define
\begin{align}
	\Front(S) & \defeq \{v \in \On(S) \mid \lambda_v > 0 \}, \\
	\Back(S)  & \defeq \{v \in \On(S) \mid \lambda_v < 0 \}.
\end{align}
These definitions are borrowed from~\cite{bauer_morse_2016}.

\begin{proposition}[Combinatorial KKT Conditions for Stacks]\label{thm: combinatorial KKT conditions for stacks}
	Let $E, \sigma \subset X$ and let $S=(S_m)_{m \in\gamma}$ be a stack in $\RR^d$ such that $\sigma \subset \Incl(S)$ and $E \subset \Excl(S)$.
	Then $S = S(\sigma,\mu; E)$ if and only if $\mu(\sigma) = \gamma$ and the following are true:
	\begin{enumerate}
		\item $S$ is the stack (not necessarily empty) of minimum radius passing through $\On(S)$.
		\item $\Front(S) \subset \sigma$.
		\item $\Back(S) \subset E$.
	\end{enumerate}
\end{proposition}
\begin{proof}
	The stack of minimum radius passing through $\On(S)$ is nothing but $S(\On(S),\mu; \On(S))$.
	If $S = S(\On(S),\mu; \On(S))$ or $S=S(\sigma,\mu; E)$, then $\Front(S)$ and $\Back(S)$ are well-defined, and the coefficients thereby obtained from \Cref{thm: KKT for stacks} satisfy conditions (1) and (4)--(6) of that proposition, which do not depend on $E$ and $\sigma$.
	Then conditions (2) and (3) of \Cref{thm: KKT for stacks} are equivalent to $\Front(S) \subset \sigma$ and $\Back(S) \subset E$.
\end{proof}

We now examine the level sets of the filtration function of $\Alpha_{\bullet}(X,\mu; E)$, which we denote by $\Rad_{\mu, E}$.
Suppose $E \subset X$, $\gamma, \zeta \subset [s+1]$ and $S=(S_m)_{m \in \gamma}$ and $T= (T_n)_{n \in \zeta}$ are stacks that are empty of $E$.
We say that $T$ \textit{extends} $S$ and write $S \preceq T$ if $\gamma \subset \zeta$, $\Rad(T) = \Rad(S)$, and $S_m = T_m$ for each $m \in \gamma$.
In other words, $S \preceq T$ if and only if $T$ is obtained from $S$ by adding spheres, indexed by colours other than those of $S$, without changing the outer radius of the stack.
This defines a partial ordering on stacks that are empty of $E$, and we define the $S^*$ to be the maximal stack with respect to this partial ordering that extends $S$ and is empty of $E$.
We claim that the level sets of $\Rad_{\mu, E}$ have the form $[\Front(S), \Incl(S^*)]$ as $S$ ranges over the stacks $S(\sigma,\mu; E)$ for $\sigma \in \Alpha_{\infty}(X,\mu; E)$.

\begin{lemma}\label{thm: selective chromatic Delaunay filtration is Morse lemma 1}
	Let $S = S(\sigma,\mu; E)$ for some $\sigma, E \subset X$.
	If $\tau \in [\Front(S), \Incl(S^*)]$ then
	\begin{enumerate}
		\item $\Rad_{\mu, E}(\tau) = \Rad_{\mu, E}(\sigma)$.
		\item $S(\tau,\mu; E) \preceq S^*$.
		\item $\Front(S(\tau,\mu; E)) = \Front(S)$ and $\Incl(S(\tau,\mu; E)^*) = \Incl(S^*)$.
	\end{enumerate}
\end{lemma}
\begin{proof}
	Let $\gamma = \mu(\sigma)$, let $S = (S_m)_{m \in \gamma}$, and define $J \defeq \{j \in \gamma \mid \Front(S) \cap \On(S_j) \neq \emptyset\}$.
	Let $S_J$ be the stack whose set of spheres is $\{S_{j} \mid j \in J\}$.
	Then $S_J$ passes through $\Front(S)$ and excludes $E$.
	Note that $J$ contains at least one of the colours in $\Out(S)$ because of conditions (5) and (6) of \Cref{thm: KKT for stacks}, so $\Rad(S_J) = \Rad(S)$ and $S_J \preceq S$.
	Let $S'$ be the stack obtained by restricting $S^*$ to the spheres that include vertices of $\tau$, so $S' \preceq S^*$.
	Since $\tau \in [\Front(S), \Incl(S^*)]$, by definition of $J$ we must have $S_J \preceq S' \preceq S^*$.
	In particular $\Rad(S_J) = \Rad(S') = \Rad(S^*)$.
	We will show that $S' = S(\tau,\mu; E)$, $\Front(S') = \Front(S)$, and $\Back(S') = \Back(S)$.

	Let $\{\lambda_v \mid v \in \On(S)\}$ be the set of coefficients obtained from \Cref{thm: KKT for stacks} for $\sigma$ and $S$.
	We claim that $\lambda_v = 0$ for each $v \in \On(S) \setminus \On(S_J)$.
	To see this, let $j \in \gamma$ and let $S_j$ pass through some point in $\Back(S)$.
	Then conditions (4) and (5) of \Cref{thm: KKT for stacks} imply that $S_j$ must also pass through some point of $\Front(S)$.
	Now define coefficients $\{\lambda'_v \mid v \in \On(S')\}$ by setting $\lambda'_v = \lambda_v$ if $v \in \On(S_J)$, and $0$ otherwise.
	The fact that $\lambda_v=0$ for $v\not\in \On(S_J)$ implies that the coefficients $\{\lambda_v'\}_{v \in \On(S')}$ satisfy the \fullref{thm: KKT for stacks} for $S'$ and $\On(S')$, which shows that $S'$ is the stack of minimum radius passing through $\On(S')$.
	By construction of $\{\lambda_v' \mid v \in \On(S')\}$, we have $\Front(S') = \Front(S)$ and $\Back(S') = \Back(S)$.
	Then $S' = S(\tau,\mu; E)$ by the \fullref{thm: combinatorial KKT conditions for stacks}.
\end{proof}

\begin{lemma}\label{thm: selective chromatic Delaunay filtration is Morse lemma 2}
	Let $\tau \leq \sigma$ be simplices in $\Alpha_{\infty}(X,\mu; E)$.
	If $\Rad_{\mu,E}(\tau) = \Rad_{\mu,E}(\sigma)$ then $\tau$ contains $\Front(S(\sigma,\mu; E))$.
\end{lemma}
\begin{proof}
	Let $T$ be the stack obtained from $S(\sigma,\mu; E)$ by restricting to the spheres that include vertices of $\tau$.
	Then $\Rad(T) \leq \Rad_{\mu,E}(\sigma) = \Rad_{\mu,E}(\tau) \leq \Rad(T)$, so $\Rad(T) = \Rad_{\mu,E}(\tau)$.
	Then $T = S(\tau,\mu; E)$ by uniqueness of the latter.

	Let $\{\lambda_v\}_{v \in \On(S')}$ and $\{\lambda'_v\}_{v \in \On(S')}$ be the coefficients that appear in \Cref{thm: KKT for stacks} for the minimizing stacks $S = S^{\mu}(\sigma, E)$ and $S'=S^{\mu}(\tau, E)$ respectively.
	Setting $\lambda'_v = 0$ for all $v \in \On(S) \setminus \On(S')$, we see that the set of coefficients $\{\lambda'_v\}_{v \in \On(S)}$ satisfies the requirements of \Cref{thm: KKT for stacks} for $\sigma$ and $S$, so $\lambda_v = \lambda'_v$ for all $v \in \On(S)$ by the uniqueness assertion in that proposition.
	Then condition (2) in the \fullref{thm: combinatorial KKT conditions for stacks} implies that $\Front(S) \subset \tau$.
\end{proof}

We arrive at the main result for this section, which is a generalization of Theorem 4.6 in~\cite{Montesano2025chromatic}.
\begin{theorem}[Selective Chromatic Delaunay Radius is Morse]\label{thm: selective chromatic Delaunay filtration function is Morse}
	Suppose $X \subset \RR^d$ is in general position, $\mu$ is a colouring of $X$, and $E \subset X$.
	Then the filtration function $\Rad_{\mu,E}: \Alpha_{\infty}(X,\mu; E) \to [0, \infty)$ is a generalized discrete Morse function, with gradient given by
	\begin{equation}
		\Sigma = \{[\Front(S), \Incl(S^*)] \mid \sigma \in \Alpha_{\infty}(X,\mu; E), S = S(\sigma,\mu; E)\}.
	\end{equation}
	Moreover, $\sigma \in \Alpha_{\infty}(X,\mu; E)$ is a critical simplex of $\Rad_{\mu,E}$ if and only if $\sigma \in \Del(X)$ and the centre of the empty circumsphere of minimum radius for $\sigma$ in $\RR^d$ is contained in the relative interior of $|\sigma|$ in $\RR^d$.
\end{theorem}
\begin{proof}
	The first part of the theorem is a direct corollary of \Cref{thm: selective chromatic Delaunay filtration is Morse lemma 1,thm: selective chromatic Delaunay filtration is Morse lemma 2}.
	For the second part, let $\sigma$ be a critical simplex of $\Rad_{\mu,E}$ and let $S = S(\sigma,\mu; E)$.
	Then by the definition of a critical simplex we have
	\begin{equation*}
		\Front(S) = \Incl(S^*) = \sigma,\ \Back(S) = \emptyset
	\end{equation*}
	From this we infer that $\Incl(S^*)\setminus \On(S)$ is empty, i.e., $S = S^*$ is an empty stack.
	We also conclude that $\sum_{v \in \On(S_m)} \lambda_v > 0$ for all $m \in \mu(\sigma)$.
	Using condition (4) of the \fullref{thm: KKT for stacks}, this implies $\mu(\On(S)) = \Out(S)$, i.e., all spheres of $S$ have the same radius.
	This means that $S$ is an empty sphere passing through the points of $\sigma$, i.e., $\sigma \in \Del(X)$.
	Moreover, the conditions (1), (5), and (6) of the \fullref{thm: KKT for stacks} along with the fact that $\Front(S) = \sigma$ together imply that the centre of $S$ lies in the relative interior of the convex hull of the vertices of $\sigma$.
	It only remains to check that $S$ is the minimum empty circumsphere passing through $\sigma$, and this follows from the fact that every empty circumsphere for $\sigma$ is also a stack that includes $\sigma$ and excludes $E$.
\end{proof}

From the above theorem we recover the fact the filtration function of the \v{C}ech and chromatic alpha filtrations are generalized discrete Morse functions.

\begin{corollary}
	Suppose $X \subset \RR^d$ is in general position and $\mu$ is a colouring of $X$.
	Then the filtration functions of $\Cech_{\bullet}(X)$ and $\Alpha_{\bullet}(X,\mu)$ are generalized discrete Morse functions.
\end{corollary}
\begin{proof}
	This follows from \Cref{thm: selective chromatic Delaunay filtration function is Morse} and the fact that $\Cech_{\bullet}(X) = \Alpha_{\bullet}(X,\mu; \emptyset)$ and $\Alpha_{\bullet}(X,\mu) = \Alpha_{\bullet}(X,\mu; X)$ (see \Cref{rem: special case of selective Delaunay filtrations}).
\end{proof}

\subsection{Pairing Lemmas}\label{sec: pairing lemmas}

The goal of this subsection is to prove three key technical lemmas, which are needed for the proofs of \Cref{thm: theorem A,thm: theorem B}, that relate the Morse gradients of the chromatic alpha and \v{C}ech filtrations.
These lemmas are direct generalizations of Lemmas 5.4, 5.5, and 5.7 in~\cite{bauer_morse_2016}, and the proofs are very similar.
As in that paper, we will write $\sigma + x$ for $\sigma \cup \{x\}$ and $\sigma - x$ for $\sigma \setminus \{x\}$, and we note that at least one of $\sigma \pm x$ is equal to $\sigma$.
We continue to let $X \subset \RR^d$ be a fixed point cloud in general position and $\mu$ be a fixed colouring of $X$.

\begin{lemma}[Same Stacks Lemma]\label{thm: same stacks lemma}
	Let $E \subset X$, $\sigma \in \Alpha_{\infty}(X,\mu; E)$, and $S = S(\sigma,\mu; E)$.
	For any $x \in X$:
	\begin{enumerate}
		\item $S(\sigma - x,\mu; E) \preceq S \preceq S(\sigma + x,\mu; E)$ if and only if $x \in \Incl(S^*) \setminus \Front(S)$.
		      Additionally, $S = S(\sigma + x,\mu; E)$ if and only if $\mu(\sigma+x) = \mu(\sigma)$ and $S = S(\sigma - x,\mu; E)$ if and only if $\mu(\sigma - x) = \mu(\sigma)$.
		\item $S = S(\sigma,\mu; E \pm x)$ if and only if $x \in \Excl(S)\setminus \Back(S)$.
	\end{enumerate}
\end{lemma}
\begin{proof}

	\begin{enumerate}
		\item The \fullref{thm: combinatorial KKT conditions for stacks} show that $\Front(S) \subset \sigma \subset \Incl(S)$.
		      Then $x \in \Incl(S^*)\setminus\Front(S)$ is equivalent to $\sigma \pm x \in [\Front(S), \Incl(S^*)]$, and in this case \Cref{thm: selective chromatic Delaunay filtration is Morse lemma 1} implies that $S(\sigma \pm x,\mu; E) \preceq S^*$.
		      The first part of the result then follows by noting that $S, S(\sigma-x,\mu; E)$, and $S(\sigma + x,\mu; E)$ must all be comparable with respect to the partial ordering $\preceq$.
		      If in addition $\mu(\sigma - x) = \mu(\sigma + x)$ then the second part of the claim follows directly from the \fullref{thm: combinatorial KKT conditions for stacks}.
		\item The claim follows directly from the \fullref{thm: combinatorial KKT conditions for stacks}.\qedhere
	\end{enumerate}
\end{proof}

\begin{lemma}[First Simplex Pairing Lemma]\label{thm: first pairing lemma}
	Let $E \subset F \subset X$ and $\sigma \in \Alpha_{\infty}(X,\mu; F)$ such that $\Rad_{\mu,F}(\sigma) > \Rad_{\mu,E}(\sigma)$.
	Then there exists a point $x \in F \setminus E$ such that
	\begin{enumerate}
		\item $S(\sigma - x,\mu; E) \preceq S(\sigma,\mu; E) = S(\sigma + x,\mu; E)$.
		\item $S(\sigma - x,\mu; F) \preceq S(\sigma,\mu; F) = S(\sigma + x,\mu; F)$.
	\end{enumerate}
\end{lemma}
\begin{proof}
	Let $S = S(\sigma,\mu; E)$, $T = S(\sigma,\mu; F)$, and $\gamma = \mu(\sigma)$.
	By the \fullref{thm: combinatorial KKT conditions for stacks}, we have $T = S(\sigma,\mu; \Back(T))$.
	By assumption $\Rad(S) < \Rad(T)$, and since $T$ is the stack of minimum radius that includes $\sigma$ and excludes $\Back(T)$, $S$ cannot exclude all of $\Back(T)$.
	Therefore, there exists some point $x \in \Back(T) \setminus \Excl(S)$.
	Now $\Excl(S) = E \cap \mu^{-1}(\gamma)$ and $\Back(T) \subset F \cap \mu^{-1}(\gamma)$, from which it follows that $x \in F \setminus E$.

	Now $x \notin \Excl(S)$ and $\mu(x) \in \gamma$ together imply that $x \in \Incl(S) \setminus \On(S) \subset \Incl(S) \setminus \Front(S)$.
	Then applying the \fullref{thm: same stacks lemma}, we get that $S(\sigma -x,\mu; E) \preceq S = S(\sigma + x,\mu; E)$ (in particular these stacks exist).
	Next, $x \in \Back(T)$ means that $x \in \Incl(T) \setminus \Front(T)$, so the \fullref{thm: same stacks lemma} implies that $S(\sigma - x,\mu; F) \preceq S(\sigma,\mu; F) = S(\sigma + x,\mu; F)$ as desired.
\end{proof}

\begin{lemma}[Second Simplex Pairing Lemma]\label{thm: second pairing lemma}
	Let $E \subset F \subset X$ and let $\sigma \in \Alpha_{\infty}(X,\mu; E) \setminus \Alpha_{\infty}(X,\mu; F)$.
	Then there exists a point $x\in F\setminus E$ such that
	\begin{enumerate}
		\item $S(\sigma - x,\mu; E) \preceq S(\sigma,\mu; E) = S(\sigma + x,\mu; E)$.
		\item Both $\sigma+x$ and $\sigma-x$ are not in $\Alpha_{\infty}(X,\mu; F)$.
	\end{enumerate}
\end{lemma}
\begin{proof}
	Write $S = S(\sigma,\mu; E)$, and $F_{\sigma} = F \cap \Excl(S)$; then the \fullref{thm: combinatorial KKT conditions for stacks} imply that $S = S(\sigma,\mu; F_{\sigma})$ and hence $\sigma \in \Alpha_{\infty}(X,\mu; F_\sigma)$.
	Let $A \subset X$ and $x \in X$ satisfy $F_{\sigma} \subset A \subset A + x \subset F$ such that $\sigma \in \Alpha_{\infty}(X,\mu; A)\setminus \Alpha_{\infty}(X,\mu; A+x)$.
	Writing $\gamma = \mu(\sigma)$, if $\mu(x) \notin \gamma$ then $S(\sigma,\mu; A)$ automatically excludes $A + x$ and $\sigma \in \Alpha_{\infty}(X,\mu; A+x)$, so it must be that $\mu(x) \in \gamma$.
	Together with the fact that $x \in F \setminus \Excl(S)$, this means that $x \in F\setminus E$.
	It also means that $x \in \Incl(S) \setminus \On(S) \subset \Incl(S) \setminus \Front(S)$.
	Applying the \fullref{thm: same stacks lemma} yields part (1) of the claim.
	In fact, the same argument also gives that $x \notin \Excl(S(\sigma,\mu; A))$ and $S(\sigma-x,\mu; A) \preceq S(\sigma,\mu; A) = S(\sigma + x,\mu; A)$.

	Now $\sigma \notin \Alpha_{\infty}(X,\mu; A+x)$, so certainly $\sigma + x \notin \Alpha_{\infty}(X,\mu; A + x)$, and since $\Alpha_{\infty}(X,\mu; F) \subset \Alpha_{\infty}(X,\mu; A+x)$ we see that $\sigma + x \notin \Alpha_{\infty}(X,\mu; F)$ as claimed in (2).
	It remains to show that $\sigma -x \notin \Alpha_{\infty}(X,\mu; F)$, so suppose for contradiction that $\sigma -x \in \Alpha_{\infty}(X,\mu; F)$.
	Then $\sigma - x \in \Alpha_{\infty}(X,\mu; A+x)$.
	Write $T = S(\sigma - x,\mu; A+x)$.
	Since $\sigma +x \notin \Alpha_{\infty}(X,\mu; A+x)$, we must have that $x\notin \Incl(T^*)$.
	Together with the fact that $\mu(x) \in \gamma$, this implies that $x \in \Excl(T) \setminus \On(T) \subset \Excl(T)\setminus\Back(T)$, so applying part (2) of the \fullref{thm: same stacks lemma} gives $T = S(\sigma - x,\mu; A)$.
	We have already shown that $S(\sigma - x,\mu; A) \preceq S(\sigma,\mu; A)$.
	The upshot is that $x \notin \Incl(T^*) = \Incl(S(\sigma-x,\mu; A)^*) \supset \Incl(S(\sigma,\mu; A))$.
	This, along with the fact that $\mu(x) \in \gamma$, contradicts the fact above that $x \notin \Excl(S(\sigma,\mu; A))$.
\end{proof}

\section{Proof of Main Results}\label{sec: proof of main theorems}

Throughout this section we fix a point cloud $X \subset \RR^d$ in general position and the parameter $r \in [0, \infty]$.

\subsection{Proof of \texorpdfstring{\Cref{thm: theorem A}}{Theorem \ref{thm: theorem A}}}

We recall the statement of \Cref{thm: theorem A} for convenience.

\setcounter{maintheorem}{0}
\begin{maintheorem}
	Let $X \subset \RR^d$ be a finite set of points in general position, and let $\mu, \nu$ be colourings of $X$ such that $\nu$ refines $\mu$.
	For any $r \geq 0$ there is a simplicial collapse $\DelCech_r(X,\nu) \searrow \DelCech_r(X,\mu)$.
\end{maintheorem}

Suppose $\mu: X \to \{0, \ldots, s\}, \nu: X \to \{0, \ldots, s'\}$ are colourings of $X$ such that $\nu \preceq \mu$.
We wish to show that there is a simplicial collapse $\DelCech_r(X,\nu)\searrow \DelCech_r(X,\mu)$.
Since $X$ is in general position, we can identify $\DelCech_r(X,\nu)$ with its realization as a subcomplex of the chromatic Delaunay triangulation $\Del(X,\nu)$.
Using the same arguments as in the beginning of \Cref{sec: construction of vertical gradient}, we may assume that $\nu$ is obtained from $\mu$ by splitting up the last class, i.e., $s' = s+1$ and $\mu^{-1}(s) = \nu^{-1}(\{s, s+1\})$.
Then by \Cref{rem: description of membrane} there is a colouring $\zeta: X^{\mu} \to \{0,1\}$ such that $X^{\nu} = (X^{\mu})^{\zeta}$, and there is an embedding of chromatic Delaunay triangulations $\iota: \Del(X,\mu) \to \Del(X,\nu)$ such that $\iota(\Del(X,\mu))$ is the graph of some function $g: \Del(X,\mu) \subset \RR^{d+s} \to \RR$.
In this section, we identify $\DelCech_r(X,\mu)$ with its realization as a subcomplex of $\iota(\Del(X,\mu))$, and we will say that $\sigma \in \Del(X,\nu)$ lies above (resp. below) $\iota(\Del(X,\mu))$ if intersects the interior of the epigraph (resp. hypograph) of $g$ (see \Cref{sec: construction of vertical gradient} for the precise definition).

Recall from \Cref{sec: construction of vertical collapse Morse function} that there is a generalized discrete Morse function
\begin{equation*}
	f : \Del(X,\nu) \cong \DelCech_{\infty}(X,\nu) \to [0, \infty)
\end{equation*}
defined in \Cref{eq: definition of vertical collapse Morse function},
whose critical set is $\iota(\Del(X,\mu)) \cong \DelCech_{\infty}(X,\mu)$.
Let $V^f$ be the generalized discrete Morse gradient associated to $f$,
let $W^{\Cech}$ be the generalized discrete Morse gradient on the complete simplicial complex $\Delta^{|X| - 1}$ associated to the \v{C}ech filtration function $\Rad_{\Cech}$.
Let $W^{\Cech}_r$ be the restriction of this gradient to the subcomplex $\Cech_r(X)$; then $W^{\Cech}_r$ is also a generalized discrete Morse gradient by the \fullref{thm: generalized gradient restriction lemma}.
By the \fullref{thm: sum refinement lemma}, $f + \Rad_{\Cech}$ is a generalized discrete Morse function on $\DelCech_r(X,\nu)$ with gradient
\begin{equation}
	\Lambda \defeq \{I \cap J \mid I \in W^{\Cech}_r, J \in V^f, I \cap J \neq \emptyset\}.
	\label{eq: definition of gradient for DelCech collapse}
\end{equation}
We will show that this generalized gradient induces the simplicial collapse $\DelCech_r(X,\nu) \searrow \DelCech_r(X,\mu)$.
We begin with a technical lemma.

\begin{lemma}\label{thm: characterisation of upper faces in chromatic Delaunay triangulation}
	Let $\sigma$ be a subset of $X$ such that its chromatic lift $\sigma^{\nu}$ spans a $(d+s+1)$-simplex above (resp. below) $\iota(\Del(X,\mu))$ in $\Del(X,\nu)$.
	Let $\tau$ be any proper subset of $\sigma$, and let $\tau^{\nu}$ and $\tau^{\mu}$ be the chromatic lift of $\tau$ with respect to $\nu$ and $\mu$ respectively.
	Then $\tau^{\nu}$ is an upper (resp. lower) face of $\sigma^{\nu}$ if and only if $\tau^{\mu} \notin \Del(\sigma,\mu)$, where $\Del(\sigma,\mu)$ is the chromatic Delaunay triangulation of the subset $\sigma$ with respect to the colouring $\mu$ restricted to $\sigma$.
\end{lemma}
\begin{proof}
	Let $h:\RR^{d+s+1} \to \RR$ be the projection to the last coordinate and let $C$ be the circumcentre of $\sigma^{\nu}$ in $\RR^{d+s+1}$.
	We prove the lemma for the case when $\sigma^{\nu}$ lies above $\iota(\Del(X,\mu))$ in $\Del(X,\nu)$.
	In this case, letting $C(\sigma)$ denote the circumsphere of $\sigma^{\nu}$ in $\RR^{d+s+1}$, we must have $h(C(\sigma^{\nu})) > \frac{1}{2}$ by the forward implication in \Cref{thm: height of circumcentres of simplices above the membrane}.
	Now notice that because $\sigma^{\nu}$ is $(d+s+1)$-dimensional, by definition of the chromatic lift (\Cref{eq: definition of chromatic lift}) we must have that $\nu(\sigma) = \{0, \ldots, s\}$, i.e., $\sigma$ contains at least one point of each colour.
	Then \Cref{thm: dimension of chromatic lift} implies that $\Del(\sigma,\nu)$ is $(d+s+1)$-dimensional, so by a counting argument $\sigma^{\nu}$ belongs to (and is the unique maximal simplex in) $\Del(\sigma,\nu)$.
	The reverse implication in \Cref{thm: height of circumcentres of simplices above the membrane} then shows that $\sigma^{\nu}$ must lie above $\iota(\Del(\sigma,\mu))$ in $\Del(\sigma,\nu)$.

	By \Cref{thm: bichrom_triangulation_collapse} the collapse $\Del(\sigma,\nu) \searrow \iota(\Del(\sigma,\mu))$ removes the simplices in the interval $[(\sigma^{\nu})^*, \sigma^{\nu}]$ because $\sigma^{\nu}$ lies above $\iota(\Del(\sigma,\mu))$.
	Therefore, $\tau^{\nu}$ is an upper face of $\sigma^{\nu}$ if and only if
	\begin{equation*}
		\tau^{\nu}\in [(\sigma^{\nu})^*, \sigma^{\nu}] \iff \tau^{\nu} \notin \iota(\Del(\sigma,\mu)) \iff \tau^{\mu} \notin \Del(\sigma,\mu).\qedhere
	\end{equation*}
\end{proof}

\begin{lemma}\label{thm: intersection lemma}
	Let $\sigma$ be a subset of $X$ such that its chromatic lift $\sigma^{\nu}$ spans a $(d+s+1)$-simplex above (resp. below) $\iota(\Del(X,\mu))$ in $\Del(X,\nu)$, and let $\tau$ be a subset of $\sigma$ such that $\tau = \sigma$ or the chromatic lift $\tau^{\nu}$ is an upper (resp. lower) face of $\sigma^{\nu}$.
	Then there exists a point $x \in \sigma$ such that:
	\begin{enumerate}
		\item $\Rad_{\Cech}(\tau - x) = \Rad_{\Cech}(\tau+x)$.
		\item The chromatic lifts $(\tau + x)^{\nu}$ and $(\tau - x)^{\nu}$ are both upper (resp. lower) faces of $\sigma^{\nu}$.
	\end{enumerate}
\end{lemma}
\begin{proof}
	Suppose $\sigma^{\nu}$ lies above $\iota(\Del(X,\mu))$ and $\tau \subset \sigma$.
	If $\tau^{\nu}$ is an upper face of $\sigma^{\nu}$ then $\tau \notin \Del(\sigma,\mu)$ by \Cref{thm: characterisation of upper faces in chromatic Delaunay triangulation}.
	If $\tau = \sigma$ then $\tau^{\mu} \notin \Del(\sigma,\mu)$ because $|\tau| = d + s + 2$ and $\Del(\sigma,\mu)$ is a $(d+s)$-dimensional simplicial complex by \Cref{thm: dimension of chromatic lift}.
	Now recall the definition of the filtrations defined in \Cref{eq: definition of selective chromatic Delaunay filtration} in \Cref{sec: convex optimization for stacks}.
	With that definition in mind, we observe that $\Del(\sigma,\mu) = \Alpha_{\infty}(\sigma,\mu; \sigma)$.
	We also have $\sigma^{\nu} \cong \Delta^{|\sigma|-1} \cong \Alpha_{\infty}(\sigma,\mu; \emptyset)$, where the first isomorphism is due to the same arguments as in \Cref{thm: characterisation of upper faces in chromatic Delaunay triangulation} above.
	Then we can use the \fullref{thm: second pairing lemma} with $E = \emptyset$ and $F = \sigma$ to obtain a point $x \in \sigma$ such that
	\begin{enumerate}
		\item There is an extension of stacks $S(\tau - x,\mu; \emptyset) \preceq S(\tau + x,\mu; \emptyset)$.
		\item $\tau -x $ and $\tau+x$ do not belong to $\Del(\sigma,\mu)$.
	\end{enumerate}
	The first part of the result follows from noting that the stacks $S(\tau-x,\mu; \emptyset)$ and $S(\tau + x,\mu; \emptyset)$ define minimum bounding balls for $\tau-x$ and $\tau+x$ respectively.
	The second part of the result follows from \Cref{thm: characterisation of upper faces in chromatic Delaunay triangulation}.
\end{proof}

\begin{corollary}
	If $\Lambda$ is the gradient defined in \Cref{eq: definition of gradient for DelCech collapse} then $\Crit(\Lambda) = \DelCech_r(X,\mu)$.
\end{corollary}
\begin{proof}
	By \Cref{thm: description of bichromatic gradient} we have that $\Crit(V^f) = \DelCech_{\infty}(X,\mu)$, so it follows that $\Crit(\Lambda) \supset \DelCech_{r}(X,\mu)$.
	For the reverse inclusion, let $\tau^{\nu}$ be any simplex in $\DelCech_r(X,\nu)\setminus \DelCech_r(X,\mu)$.
	Since both $\DelCech_{\bullet}(X,\nu)$ and $\DelCech_{\bullet}(X,\mu)$ have the same filtration function, $\tau^{\nu}$ must belong to $\DelCech_{\infty}(X,\nu) \setminus \DelCech_{\infty}(X,\mu) = \Del(X,\nu) \setminus \Del(X,\mu)$.
	Without loss of generality let $\tau^{\nu} \in [(\sigma^{\nu})^*, \sigma^{\nu}]$ for some $(d+s+1)$-dimensional $\sigma^{\nu} \in \Del(X,\nu)$.
	Then by \Cref{thm: intersection lemma} we obtain a point $x \in \sigma$ such that $\Rad_{\Cech}(\tau-x) = \Rad_{\Cech}(\tau+x)$, and $(\tau + x)^{\nu}$ and $(\tau-x)^{\nu}$ both lie in the interval $[(\sigma^{\nu})^*, \sigma^{\nu}]$.
	This shows that the interval containing $\tau^\nu$ in $\Lambda$ is not critical; consequently $\Crit(\Lambda) \subset \DelCech_r(X,\mu)$.
\end{proof}

\Cref{thm: theorem A} follows immediately from the above corollary by applying the \fullref{thm: generalized gradient collapsing theorem}.

\begin{remark}\label{rem: counterexample to chromatic alpha collapse}
	Our method of proof for \Cref{thm: theorem A} cannot be modified to work for chromatic alpha filtrations.
	In our proof, we show that there exists a single generalized discrete Morse gradient on $\DelCech_{\infty}(X,\nu)$, whose restriction to $\DelCech_r(X,\nu)$ induces the collapse $\DelCech_r(X,\nu) \searrow \DelCech_r(X,\mu)$ for all values of $r$ simultaneously.
	Recalling the proof of the \fullref{thm: generalized gradient collapsing theorem}, we can choose a sequence of elementary simplicial collapses from $\DelCech_{\infty}(X,\nu)$ to $\DelCech_{\infty}(X,\mu)$ that restricts to an appropriate collapsing sequence $\DelCech_r(X,\nu) \searrow \DelCech_r(X,\mu)$ for each value of $r$.
	This is a strictly stronger property than being able to choose a sequence of elementary collapses for each value of $r$.
	Indeed, in the case of the chromatic alpha filtration, there does not necessarily exist a sequence of elementary collapses from $\Alpha_{\infty}(X,\nu)$ to $\Alpha_{\infty}(X,\mu)$ that restricts to a collapsing sequence from $\Alpha_r(X,\nu)$ to $\Alpha_r(X,\mu)$ for $r < \infty$.
	See \Cref{fig: counterexample for chromatic alpha collapse} for an explicit counter-example.
	Therefore, if there is a simplicial collapse $\Alpha_r(X,\nu) \searrow \Alpha_r(X,\mu)$ for each $r \geq 0$, then the discrete Morse function that induces these collapses will in general depend on $r$.
\end{remark}

\begin{figure}[h]
	\centering
	\begin{subfigure}[t]{0.35\linewidth}
	\centering
	\begin{tikzpicture}[scale=0.6]
		\coordinate (d) at (1, 4);
		\coordinate (b) at (1.5, -1);
		\coordinate (a) at (3, 1);
		\coordinate (c) at (2.25, 3.25);
		\draw (d) to (c);
		\draw [GShape, GLine] (a) -- (b) -- (c) -- cycle;
		\node[Point] at (d) {} node[Label, above       = 5pt of d] {$d$};
		\node[Point] at (a) {} node[Label, right       = 5pt of a] {$a$};
		\node[Point] at (b) {} node[Label, below       = 5pt of b] {$b$};
		\node[Point] at (c) {} node[Label, above right = 5pt of c] {$c$};
	\end{tikzpicture}
	\caption{$t_1 = 6\sqrt{\frac{2}{5}} \approx 3.795$}
\end{subfigure}
\begin{subfigure}[t]{0.35\linewidth}
	\centering
	\begin{tikzpicture}[scale=0.6]
		\coordinate (d) at (1, 4);
		\coordinate (b) at (1.5, -1);
		\coordinate (a) at (3, 1);
		\coordinate (c) at (2.25, 3.25);
		\draw (d) to (c);
		\draw [GShape, GLine] (a) -- (b) -- (c) -- cycle;
		\draw [GShape, GLine] (d) -- (b) -- (c) -- cycle;
		\node[Point] at (d) {} node[Label, above       = 5pt of d] {$d$};
		\node[Point] at (a) {} node[Label, right       = 5pt of a] {$a$};
		\node[Point] at (b) {} node[Label, below       = 5pt of b] {$b$};
		\node[Point] at (c) {} node[Label, above right = 5pt of c] {$c$};
	\end{tikzpicture}
	\caption{$t_2 = \infty$}
\end{subfigure}
\begin{subfigure}[t]{0.25\linewidth}
	\centering
	\begin{tikzpicture}[scale=2]
			\coordinate (x) at (1, 0);
			\coordinate (y) at (0, 1);
			\coordinate (o) at (0, 0);
			\draw[->] (o) -- (x) node[Label, right=4pt of x] {$x$};
			\draw[->] (o) -- (y) node[Label, above=4pt of y] {$y$};
		\end{tikzpicture}
\end{subfigure}

	\vspace{\baselineskip}

	\begin{subfigure}[t]{0.35\linewidth}
	\centering
	\begin{tikzpicture}[
			scale=0.6,
			x={(15:1cm)},
			y={(150:8mm)},
			z={(90:40mm)}
		]
		\coordinate (d) at (1, 4, 1) {};
		\coordinate (b) at (1.5, -1, 0) {};
		\coordinate (a) at (3, 1, 0) {};
		\coordinate (c) at (2.25, 3.25, 0) {};
		\draw [BShape, BLine] (a) -- (b) -- (c) -- cycle;
		\draw [GShape, GLine] (a) -- (d) -- (c) -- cycle;
		\draw [GShape, GLine] (a) -- (b) -- (d) -- cycle;
		\node [OPoint] at (d) {} node[Label, above       = 5pt of d] {$d$};
		\node [BPoint] at (b) {} node[Label, below       = 5pt of b] {$b$};
		\node [BPoint] at (a) {} node[Label, above right = 5pt of a] {$a$};
		\node [BPoint] at (c) {} node[Label, left        = 5pt of c] {$c$};
	\end{tikzpicture}
	\caption{$t_1 = 6\sqrt{\frac{2}{5}} \approx 3.795$}
\end{subfigure}
\begin{subfigure}[t]{0.35\linewidth}
	\centering
	\begin{tikzpicture}[
			scale=0.6,
			x={(15:1cm)},
			y={(150:8mm)},
			z={(90:40mm)}
		]
		\coordinate (d) at (1, 4, 1) {};
		\coordinate (b) at (1.5, -1, 0) {};
		\coordinate (a) at (3, 1, 0) {};
		\coordinate (c) at (2.25, 3.25, 0) {};
		\draw [BShape, BLine] (a) -- (b) -- (c) -- cycle;
		\draw [GShape, GLine] (a) -- (d) -- (c) -- cycle;
		\draw [GShape, GLine] (a) -- (b) -- (d) -- cycle;
		\draw [GShape, GLine] (c) -- (b) -- (d) -- cycle;
		\node [OPoint] at (d) {} node[Label, above       = 5pt of d] {$d$};
		\node [BPoint] at (b) {} node[Label, below       = 5pt of b] {$b$};
		\node [BPoint] at (a) {} node[Label, above right = 5pt of a] {$a$};
		\node [BPoint] at (c) {} node[Label, left        = 5pt of c] {$c$};
	\end{tikzpicture}
	\caption{$t_2 = \infty$}
\end{subfigure}
\begin{subfigure}[t]{0.25\linewidth}
	\centering
	\begin{tikzpicture}[
			scale = 2,
			x={(15:1cm)},
			y={(150:8mm)},
			z={(90:10mm)}
		]
		\coordinate (x) at (1, 0, 0);
		\coordinate (y) at (0, 1, 0);
		\coordinate (z) at (0, 0, 1);
		\coordinate (o) at (0, 0, 0);
		\draw[->] (o) -- (x) node[right] {$x$};
		\draw[->] (o) -- (y) node[left] {$y$};
		\draw[->] (o) -- (z) node[above] {$z$};
	\end{tikzpicture}
\end{subfigure}
	\caption{%
		The alpha filtration ((a)--(b)) and chromatic alpha filtration ((c)--(d)) for $X = \{a, b, c, d\}$ with $a = (3, 1)$, $b=(2, -2)$, $c=(2, 3)$, and $d=(1, 4)$, and the colouring $\mu = (\{a, b, c\}, \{d\})$.
		2D and 3D axes are shown alongside the figures for clarity.
		Note that $\Alpha_{\infty}(X,\mu) \cong \Delta^3$.
		At time $t_1$ there is a unique sequence of elementary collapses from $\Alpha_{t_1}(X,\mu)$ to $\Alpha_{t_1}(X)$, and this sequence starts with the removal of the pair of simplices $\{bd, abd\}$.
		However, there is no sequence of elementary collapses from $\Alpha_{\infty}(X,\mu)$ to $\Alpha_{\infty}(X)$ that includes the removal of that pair, because $bd \in \Alpha_{\infty}(X)$.
	}%
	\label{fig: counterexample for chromatic alpha collapse}
\end{figure}

\subsection{Proof of \texorpdfstring{\Cref{thm: theorem B}}{Theorem \ref{thm: theorem B}}}
We recall the statement of \Cref{thm: theorem B}.
\begin{maintheorem}
	Let $X \subset \RR^d$ be a finite set of points in general position and let $\mu$ be a colouring of $X$.
	Then for any $r \geq 0$ there are simplicial collapses $\Cech_r(X) \searrow \DelCech_r(X,\mu) \searrow \Alpha_r(X,\mu)$.
\end{maintheorem}

Let $\mu$ be fixed.
If we choose $\nu$ to be the maximal colouring of $X$, then $\DelCech_r(X, \nu) \cong \Cech_r(X)$ and by \Cref{thm: theorem A} there is a simplicial collapse $\Cech_r(X) \searrow \DelCech_r(X,\mu)$.
Therefore, to prove \Cref{thm: theorem B} it suffices to show that there is a simplicial collapse $\DelCech_r(X,\mu) \searrow \Alpha_r(X,\mu)$.

Let $W^{\Cech}$ and $W^{\Alpha,\mu}$ be the generalized discrete gradients of the \v{C}ech and chromatic alpha filtration functions $\Rad_{\Cech}$ and $\Rad_{\Alpha,\mu}$ respectively.
By the \fullref{thm: sum refinement lemma} the function $\Rad_{\Cech} + \Rad_{\Alpha,\mu} : \Alpha_{\infty}(X,\mu) \to \RR$ is a generalized discrete Morse function with gradient
\begin{equation}
	\Omega = \{I \cap J \mid I \in W^{\Cech}, J \in W^{\Alpha,\mu}, I \cap J \neq \emptyset\}.
\end{equation}
The set of simplices in $\DelCech_r(X,\mu) \setminus \Alpha_r(X,\mu)$ is partitioned by a subset of intervals $\Omega'$ in $\Omega$, corresponding to intersections $I \cap J$ for intervals $I$ and $J$ in $\Rad_{\Cech}^{-1}(r)$ and $\Rad_{\Alpha,\mu}^{-1}((r, \infty])$ respectively.
For each such simplex $\sigma$, the \fullref{thm: first pairing lemma} applied to $\sigma$ with $E= \emptyset$ and $F=X$ shows that there is a point $x \in X$ such that
\begin{enumerate}
	\item $S(\sigma - x,\mu; \emptyset) \preceq S(\sigma + x,\mu; \emptyset)$.
	      Since each of these stacks is effectively a minimum bounding ball, we have $\Rad_{\Cech}(\sigma-x) = \Rad_{\Cech}(\sigma + x)$.
	\item $S(\sigma -x,\mu; X) \preceq S(\sigma + x,\mu; X)$.
	      In particular these empty stacks exist so $\sigma \pm x \in \Alpha_{\infty}(X,\mu)$, and $\Rad_{\Alpha,\mu}(\sigma -x) = \Rad_{\Alpha,\mu}(\sigma+x)$.
\end{enumerate}
The upshot is that the interval in $\Omega'$ containing $\sigma$ is not a singleton interval.
Then by the \fullref{thm: generalized gradient collapsing theorem} the gradient $\Omega$ induces a collapse $\DelCech_r(X,\mu) \searrow \Alpha_r(X,\mu)$ as desired.
\begin{remark}
	We remark that our proof of the collapse $\DelCech_r(X,\mu) \searrow \Alpha_r(X,\mu)$ follows closely the proof of the main result Theorem 5.9 in~\cite{bauer_morse_2016}.
	However, our proof of the collapse $\Cech_r(X) \searrow \DelCech_r(X,\mu)$ is different, and sidesteps the need for a vertex ordering and a consistent pairing lemma (cf. Lemma 5.8 in~\cite{bauer_morse_2016}).
	Instead, the proof requires choosing a sequence of elementary refinements of $\mu$ that arrive at the maximal colouring.
\end{remark}

\section{Consequences for Practical Computations}\label{sec: consequences for practical computations}

In this section, we highlight the consequences of our main results for computational applications, for both the chromatic Delaunay--\v{C}ech and chromatic Delaunay--Rips filtration.
We also prove that the chromatic Delaunay--Rips filtration is locally stable under perturbations of chromatic sets.
In fact, we prove a more general result, which is that diagrams of chromatic Delaunay--Rips filtrations induced by inclusions of chromatic sets are locally stable to perturbations of the underlying point clouds.
We formulate and prove this statement using the notion of slice categories, and we refer the reader to~\cite{riehl_category_2017} for a reference on the latter.
Finally, we show the computational advantage of the chromatic Delaunay--\v{C}ech and chromatic Delaunay--Rips filtrations through numerical experiments.

\subsection{Chromatic Delaunay--\v{C}ech Filtration}

Given $X$ in general position and $\mu$ a colouring of $X$, \Cref{thm: theorem B} shows that the inclusions
\[
	\mathcal{A}_r(X,\mu) \hookrightarrow \DelCech_r(X,\mu) \hookrightarrow \Cech_r(X)
\]
are homotopy equivalences at every scale $r \geq 0$.
Viewing each filtration as a functor $\cat{ChrSet}_d \times [0, \infty] \to \cat{SimpComp}$, and using the commutativity of \Cref{dgm: functoriality of chromatic DelCech filtration}, these inclusions are components of natural transformations $\Alpha \Rightarrow \DelCech \Rightarrow \Cech$.
After composing with the homology functor, these inclusions are pointwise isomorphisms and hence together constitute a natural isomorphism.

\begin{corollary}
	The \v{C}ech, chromatic Delaunay--\v{C}ech and chromatic alpha filtrations have naturally isomorphic persistent homology.
	More precisely,
	$H_{\ast} \circ \Cech$, $H_{\ast}\circ \DelCech$ and $H_{\ast} \circ \Alpha$ are naturally isomorphic functors $\cat{ChrSet}_d  \times [0, \infty] \to \cat{grVec}$,
	where $H_{\ast}: \cat{SimpComp} \to \cat{grVec}$ is the simplicial homology functor.
\end{corollary}

As mentioned in the introduction, the inclusion maps among the subfiltrations of the chromatic alpha filtration encode information about spatial relationships in the data.
If $(X, \mu)$ is a chromatic set and $I \subset \{0, \ldots, s\}$ is a class of colours, then the inclusion map $\Alpha_{\bullet}(\mu^{-1}(I), \mu) \hookrightarrow \Alpha_{\bullet}(X, \mu)$ encodes the relationship between the points of colours in $I$ and the rest of the points.
This type of inclusion has some drawbacks: it effectively only captures pairwise interactions (the colours in $I$ vs the rest) which do not carry information about higher-order spatial relationships, and it is asymmetric by definition.
To circumvent these limitations, Montesano et al.~\cite{Montesano2025chromatic} define the \textit{$k$-chromatic subfiltration}, which is the subfiltration of $\Alpha_{\bullet}(X, \mu)$ of simplices coloured with at most $k$ colours, and they propose studying the inclusion of the $k$-chromatic subfiltration into $\Alpha_{\bullet}(X,\mu)$.
More generally, let $\Delta^{s}$ denote the abstract simplex whose vertex set is the set of colours.
The colouring $\mu : X \to \{0, \ldots, s\}$ induces a simplicial map $\mu: \Alpha_{\infty}(X, \mu) \to \Delta^{s}$ that sends a simplex $\sigma = \{x_0, \ldots, x_k\}$ to $\{\mu(x_0), \ldots, \mu(x_k)\}$.
For any subcomplex $\Gamma$ of $\Delta^s$, the \textit{$\Gamma$-subfiltration} $\Alpha_{\bullet}(X, \mu, \Gamma)$ is the subfiltration of $\Alpha_{\bullet}(X, \mu)$ given by
\begin{equation}\label{eqn: definition of Gamma subfiltration}
	\Alpha_{\bullet}(X, \mu, \Gamma) \defeq \mu^{-1}(\Gamma) \cap \Alpha_{\bullet}(X, \mu).
\end{equation}
Then the $k$-chromatic subfiltration is a special case of the above construction when $\Gamma$ is the $(k-1)$-skeleton of $\Delta^s$.
In fact the chromatic alpha filtration is also a special case of the above construction: given a subset $I \subset [s]$ of colours, which we can identify with a face $\Delta^I$ of $\Delta^{s}$, the chromatic alpha filtration of the subset $\mu^{-1}(I) \subset X$ is given by
\begin{equation}\label{eqn: chromatic alpha filtration of a subset}
	\Alpha_{\bullet}(\mu^{-1}(I), \mu) = \Alpha_{\bullet}(X, \mu, \Delta^I).
\end{equation}

One can analogously define the $\Gamma$-subfiltrations $\DelCech_{\bullet}(X, \mu, \Gamma)$ and $\Cech_{\bullet}(X, \mu, \Gamma)$ of the chromatic Delaunay--\v{C}ech and \v{C}ech filtrations respectively.
Theorem 3.8 in~\cite{Montesano2025chromatic} shows that the inclusion $\Alpha_{\bullet}(X, \mu, \Gamma) \hookrightarrow \Cech_{\bullet}(X, \mu, \Gamma)$ is a homotopy equivalence.
We will now show that this result also holds for the inclusions $\Alpha_{\bullet}(X, \mu, \Gamma) \hookrightarrow \DelCech_{\bullet}(X, \mu, \Gamma)$ and $\DelCech_{\bullet}(X, \mu, \Gamma) \hookrightarrow \Cech_{\bullet}(X, \mu, \Gamma)$, and our proof also provides an alternative proof of the result in \cite{Montesano2025chromatic}.

\begin{theorem}\label{thm: homotopy equivalence of Gamma subfiltrations}
	Let $X \subset \RR^d$ be in general position, let $\mu:X \to \{0, \ldots, s\}$ be a colouring of $X$, and let $\Gamma$ be any subcomplex of $\Delta^{s}$.
	Then for each $t \in [0, \infty]$ the inclusions $\Alpha_t(X, \mu, \Gamma) \hookrightarrow \DelCech_t(X, \mu, \Gamma) \hookrightarrow \Cech_t(X, \mu, \Gamma)$ are homotopy equivalences.
\end{theorem}
\begin{proof}
	The proof for both inclusions is similar, so we will only prove the result for the first inclusion.
	We proceed by induction on $\dim(\Gamma)$.
	If $\Gamma$ is $0$-dimensional, i.e., $\Gamma = \{\{c_0\}, \ldots, \{c_k\}\}$ where each $c_i$ is a colour in $\{0, \ldots, s\}$, then $\Alpha_t(X, \mu, \Gamma)$ and $\DelCech_t(X, \mu, \Gamma)$ are the disjoint union of monochromatic complexes:
	\[
		\Alpha_t(X, \mu, \Gamma) &= \bigsqcup_{i=0}^k \Alpha_t(\mu^{-1}(c_i), \mu) & \DelCech_t(X, \mu, \Gamma) &= \bigsqcup_{i=0}^k \DelCech_t(\mu^{-1}(c_i), \mu).
	\]
	By \Cref{thm: theorem B}, for each $i \in \{0, \ldots, k\}$ the inclusion $\Alpha_t(\mu^{-1}(c_i) , \mu) \hookrightarrow \DelCech_t(\mu^{-1}(c_i), \mu)$ is a homotopy equivalence, and hence the result holds in this case.

	Next suppose $\Gamma$ is $k$-dimensional.
	Let $\Gamma^{(k-1)}$ denote the $(k-1)$-skeleton of $\Gamma$, let $\sigma_0, \ldots, \sigma_r$ be any ordering of the $k$-simplices of $\Gamma$, and for all $0 \leq j \leq r$ let $\Gamma_j = \Gamma^{(k-1)} \cup \sigma_0 \cup \ldots \cup \sigma_j$.
	We have that $\Sigma_j$ is obtained by gluing $\Sigma_{j-1}$ and $\sigma_j$ along the boundary $\partial \sigma_j$, and correspondingly we have a commutative diagram where every arrow is an inclusion, and whose front and back faces are pushout squares:
	\begin{diagram*}[row sep=2ex, column sep=0.5ex]
		{} & {\Alpha_t(X, \mu, \partial \sigma_j)} && {\Alpha_t(X, \mu, \sigma_j)} \\
		{\DelCech_t(X, \mu, \partial \sigma_j)} && {\DelCech_t(X, \mu, \sigma_j)} \\
		& {\Alpha_t(X, \mu, \Gamma_{j-1})} && {\Alpha_t(X, \mu, \Gamma_j)} \\
		{\DelCech_t(X, \mu, \Gamma_{j-1})} && {\DelCech_t(X, \mu, \Gamma_j)}
		\arrow[from=1-2, to=1-4]
		\arrow[from=1-2, to=2-1, "\iota", "\simeq"']
		\arrow[from=1-2, to=3-2]
		\arrow[from=1-4, to=2-3, "\iota'", "\simeq"']
		\arrow[from=1-4, to=3-4]
		\arrow[from=2-1, to=2-3, crossing over]
		\arrow[from=2-1, to=4-1]
		\arrow[from=3-2, to=3-4]
		\arrow[from=2-3, to=4-3, crossing over]
		\arrow[from=3-2, to=4-1, "\iota''", "\simeq"']
		\arrow[from=3-4, to=4-3, "\kappa"]
		\arrow[from=4-1, to=4-3]
	\end{diagram*}
	The map $\iota$ is a homotopy equivalence by the induction hypothesis, and $\iota''$ is a homotopy equivalence by induction on $j$.
	Now $\Alpha_t(X, \mu, \sigma_j)$ and $\DelCech_t(X, \mu, \sigma_j)$ simply restrictions of the corresponding chromatic complexes to points with colour in $\sigma_j$ (see \Cref{eqn: chromatic alpha filtration of a subset}).
	Therefore the map $\iota'$ is a homotopy equivalence by \Cref{thm: theorem B}.
	It follows that the map $\kappa$ is a homotopy equivalence by the Gluing Theorem for Adjunction Spaces \cite[Theorem 7.5.7]{brown2006topology} and this concludes the proof.
\end{proof}

Thus, the chromatic Delaunay--\v{C}ech filtration can be used as a drop-in alternative to the \v{C}ech or chromatic alpha filtrations, for analysing spatial relations in Euclidean point clouds with persistent homology.

\subsection{Chromatic Delaunay--Rips Filtration}

In general, the \v{C}ech, Vietoris--Rips and Delaunay--Rips filtrations do not share the same persistent homotopy type; see~\cite{mishra_stability_2023} for concrete examples.
Therefore, the chromatic Delaunay--Rips filtration (of which the Delaunay--Rips filtration is a special case) does not in general have the ``correct'' persistent homotopy type.
However, the chromatic Delaunay--Rips filtration does approximate the \v{C}ech filtration, in the following sense.
De Silva and Ghrist~\cite{de_silva_coverage_2007} show that if $\delta = \sqrt{\frac{2d}{d+1}}$ then for any $t \in \RR$ there is an inclusion of complexes
\begin{equation}
	\Cech_t(X) \subset \VR_{t}(X) \subset \Cech_{\delta t}(X)
\end{equation}
from which it follows that
\begin{equation}
	\DelCech_t(X,\mu) \subset \DelVR_t(X,\mu) \subset \DelCech_{\delta t}(X,\mu).
\end{equation}

\Cref{thm: theorem B} shows that $\DelCech_t(X,\mu) \simeq \Cech_t(X)$, so after applying the homology functor we get a commutative diagram in which the horizontal maps are induced by inclusions.
\begin{diagram}\label{dgm: nesting of chromatic DelRips}
	\ldots \ar[r] & H_{\ast}(\Cech_t(X)) \ar[r] \ar[d] & H_{\ast}(\Cech_{\delta t}(X)) \ar[d] \ar[r] & \ldots \\
	\ldots \ar[r] & H_{\ast}(\DelVR_{t}(X,\mu)) \ar[ru] \ar[r] & H_{\ast}(\DelVR_{\delta t}(X,\mu)) \ar[r] & \ldots
\end{diagram}%
The top row of the diagram, as a persistence module, is the usual persistent homology of the point cloud $X$.
The bottom row is the persistent homology calculated using the chromatic Delaunay--Rips filtration.
This relationship is analogous to that between the Vietoris--Rips and \v{C}ech filtrations, which is frequently cited to justify the use of the former to approximate the latter.

Another justification for the use of Vietoris--Rips filtrations is that its persistent homology is a stable descriptor; the persistent homology of the Vietoris--Rips filtration is a Lipschitz map with respect to the Gromov--Hausdorff distance on point clouds and the interleaving distance on persistence modules~\cite{chazal2009gromov}.
Unfortunately, the (chromatic) Delaunay--Rips filtration does not enjoy such a result since perturbations of the underlying point cloud can change the Delaunay triangulation (see~\cite{mishra_stability_2023} for an example).
However, we show that the chromatic Delaunay--Rips filtration, as a mapping $(X, \mu) \mapsto \DelVR_{\bullet}(X,\mu)$, is locally stable to perturbations of the underlying point cloud, and that this local stability extends to diagrams of chromatic Delaunay--Rips filtrations induced by diagrams of inclusions of chromatic sets.

To begin, we relate perturbations and Vietoris--Rips filtration values.
Let $X$ and $Y$ be point clouds.
Given any $f : X \to Y$, the \textit{distortion} of $f$ is given by
\begin{equation}
	\dis(f) \defeq \sup_{x, x' \in X} \big| \norm{x - x'} - \norm{ f(x) - f(x') } \big|.
\end{equation}
\begin{lemma}\label{lem:rips_filtration_values_interleave}
	Let $\Rad_{\VR}$ denote the filtration function of the Vietoris--Rips filtration.
	Let $X, Y\subseteq\RR^d$ and let $f: X \to Y$ be any bijection.
	Then for any $\sigma\subseteq X$,
	\begin{equation*}
		| \Rad_{\VR}(\sigma) - \Rad_{\VR}(f(\sigma)) |  \leq \frac{1}{2}\dis(f).
	\end{equation*}
\end{lemma}
\begin{proof}
	This follows immediately from the definitions since, for any $x, y \in X$,
	\[
		\big| \norm{x - y} - \norm{f(x) - f(y)} \big|\leq \dis(f)
	\]
	and $\Rad_{\VR}(\sigma) = \frac{1}{2}\max_{x, y \in \sigma} \norm{x - y}$.
\end{proof}

In order to state our stability results, we need a metric on chromatic sets.
Suppose we have $d$-dimensional chromatic sets $(X, \mu)$ and $(Y, \nu)$ and a bijection $f : X \to Y$.
We say that $f$ is a \textit{colour-preserving bijection} if $\nu(f(x)) = \mu(x)$.
We define the \textit{chromatic distance} between $X$ and $Y$ to be
\begin{equation}
	d_{C}(X, Y) \defeq \inf \left\{ \sup_{x \in X} \norm{x - f(x)} \mathrel{} \Big| \mathrel{} \text{$f$ is a colour-preserving bijection from $X$ to $Y$}\right\}.
\end{equation}
If there are no colour-preserving bijections from $X$ to $Y$ we define $d_C(X, Y) = \infty$.
This means that, in particular, $d_C(X, Y) = \infty$ whenever $X$ and $Y$ have a different number of points of a given colour, or if their sets of colours are different.
We use $B_{C}(X, r)$ to denote the open ball of radius $r$ around $(X, \mu)$ in the set of $d$-dimensional chromatic sets.

We also need a metric on diagrams of filtrations, for which we use the interleaving distance.
Here we recall the basic definition, and we refer the reader to~\cite{bubenik_categorification_2014} for more details.
If $\mathcal{C}$ is a category then a \textit{persistent $\mathcal{C}$ object} is a functor $F:[0, \infty] \to \mathcal{C}$.
For each $\epsilon\geq 0$, the \textit{$\epsilon$-shift} of $F$ is the persistent $\Ccal$-object $F[\epsilon]$ defined on objects by $F[\epsilon]_t \defeq F_{t + \epsilon}$, and on morphisms in the obvious way.
The mapping $F \to F[\epsilon]$ defines an endofunctor $\Tcal^{\epsilon}$ on the functor category $[[0, \infty], \Ccal]$.
For each $F$ there is a natural transformation $F \to F[\epsilon]$ whose component at $F_t$ is the map $F_t \xrightarrow{F(t \leq t+ \epsilon)} F_{t + \epsilon}$.
These assemble into a natural transformation from the identity endofunctor on $[[0, \infty], \Ccal]$ to $\Tcal_{\epsilon}$.
Given persistent $\mathcal{C}$ objects, $F, G : [0, \infty] \to \mathcal{C}$, an $\epsilon$-\textit{interleaving} between them is a pair of natural transformations $a : F \Rightarrow G[\epsilon]$ and $b : G \Rightarrow F[\epsilon]$ such that
\begin{equation*}
	\Tcal_{\epsilon}(b) \circ a = \mathcal{T}^F_{2\epsilon}
	\quad\text{ and }\quad
	\Tcal_{\epsilon} \circ b = \mathcal{T}^G_{2\epsilon}.
\end{equation*}
Note that an interleaving restricts to an isomorphism between $F_{\infty}$ and $G_{\infty}$.
This induces an extended pseudo-metric on persistent $\mathcal{C}$ objects~\cite{bubenik_categorification_2014} called the \textit{interleaving distance}, given by
\begin{equation}
	d_I(F, G) \defeq \inf \left\{\epsilon \geq 0 \mid \exists \text{ an }\epsilon\text{-interleaving between }F\text{ and }G \right\} \cup \{ \infty \}.
\end{equation}
\begin{lemma}\label{thm: interleaving bound for filtered simplicial complexes}
	Suppose $K_{\bullet}$ and $L_{\bullet}$ are filtered simplicial complexes on the same underlying simplicial complex $M$, with filtration functions $f$ and $g$.
	Denoting $\norm{f - g}_{\infty} \defeq \max_{\sigma \in M} |f(\sigma) - g(\sigma)|$, we have
	\begin{equation}
		d_I(K_{\bullet}, L_{\bullet}) \leq \norm{f-g}_{\infty}.
	\end{equation}
\end{lemma}
\begin{proof}
	Setting $\epsilon = \norm{f - g}_{\infty}$, the triangle inequality yields the expressions $K_a = f^{-1}(-\infty, a] \subset g^{-1}(-\infty, a + \epsilon] = L_{a + \epsilon}$ and $L_b = g^{-1}(-\infty, b] \subset f^{-1}(-\infty, b + \epsilon] = K_{b + \epsilon}$.
	These inclusions comprise a pair of natural transformations that define an $\epsilon$-interleaving between $K_{\bullet}$ and $L_{\bullet}$.
\end{proof}

The diagrams of filtrations which we consider for our stability results are given by applying the chromatic Delaunay--Rips construction to a chromatic set $(X, \mu)$, and considering all possible sub-filtrations that correspond to morphisms in $\cat{ChrSet}_d$.
For example, when $\mu :X \to \{0, 1\}$ is a bi-colouring, then we want to show that the following diagram is stable as a persistent object (here $\DelVR$ without reference to $\mu$ denotes the monochromatic Delaunay--Rips filtration).
\begin{equation*}
	\begin{tikzcd}
		& {\DelVR_{\bullet}(X)} \\
		& {\DelVR_{\bullet}(X,\mu)} \\
		{\DelVR_{\bullet}(X_0)} && {\DelVR_{\bullet}(X_1)}
		\arrow[hook, from=1-2, to=2-2]
		\arrow[hook, from=3-1, to=2-2]
		\arrow[hook', from=3-3, to=2-2]
	\end{tikzcd}
\end{equation*}
We formalize this idea using the slice category $\cat{ChrSet}_d \downarrow (X, \mu)$, which is the category of morphisms into $(X, \mu)$ in $\cat{ChrSet}_d$ (see \Cref{sec: notation} for the definition of slice category).
Let $\Fcal: \cat{ChrSet}_d \downarrow (X, \mu) \to \cat{ChrSet}_d$ be the forgetful functor, whose action on objects is given by
\[
	[(Y, \nu) \hookrightarrow (X, \mu)] \mapsto (Y, \nu).
\]
The diagram we are interested in is the image of the composite $\DelVR \circ \Fcal$.
Notice that the ``shape'' of the category $\cat{ChrSet}_d \downarrow (X, \mu)$ (i.e., the combinatorial structure of the inclusions) depends only on the number of colours $s$.
That is, there is an indexing category $\Pcal_{s+1}$ such that for any $(X, \mu)$ where $\mu$ has label set $[s+1]$, there is an isomorphism of categories $\phi_{(X, \mu)} : \Pcal_{s+1} \to \cat{ChrSet}_d\downarrow (X, \mu)$.
Consider the composite
\begin{equation}
	\Pcal_{s+1} \xrightarrow{\phi_{(X, \mu)}} \cat{ChrSet}_d \downarrow (X, \mu) \xhookrightarrow{\Fcal} \cat{ChrSet}_d \xrightarrow{\DelVR} [[0, \infty], \cat{SimpComp}].
\end{equation}
This is equivalent to the data of a functor $[0, \infty] \to [\Pcal_{s+1}, \cat{SimpComp}]$, i.e., a persistent object valued in diagrams of shape $\mathcal{P}_{s+1}$ in $\cat{SimpComp}$, which we denote by $\DelVR \downarrow (X, \mu)$.

\begin{lemma}\label{thm: colour preserving bijection induces interleaving}
	Let $(X, \mu)$ and $(X', \mu')$ be chromatic sets and suppose there is a colour-preserving bijection $f:X \to X'$ that induces an isomorphism $\Del(X,\mu) \to \Del(X',\mu')$ of simplicial complexes.
	Then $f$ induces a $\left(\frac{1}{2}\dis(f)\right)$-interleaving between
	$\DelVR \downarrow (X, \mu)$ and $\DelVR \downarrow (X', \mu')$.
\end{lemma}
\begin{proof}
	Let $(Y, \nu) \hookrightarrow (X, \mu)$ be a morphism of chromatic sets in $\cat{ChrSet}_d$, and let $Y' = f(Y)$.
	Since $f$ is a colour-preserving bijection, the isomorphism it induces between $\Del(X,\mu)$ and $\Del(X',\mu')$ restricts to an isomorphism between $\Del(Y, \nu)$ and $\Del(Y', \nu \circ f^{-1})$.
	The filtration values are determined by the Vietoris--Rips filtration function so \Cref{lem:rips_filtration_values_interleave,thm: interleaving bound for filtered simplicial complexes} ensure that $f$ induces a $\left(\frac{1}{2}\dis(f)\right)$-interleaving between the filtrations $\DelVR_{\bullet}(X, \mu)$ and $\DelVR_{\bullet}(X', \mu')$, which restricts to an interleaving between $\DelVR_{\bullet}(Y, \nu)$ and $\DelVR_{\bullet}(Y', \nu \circ f^{-1})$.
	Since both interleavings are given by the same pair of maps, they also define an interleaving between $\DelVR_{\bullet}(Y, \nu) \hookrightarrow \DelVR_{\bullet}(X, \mu)$ and $\DelVR_{\bullet}(Y', \nu \circ f^{-1}) \hookrightarrow \DelVR_{\bullet}(X', \mu')$ as persistent objects in the functor category $[(\bullet \to \bullet), \cat{SimpComp}]$.
	Since $(Y, \nu)$ is arbitrary, the interleaving is in fact an interleaving between $\DelVR \downarrow (X, \mu)$ and $\DelVR \downarrow (X', \mu')$ as persistent objects in $[\Pcal_{s+1}, \cat{SimpComp}]$.
\end{proof}

Previous work has shown that if $d_C(X, Y)$ is small enough for monochromatic point clouds, then the Delaunay triangulations of $X$ and $Y$ are isomorphic~\cite{boissonnat2013stability}.
In the following we show that this extends to the chromatic case.

\begin{lemma}\label{thm: existence of colour preserving bijection}
	Let $(X, \mu)\in \cat{ChrSet}_d$ be in general position.
	Then there is an $\epsilon >0$ such that for all $(X', \mu') \in B_C(X, \epsilon)$, there is a colour-preserving bijection $f: X \to X'$
	such that $f$ induces an isomorphism $\Del(X,\mu) \to \Del(X',\mu')$ and $\dis(f) \leq 2 d_{C}(X, X')$.
\end{lemma}
\begin{proof}
	A consequence of the results in~\cite{boissonnat2013stability} is that
	there exists some $\epsilon > 0$ such that
	if $g:X^\mu \to (X')^{\mu'}$ is a bijection and $\sup_{x^{\mu}\in X^\mu} \norm{x^{\mu} - g(x^{\mu})} < \epsilon$,
	then $g$ induces an isomorphism of simplicial complexes
	$\Del(X, \mu) \cong \Del(X', \mu')$.
	Since $d_{C}(X, X')$ is finite and the set of colour-preserving bijections $X \to X'$ is finite, there is a colour-preserving bijection $f: X \to X'$ such that
	$\norm{x - f(x)} \leq d_{C}(X, X')$ for every $x \in X$.
	Now $f$ also induces a bijection between the lifts $f^{\mu}: X^\mu \to (X')^{\mu'}$ and moreover
	\[
		\sup_{x^{\mu}\in X^\mu} \norm{x^{\mu} - f^{\mu}(x^{\mu})} = \sup_{x \in X} \norm{x - f(x)} \leq d_{C}(X, X').
	\]
	Therefore, so long as $d_{C}(X, X') < \epsilon$, the bijection $f$ induces an isomorphism as required.
	Finally, note that a simple application of the triangle inequality shows $\dis(f) \leq 2 d_C(X, X')$.
\end{proof}

Putting together \Cref{thm: colour preserving bijection induces interleaving,thm: existence of colour preserving bijection}, we conclude that the entire persistent $\mathcal{P}_{s+1}$-shaped diagram in $\cat{SimpComp}$, produced by $\DelVR$, is locally stable.

\begin{proposition}[Chromatic Delaunay--Rips diagrams are locally 1-Lipshitz]\label{cor:local_stability_of_delvr}
	For each $(X, \mu)\in \cat{ChrSet}_d$ in general position, there is an $\epsilon > 0$ such that whenever $(X', \mu') \in B_C(X, \epsilon)$ then we have
	\begin{equation}
		d_I(
		\DelVR \downarrow (X, \mu),\DelVR \downarrow (X', \mu')
		)
		\leq
		d_{C}(X, X').
	\end{equation}
\end{proposition}
The main takeaway is as follows: given a chromatic set $(X, \mu)$, $\delta \geq 0$ small enough, and a colour-preserving $\delta$-perturbation $f:(X, \mu) \to (X', \mu')$, there is a $\delta$-interleaving:
\begin{diagram}[column sep=0.8em]
	\DelVR_s(X,\mu) \arrow[rr, hook] \arrow[rd, "f", swap] & & \DelVR_{s+2\delta}(X,\mu) \arrow[rr, hook] \arrow[rd, "f"] & & \Del(X,\mu) \arrow[d, "f", swap, bend right] \\
	& \DelVR_{s+\delta}(X',\mu') \arrow[ru, "f^{-1}", swap] \arrow[rr, hook] & & \DelVR_{s+3\delta}(X',\mu') \arrow[r, hook] & \Del(X', \mu') \arrow[u, bend right, "f^{-1}"']
	\label{eq:delvr_interleaving}
\end{diagram}
Moreover, given any chromatic subset $(Y, \nu) \hookrightarrow (X, \mu)$,
\Cref{cor:local_stability_of_delvr} ensures that the interleaving in \Cref{eq:delvr_interleaving} restricts to an interleaving between $\DelVR_\bullet(Y,\nu)$ and $\DelVR_\bullet(f(Y),\nu\circ f^{-1})$.

This implies that any stable invariant for maps of filtered simplicial complexes can be applied to the chromatic Delaunay--Rips filtration, to obtain a locally stable invariant of chromatic sets.
In particular, thanks to the stability theorem for kernels, images and co-kernels of persistent homology~\cite{cohen-steiner_persistent_2009}, the 6-pack (introduced in~\cite{Montesano2025chromatic}) of the Delaunay--Rips filtration is locally stable.

\subsection{Numerical Experiments}\label{sec: numerical experiments}

\begin{figure}
	\centering
	\resizebox{\textwidth}{!}{\input{figures/benchmarks/all_benches.pgf}}
	\caption{
		Plot of median time taken to compute each filtration (including the underlying triangulation) vs the number of input points.
	}\label{fig:benchmark_log}
\end{figure}

The worst-case computational complexity for each of the chromatic alpha, chromatic Delaunay--\v{C}ech, and chromatic Delaunay--Rips filtrations is linear in the number of simplices of the chromatic Delaunay triangulation, and the constant factor is the worst-case computational complexity to compute the filtration value for an individual simplex.
This constant factor, which depends on the ambient dimension $d$ as well as the number of colours, leads to practical differences for computations with large or high-dimensional datasets where computing the filtration values can be a computational bottleneck.

The algorithm to compute the chromatic alpha filtration value~\cite[Section 3.5 in][]{Montesano2025chromatic} for a simplex consists of two main steps.
The first step is to find the stack of minimum radius passing through the simplex, which can be reduced to the problem of finding a minimum bounding ball for the simplex with the centre of the ball constrained to lie in some affine plane.
The second step is to check if the stack thus obtained is empty, which involves iterating over the points in the co-faces of the simplex.
If the stack is not empty, the simplex is assigned the minimum birth time of its co-faces, which requires the computations to proceed from highest to lowest dimensional simplices.
Computing the \v{C}ech filtration value, on the other hand, involves finding minimum bounding balls for each simplex with no constraints and no additional steps.
The Vietoris--Rips filtration value is faster to compute since it only requires the comparison of edge lengths in the simplex.
In both of our constructions, the filtration values can be computed for low-dimensional simplices without needing the values for higher-dimensional ones.
This is desirable for low-dimensional data which only has persistent homological features of low degree, where only the low-dimensional skeleton of the filtration is required.

To test the computational performance of each filtration, we run numerical experiments using three different sampling schemes:
\begin{description}
	\item[Varying the number of points:] For 10 values of $n$ chosen log-uniformly between 100 and 10000, $n$ points are sampled independently uniformly from $[0, 1]^2\subseteq\RR^2$, and each point is independently assigned a colour from $\{0, 1\}$ uniformly at random.
	\item[Varying the ambient dimension:] For each $d\in \{2, 3, 4, 5, 6\}$, $200$ points are sampled independently uniformly from the unit sphere embedded in $\RR^{d}$, and each point is independently assigned a colour from $\{0, 1\}$ uniformly at random.
	\item[Varying the number of colours:] For each $s \in \{2, 3, 4, 5, 6\}$, $500$ points are sampled independently uniformly from $[0,1]^2\subseteq\RR^2$, and each point is independently assigned a colour from $\{0, \ldots, s-1\}$ uniformly at random.
\end{description}

We record the time taken to compute the chromatic alpha, chromatic Delaunay--\v{C}ech, and chromatic Delaunay--Rips filtrations for these samples.
These times include the time taken to compute the underlying triangulation; we also record the time to compute the triangulation as a baseline.
The process is repeated 5 times for each set of parameters, and we record the median time.
\Cref{fig:benchmark_log} shows the results of our experiments, which were performed on a server with two Intel Xeon Gold 6240M 18-core processors (base frequency 2.6GHz, max 3.9GHz) and 24x128GB DDR4 RAM running at 2933MHz.
The computations were performed using the Python package \texttt{chalc} (v14.0.0) which we developed for our experiments.
Detailed installation instructions, API documentation, and usage examples for \texttt{chalc} are available at \url{https://abhinavnatarajan.github.io/Chalc}, while the source code is publicly available on GitHub at \url{https://github.com/abhinavnatarajan/chalc}.

\begin{remark}
	For the chromatic Delaunay--\v{C}ech and chromatic Delaunay--Rips filtrations, it is possible to compute filtration values in parallel across all simplices of the filtration.
	For the chromatic alpha filtration, we can compute the filtration values in parallel across all simplices of a given dimension, proceeding from high to low dimensions.
	These parallelizations are implemented in \texttt{chalc} using multi-threading, but the resultant speed-up does not scale linearly with the number of allocated CPU cores since our parallel implementations are constrained by memory bandwidth.
	This also makes the gain from parallelization highly hardware-dependent, so the results we show are for the single-threaded implementation only.
\end{remark}

\subsection{Summary}
The practical consequences of our theoretical and experimental results can be summarized as follows:
\begin{itemize}
	\item For low-dimensional data ($d\leq 3$), the chromatic Delaunay--\v{C}ech filtration is slightly faster to compute than the chromatic alpha filtration and has isomorphic persistent homology.
	      Moreover, the two filtrations induce isomorphic maps on persistent homology.
	\item The chromatic Delaunay--Rips filtration is significantly faster than both the chromatic alpha and chromatic Delaunay--\v{C}ech filtrations.
	      The persistent homology of the chromatic Delaunay--Rips filtration is locally stable and approximates the chromatic alpha filtration.
	      For higher-dimensional data or for many colours, the primary constraint for computing the chromatic Delaunay--Rips filtration is computing the underlying chromatic Delaunay triangulation.
	\item The instability of the Delaunay triangulation increases with ambient dimension~\cite{boissonnat2013stability}.
	      Since the chromatic Delaunay triangulation can be viewed as a Delaunay triangulation of points in $\RR^{d+s}$, one should be wary of stability issues with the chromatic Delaunay--Rips filtration when $d+s$ is large.
\end{itemize}

\section{Future Directions}

When $\nu, \mu$ are colourings of the point cloud $X$ such that $\nu$ is a refinement of $\mu$,
our results show that $\Alpha_r(X,\nu)$ and $\Alpha_r(X,\mu)$ have the same simple homotopy type for any $r \geq 0$.
However, whether there exists a collapse $\Alpha_r(X,\nu) \searrow \Alpha_r(X,\mu)$ for each $r \geq 0$ remains an open question.
Even if the answer to the question is affirmative, in general there is no generalized discrete Morse function that induces this collapse for all values of $r$ simultaneously, as explained in \Cref{rem: counterexample to chromatic alpha collapse}.
This motivates the question of defining ``persistent simple homotopy type'', and whether there is a corresponding algebraic notion of a ``persistent Whitehead torsion''.

In practical applications, it is sometimes useful to assign weights to the data points.
There is a notion of weighted Voronoi diagrams and weighted alpha filtrations for weighted datasets; the extension to the chromatic case has not yet been explored.
Another consideration is stability; it is desirable that the local neighbourhood of stability of the chromatic Delaunay triangulation is as large as possible.
One possibility is to extend the idea of ``approximate Delaunay triangulations'', defined in~\cite{de_silva_weak_2008}, to the chromatic setting.
It is also open whether the stability result in \Cref{cor:local_stability_of_delvr} extends to a more general metric on chromatic sets, such as the constrained Gromov--Hausdorff distance for chromatic sets defined in~\cite{draganov_gromov-hausdorff_2025}.

As mentioned in \cref{sec: numerical experiments}, the filtration values for our constructions can be computed for low-dimensional simplices without needing the values for higher-dimensional ones, which is relevant for low-dimensional data where only the low-dimensional skeleton of the filtration is required.
Taking full advantage of this fact requires being able to compute the low-dimensional skeleton of the chromatic Delaunay triangulation; a possible approach is to use the methods recently developed in~\cite{carlsson_computing_2024}.
Leveraging this approach is a subject of future work.

\AtNextBibliography{\small}
\setlength\bibitemsep{1pt}

\section{Acknowledgements}
The authors would like to thank Prof. Heather Harrington, Prof. Ulrike Tillmann, and David Beers for many helpful discussions, and Prof. Vidit Nanda and Luis Scoccola for their feedback on an earlier version of the manuscript.
The authors are grateful to Ond\v{r}ej Draganov for discusssions in relation to the results in \cref{sec: consequences for practical computations}.
A. Natarajan is supported by the Clarendon Fund at the University of Oxford.
A. Natarajan and T. Chaplin are members of the Centre for Topological Data Analysis at the University of Oxford, which is funded by the EPSRC grant `New Approaches to Data Science: Application Driven Topological Data Analysis' \href{https://gow.epsrc.ukri.org/NGBOViewGrant.aspx?GrantRef=EP/R018472/1}{\texttt{EP/R018472/1}}.
For the purpose of Open Access, the authors have applied a CC BY public copyright licence to any Author Accepted Manuscript (AAM) version arising from this submission.
M.J. Jimenez is funded by the Spanish grant Ministerio de Ciencia e Innovacion TED2021-129438B-I00.
The authors are also grateful to the Centre for Topological Data Analysis for hosting M.J. Jimenez, enabling the collaboration that led to this paper. Her visit was funded by Convocatoria de la Universidad de Sevilla para la recualificacion del sistema universitario español, 2021-23, supported by the European Union, NextGenerationEU.

\printbibliography

\newpage

\appendix
\section{General Position Assumptions}\label{sec: general position assumptions}

\subsection{Genericity Results}\label{sec: genericity of general position assumptions}
Here we prove that the general position assumptions we impose through the main manuscript are satisfied by generically chosen point clouds.
We begin by explaining the sense in which we use the term ``generic'', using the machinery of semi-algebraic geometry.
We will state without proof some properties of semi-algebraic sets and semi-algebraic functions, and we refer the interested reader to \cite{dries_tame_1998} for more details.

Let $d, n\geq 1$ be fixed and let $E_n = \RR^{d \times n}$, so that $E_n$ is the space of $d$-dimensional point clouds containing $n$ points.
A \textit{semi-algebraic set} in $E_n$ is a finite union of subsets $A \subset E_n$, such that for each such $A$ there exist polynomials $f, g_1, \ldots, g_p \in \RR[t_1, \ldots, t_{d\times n}]$ with
\begin{equation}
	A = \{x \in E_n \mid f(x) = 0, g_1(x) > 0, \ldots, g_p(x) > 0 \}.
\end{equation}
The collection of semi-algebraic sets is closed under finite unions, finite intersections, complements, coordinate permutations, coordinate projections, and topological closure.
It is a classical result of Lojasiewicz (\cite{lojasiewicz_triangulation_1964}, see also \cite{dries_tame_1998}) that every semi-algebraic set $A$ has a decomposition into a finite union of pairwise disjoint open \textit{cells}, each of which is a homeomorphic copy of $\RR^m$ for some $m$.
The \textit{dimension} of $A$ is defined to be the maximum of the dimension of these cells.
By invariance of domain, this definition does not depend on the choice of decomposition.

We will use the fact that for semi-algebraic sets $A, B$ we have
\begin{enumerate}
	\item $\dim(A) = \dim(\topcl{A})$,
	\item $\dim(A \cup B) = \max(\dim(A), \dim(B))$.
\end{enumerate}
We will also use the fact that the class of semi-algebraic sets contains all algebraic (i.e. Zariski-closed) sets, and that the dimension of an algebraic set as an affine variety coincides with its dimension as a semi-algebraic set.

Finally, a \textit{semi-algebraic function} is a function $f:A \to \RR^m$ whose graph is semi-algebraic.
The domain and image of semi-algebraic functions are semi-algebraic, and the preimage of a semi-algebraic set under a semi-algebraic function is also semi-algebraic.

We say that a property $P$ is \textit{generic} in $E_n$ if the set of points in $E_n$ not satisfying that property is contained in a closed semi-algebraic set $A$ with $\dim(A) < \dim(E_n) = nd$.
In this case, from the cellular decomposition of $A$ we can infer that the set of points in $E_n$ that satisfy $P$ contains an open dense subset of $E_n$, whose complement is a nowhere dense subset of $E_n$ with empty interior and zero Lebesgue measure.

\begin{proposition}\label{thm: algebraic genericity of GP3}
	The set of points in $E_n$ that do not satisfy \Cref{prop: GP3} is an algebraic set of $E_n$ with dimension at most $nd-1$.
	Hence, \Cref{prop: GP3} is generic.
\end{proposition}
It follows that the set of points in $E$ that satisfy \Cref{prop: GP3} is generic, since proper algebraic subsets of $E_n$ are closed.
\begin{proof}
	For $X \in E_n$ and a partition $P =(P_0, \ldots, P_k)$ of $X$ into non-empty disjoint subsets, we say that the partition is \textit{good} if either $|X| > d + k + 1$ or $\dim\left(\sum_{i=0}^k \Lin(P_i)\right) = |X| - (k+1)$.
	We say that a partition is \textit{bad} if it is not good.
	Now $X \in E$ fails to satisfy \Cref{prop: GP3} if there exists an ordered subset $Y\subset X$ and a bad partition $(P_0, \ldots, P_{k})$ of $Y$.
	In such a situation, there is a partition $P'$ of $X$ given by
	\begin{equation*}
		\{P_0, \ldots, P_k\} \cup \bigcup_{x \in X \setminus Y} \{x\}.
	\end{equation*}
	We claim that this is a bad partition of $X$.
	First note that $P'$ has $k + 1 + |X| - |Y|$ elements, and
	\begin{equation*}
		|X| = |Y| + |X| - |Y| \leq d + k + 1 + |X| - |Y|.
	\end{equation*}
	Moreover,
	\begin{equation*}
		\dim\left(\sum_{i=0}^{k+|X|-|Y|} \Lin(P'_i)\right) = \dim \left(\sum_{i=0}^k \Lin(P_i)\right) < |Y| - (k + 1) = |X| - (|X| - |Y| + k + 1),
	\end{equation*}
	so $P'$ is a bad partition of $X$ as claimed.
	Therefore, the set of $X \in E_n$ that do not satisfy \Cref{prop: GP3} is given by
	\begin{equation}
		B_n \defeq \{ X \in E_n \mid \text{there exists a bad partition of $X$} \}.
	\end{equation}

	Now suppose $X \in B_n$ has a bad partition $(P_0, \ldots, P_{k})$ with $|X| < d + k + 1$.
	We claim that there is a bad partition $P' = (P'_0, \ldots, P'_{k-1})$ for $X$.
	Since $P$ is bad, either $P_i$ is an affine dependent set for some $i$ or the sum $\sum_{i = 0}^k \Lin(P_i)$ is not a direct sum.
	In the former case, simply set $P'_i = P_i$ for $i < k-1$ and $P'_{k-1} = P_{k-1} \cup P_k$; then it is obvious that $P'$ is a bad partition.
	On the other hand suppose $k \geq 1$ and $\sum_{i=0}^k \Lin(P_i)$ is not a direct sum.
	Let $j$ be the least integer for which $\Lin(P_j) \cap \sum_{i<j} \Lin(P_i) \neq 0$.
	Then set
	\begin{equation*}
		P'_i = \begin{cases}
			P_i              & i < j, \\
			P_i \cup P_{i+1} & i = j,  \\
			P_{i-1}          & i > j.
		\end{cases}
	\end{equation*}
	It is then clear that $\Lin(P'_j) \cap \sum_{i < j} \Lin(P'_i)\neq 0$ so $P'$ is a bad partition of $X$.

	By induction if $X \in B_n$ then there is a bad partition for $X$ with the minimum possible number of partition elements.
	Setting $k_{n,d} \defeq \max(n-d-1, 0)$, we have
	\begin{equation}
		B_n = \{ X \in E_n \mid \exists \text{ a bad partition } (P_0, \ldots, P_{k}) \text{ of } X \}.
	\end{equation}

	Now we show that this set is algebraic.
	Let $\mathfrak{S}_n$ be the set of ordered partitions of $\{1, \ldots, n\}$ into $k+1$ non-empty disjoint subsets.
	Suppose $\Pcal = (\Pcal_0, \ldots, \Pcal_k) \in \mathfrak{S}_n$ and $\Pcal_i = \{m_1, \ldots, m_{|\Pcal_i|}\}$.
	For $X \in E_n$, let $M_{\Pcal_i}(X)$ be the $d \times (|\Pcal_i| - 1)$ matrix whose columns are given by $x_{m_j} - x_{m_1}$ for $j \in \{2, \ldots, |\Pcal_i|\}$.
	Then the column span of $M_{\Pcal_i}(X)$ is $\Lin(P_i)$, where $P_i$ is the subset of $X$ corresponding to the indices in $\Pcal_i$.
	Let $M_{\Pcal}(X)$ be the $d \times (n - k -1)$ matrix $[M_{\Pcal_0}(X), \ldots, M_{\Pcal_k}(X)]$, i.e., the columnwise concatenation of the $M_{\Pcal_i}(X)$.
	Then $\Pcal$ induces a bad partition of $X$ if and only if $M_{\Pcal}(X)$ has rank at most $n - k - 2$.

	The set of $d \times (n-k-1)$ matrices of rank at most $n-k-2$, which we denote by $\Dcal_{d, n}$, is known to be a $(n-k-2)(d+1)$-dimensional affine variety belonging to a family of affine varieties called \textit{determinantal varieties} (see e.g. Proposition 1.1(b) in \cite{bruns_determinantal_1988}).
	Let $B_{n, \Pcal}\subset B_n$ be the set of $n$-tuples for which $\Pcal$ induces a bad partition; it follows that $B_{n, \Pcal}$ is isomorphic to the product $E_{k+1} \times \Dcal_{d, n}$ which is an affine variety of dimension $d(k+1) + (n-k-2)(d+1) = (n-1)(d+1) - (k+1)$.
	Substituting $k = \max(n -d -1, 0)$ we get
	\begin{equation*}
		\dim(B_{n,\Pcal}) = (n-1)(d+1) - (k+1) \leq nd - 1.
	\end{equation*}

	On the other hand $B_n$ is a finite union of such affine varieties:
	\begin{equation*}
		B_n = \bigcup_{\Pcal \in \mathfrak{S}_n} B_{n, \Pcal}.
	\end{equation*}
	Since each variety in the union has dimension at most $nd-1$, this shows that $B_n$ is an algebraic subset of $E^n$ with dimension at most $nd-1$.
\end{proof}

\begin{proposition}\label{thm: algebraic genericity of GP2}
	\Cref{prop: GP2} is generic.
\end{proposition}
\begin{proof}
	Let $C_n$ be the set of $X \in E_n$ that do not satisfy \Cref{prop: GP2}, and let $B_{n-1}$ be the set of $X \in E_{n-1}$ that do not satisfy \Cref{prop: GP3}.
	Let $C_n' = C_n \setminus (C_{n-1} \times \RR^d)$.
	Then we have
	\[
		C_n = (C_{n-1} \times \RR^d) \cup ((B_{n-1} \times \RR^d) \cap C_n') \cup ((B_{n-1}^c \times \RR^d) \cap C_n').
	\]
	By induction, we may assume that $C_{n-1} \times \RR^d$ is contained in a semi-algebraic set $A$ of dimension at most $nd-1$.
	\Cref{thm: algebraic genericity of GP3} shows that $(B_{n-1} \times \RR^d) \cap C_n'$ is contained in an algebraic set $A'$ of dimension at most $nd-1$.
	Therefore, it suffices to consider $C_n'' \defeq (B_{n-1}^c \times \RR^d) \cap C_n'$.
	We claim
	\begin{equation}\locallabel{thm: main claim}
		\text{$C_n''$ is contained in a semi-algebraic set $A''$ of dimension at most $nd-1$.}\tag{$*$}
	\end{equation}
	Then we have
	\[
		C_n \subset A \cup A' \cup A'' &\subset \topcl{A} \cup \topcl{A'} \cup \topcl{A''}\\
		\text{ and } \dim(\topcl{A} \cup \topcl{A'} \cup \topcl{A''}) &= \max(\dim(\topcl{A}), \dim(\topcl{A'}), \dim(\topcl{A''}))\\
		&= \max(\dim(A), \dim(A'), \dim(A''))\\
		&\leq nd-1.
	\]
	This shows that \Cref{prop: GP2} is generic.

	To prove \localref{thm: main claim}, we begin by simplifying the conditions that need to be checked to determine whether a given $X \in E_n$ satisfies \Cref{prop: GP2}.
	First we claim that we need only consider stacks that pass through all of $X$.
	To see this, suppose $X \in E_n$ violates \Cref{prop: GP2}.
	Then there is a colouring $\mu: X \to [k+1]$, an ordered subset $Y \subset X$ with $|Y| > d+ k + 1$ and $\mu(Y) = [k+1]$, and a $[k+1]$-stack $S$ that passes through $Y$.
	Let $k' = k + |X| - |Y|$, and let $\mu': X \to [k'+1]$ be the colouring of $X$ that agrees with $\mu$ on $Y$ and assigns a distinct new colour to every point of $X \setminus Y$.
	Let $S'$ be the $[k'+1]$-stack with the same centre as $S$, comprised of the spheres from $S$ with an additional sphere for each point of $X \setminus Y$ and passing through that point.
	Then $S'$ is a $[k'+1]$-stack passing through $X$, and
	\begin{equation*}
		|X| = |Y| + |X| - |Y| > d + k + 1 + |X| - |Y| = d + k' + 1.
	\end{equation*}
	Say that a colouring $\mu: X \to [k+1]$ is \textit{bad} if $k < n - d - 1$ and there exists a $[k+1]$-stack passing through $X$.
	Then by the above arguments we have
	\begin{equation*}
		C_n = \{X \in E_n \mid \text{there exists a bad colouring of $X$}\}.
	\end{equation*}
	We note that if $n \leq d + 1$ then $n-d-1 \leq 0$ and there are no bad $[k+1]$-colourings.
	Hence, $C_n = \emptyset$ for $n \leq d + 1$, and we henceforth assume that $n > d + 1$.

	Our second claim is that if $X$ violates \Cref{prop: GP2} for some colouring $\mu: X \to [k+1]$ then we can assume that $k$ is as large as possible.
	To see this, suppose $X \in C_n$ and $\mu$ is a bad $[k+1]$-colouring of $X$ with $k < n - d -2$.
	Let $S$ be a $[k+1]$-stack passing through $X$.
	Then there is at least one sphere in $S$, say $S_i$, that passes through $x, y \in X$ with $\mu(x) = \mu(y)$.
	Define a new colouring $\mu': X \to [k+2]$ that coincides with $\mu$ on $X \setminus \{y\}$ and sends $y$ to a new colour, i.e., $\mu'(y) = k+1$.
	Define the $[k+2]$-stack $S'$ that comprises the spheres of $S$, plus an additional copy of $S_i$ labelled with the colour $k+1$.
	Then $S'$ is a $[k+2]$-stack that passes through $X$ (with respect to the colouring $\mu'$), and $k+1 < n - d -1$.
	We conclude that
	\begin{equation}
		C_n = \{X \in E_n \mid \text{there exists a bad $[n-d-1]$-colouring of $X$}\} \tag{$\ast$}.
	\end{equation}

	Our third claim is that $X$ is in $C_n''$ if and only if $X=(Y, x)$ for some $Y \in B_{n-1}^c$ and $x \in \RR^d$ satisfying the following property: there is an $[n-d-1]$-colouring $\mu$ of $Y$ and an $[n-d-1]$-stack passing through $Y$, such that $x$ lies on this stack.
	One direction is clear: if $(Y, x)$ satisfies the property above then clearly $(Y, x) \in C_n'$.
	On the other hand suppose $X = (Y, x) \in (B_{n-1}^c \times \RR^d) \cap C_n'$.
	$Y$ satisfies \Cref{prop: GP3} by construction, so it is sufficient to check that $(Y, x)$ satisfy property (2) above.
	Let $\mu$ be a bad $[n-d-1]$ colouring of $X$ and $S_{\mu}$ an $[n-d-1]$-stack passing through $x$, which attests to the membership of $X$ in $C_n'$.
	Suppose $x$ lies on a sphere of $S_{\mu}$ with no other points.
	Without loss of generality we can assume that this sphere is labelled with the colour $n-d-2$.
	After discarding this sphere, we are left with a bad $[n-d-2]$-colouring of $Y$ and hence $Y \in C_{n-1}$ using ($\ast$).
	This contradicts the fact that $C_n'$ is disjoint from $C_{n-1} \times \RR^d$.
	We conclude that $x$ must lie on a sphere of $S_\mu$ that passes through some point of $Y$, and hence $\mu$ restricts to an $[n-d-1]$-colouring of $Y$.

	Now we can return to proving (\localref{thm: main claim}).
	Let $\Sfrak$ be the set of surjective functions $[n-1] \to [n-d-1]$, which is finite.
	For arbitrary $Y \in B_{n-1}^c$ and $\mu \in \Sfrak$ there is an induced colouring of $Y$ which we also denote by $\mu$.
	The correspondence between stacks and spheres (see \Cref{sec: chromatic Delaunay triangulation}) shows that if there is an $[n-d-1]$-stack in $\RR^d$ with respect to $\mu$ that passes through $Y$, then there is a corresponding circumsphere for the chromatic lift $Y^{\mu}$ in $\RR^{n-2}$.
	Conversely, by \Cref{thm: dimension of chromatic lift} and using the fact that $Y \in B_{n-1}^c$, the chromatic lift $Y^{\mu}$ is a set of $(n-1)$ affine independent points in $\RR^{n-2}$, so there is a unique circumsphere for $Y^{\mu}$ in $\RR^{n-2}$, and this corresponds to an $[n-d-1]$-stack (with respect to the colouring $\mu$) passing through $Y$.
	We conclude that there are only finitely many $[n-d-1]$-stacks passing through $Y$, each of which corresponds to an $[n-d-1]$-colouring of $Y$.
	Therefore, for $Y \in B_{n-1}^c$ and $\mu \in \Sfrak$ we let $S_{\mu}(Y)$ denote the corresponding $[n-d-1]$-stack passing through $Y$.
	Then we have
	\begin{equation*}
		C_n'' = \bigcup_{\mu \in \Sfrak} \bigcup_{Y \in B_{n-1}^c} \{Y\} \times S_{\mu}(Y).
	\end{equation*}
	Since $\Sfrak$ is finite it suffices to show that $\bigcup_{Y \in B_{n-1}^c} \{Y \} \times S_{\mu}(Y)$ is a semi-algebraic set of dimension at most $nd-1$, for each $\mu \in \Sfrak$.
	Henceforth, let $\mu$ be fixed.

	For any $Y \in B_{n-1}^c$, the centre $c_{\mu}(Y)$ of $S_{\mu}(Y)$ is the unique solution to a linear equation of the form $A_{\mu}(Y) c_{\mu}(Y) = b_{\mu}(Y)$, where $A_{\mu}(Y)$ is a $d \times d$ non-singular matrix, $b_{\mu}(Y)$ is a point in $\RR^d$, and the entries of $A_{\mu}(Y)$ and $b_{\mu}(Y)$ are polynomials in $Y$ that depend on $\mu$.
	Consider the polynomial function
	\[
		B_{n-1}^c \times \RR^d &\to \RR^d,\\
		(Y, x) &\mapsto A_{\mu}(Y)x - b_{\mu}(Y).
	\]
	The preimage of $0$ under this function is a semi-algebraic set, and this pre-image is exactly the set of pairs $(Y, c_{\mu}(Y))$ where $Y \in B_{n-1}^c$ and $c_{\mu}(Y)$ is the centre of the unique $[n-d-1]$-stack passing through $Y$.
	It follows that the function $c_{\mu} : B_{n-1}^c \to \RR^d$ is semi-algebraic, and using this fact it is straightforward to verify that the vector $r_{\mu}(Y)$ of the radii of the spheres in the stack $S_{\mu}(Y)$ is a semi-algebraic function $B_{n-1}^c \to \RR^{n-d-1}$.
	Let $S$ denote a fixed stack of $n-d-1$ concentric spheres of distinct radii centred at the origin in $\RR^d$, which is a semi-algebraic set of dimension $d-1$.
	We then have an injective map $F_{\mu}: B_{n-1}^c \times S \to E_n$ that sends $\{Y\} \times S$ to $\{Y\} \times S_{\mu}(Y)$.
	By the above remarks this map is semi-algebraic, so we have
	\begin{equation*}
		\dim\left( \bigcup_{Y \in B_{n-1}^c} \{Y\} \times S_{\mu}(Y) \right) = \dim(B_{n-1}^c \times S) \leq \dim(E_{n-1}\times S) = nd-1.\qedhere
	\end{equation*}
\end{proof}

We can also formulate the genericity of \Cref{prop: GP2} and \Cref{prop: GP3} in probabilistic terms, as in the following proposition.
We refer the reader to \cite{schneider_stochastic_2008} for technical details about Poisson processes.
\begin{proposition}\label{thm: probabilistic genericity of GP2 and GP3}
	Let $X$ be a Poisson point process in $\RR^d$ $(d \geq 1)$ whose intensity measure $\Theta$ is absolutely continuous with respect to the Lebesgue measure on $\RR^d$.
	Then (the realization of) $X$ satisfies \Cref{prop: GP2} and \Cref{prop: GP3} almost surely.
\end{proposition}
\begin{proof}
	We follow the notation from the proofs of \Cref{thm: algebraic genericity of GP3} and \Cref{thm: algebraic genericity of GP2} so $E = \RR^d$, $B_n$ and $C_n$ are the set of $n$-tuples of points in $E^n$ that violate \Cref{prop: GP3} and \Cref{prop: GP2} respectively.
	Let $E^{n+1}_{\neq} \defeq \{(x_1, \ldots, x_{n+1}) \in E^{n+1} \mid x_i \neq x_j \text{ for } i \neq j\}$ and let $X^{n+1}_{\neq}$ be the random measure induced by $X$ on $E^{n+1}_{\neq}$.
	Then the statement of the proposition is equivalent to the assertion that $X^n_{\neq}(B_n\cup C_n) = 0$ almost surely for all $n\geq 1$.

	Let $\Theta^n$ be the $n$\textsuperscript{th} moment measure of $X$.
	Then by the Slivnyak--Mecke formula \cite[][Corollary 3.2.4]{schneider_stochastic_2008} we have
	\[
		\mathbb{E}X^n_{\neq}(B_n\cup C_n) = \mathbb{E}\left(\sum_{(x_1, \ldots, x_{n}) \in X^n_{\neq}}\mathbf{1}_{B_n \cup C_n}(x_1, \ldots, x_{n})\right) = \int_{B_n\cup C_n} d\Theta^n.
	\]
	By \Cref{thm: algebraic genericity of GP3,thm: algebraic genericity of GP2}, the set $B_n\cup C_n$ has measure zero in $E^n$.
	Then the absolute continuity of $\Theta$ implies that the integral on the right is zero.
	Therefore, $X_{\neq}^n(B_n) = 0$ almost surely.
\end{proof}

\subsection{General Position and Chromatic Genericity}\label{sec: general position vs chromatic genericity}
In Section 4.1 of~\cite{Montesano2025chromatic}, the authors define the notion of `chromatic genericity', and they use this condition to prove that the chromatic alpha filtration function is a generalized discrete Morse function.
Here we clarify the relationship between their definition and our notion of general position.

The authors provide several equivalent definitions of chromatic genericity, and the one we use here is as follows.
For a set of points $P \subset \RR^d$, let $E(P)$ denote the maximal affine subspace of $\RR^d$ of points equidistant to all the points in $P$.
A finite set $X \subset \RR^d$ is said to be \textit{chromatically generic} if whenever $P_0, \ldots, P_k$ are non-empty disjoint subsets of $X$, then $E = E(P_0) \cap \ldots \cap E(P_k)$ is either empty or has codimension equal to $(\sum_{i=0}^k |P_i|) - k - 1$.

\begin{proposition}\label{thm: general position vs chromatic genericity}
	 If $X \subset \RR^d$ is in general position then it is chromatically generic.
\end{proposition}
\begin{proof}
	Let $X$ be in general position, let $P_0, \ldots, P_k$ be disjoint non-empty subsets of $X$, and let $P = \cup_{i=0}^k P_i$.
	Let $\mu = \{X_0, \ldots, X_k\}$ be any colouring of $X$ such that $X_k \cap P = P_k$.
	Let $E = E(P_0) \cap \ldots \cap E(P_k)$ be as in the definition above.
	Notice that $E$ is the affine subspace of $\RR^d$ such that for all $x \in E$ there is a stack centred on $x$ and passing through $P$.
	If $|P| > d + k + 1$ then by \Cref{prop: GP2} such a stack cannot exist, so $E$ must be empty.
	Otherwise, we have that each $E(P_i)$ is a translate of $\Lin(P_i)^{\perp}$, the orthogonal complement to $\Lin(P_i)$.
	Then either $E$ is empty, in which case we are done, or
	\begin{align*}
		\codim(E) &= \codim\left(\bigcap_{i=0}^k E(P_i)\right)\\
		&= \codim\left(\bigcap_{i = 0}^k \Lin(P_i)^{\perp}\right)\\
		&= \codim\left(\left(\sum_{i=0}^k \Lin(P_i)\right)^{\perp}\right)\\
		&= \dim\left(\sum_{i=0}^k \Lin(P_i)\right)\\
		&= |P| - k - 1
	\end{align*}
	where the last equality follows from \Cref{prop: GP3}.
\end{proof}

\begin{remark}
	\Cref{eg: trapezium} is chromatically generic but does not satisfy general position.
	This example and \Cref{thm: general position vs chromatic genericity} show that our notion of general position is strictly stronger than chromatic genericity.
\end{remark}

\end{document}